\documentclass[12pt]{amsart}
\usepackage{amscd,amsthm,amssymb,amsfonts,amsmath,euscript, bbm}
\usepackage{tikz-cd}
\usepackage{mathrsfs,mathtools}
\usepackage{xcolor}
\usepackage[shortlabels]{enumitem}
\usepackage{comment}
 \usepackage{hyperref}
\theoremstyle{plain}
\newtheorem{thm}{Theorem}[section]
\newtheorem{lemma}[thm]{Lemma}
\newtheorem{prop}[thm]{Proposition}
\newtheorem{cor}[thm]{Corollary}

\theoremstyle{definition}
\newtheorem{defn}[thm]{Definition}

\theoremstyle{remark}
\newtheorem{remark}[thm]{Remark}

\newtheorem*{thank}{{\bf Acknowledgments}}
\newcommand{\nc}{\newcommand}

\makeatletter
\newcommand*{\Relbarfill@}{\arrowfill@\Relbar\Relbar\Relbar}
\newcommand*{\xeq}[2][]{\ext@arrow 0055\Relbarfill@{#1}{#2}}
\makeatother

\def\makeop#1{\expandafter\def\csname#1\endcsname
  {\mathop{\rm #1}\nolimits}\ignorespaces}
\makeop{Hom}   \makeop{End}   \makeop{Aut}   \makeop{Isom}
\makeop{Emb}
\makeop{Pic}
\makeop{Gal}   \makeop{ord}   \makeop{Char}  \makeop{Div}   \makeop{Lie}
\makeop{PGL}   \makeop{Corr}  \makeop{PSL}   \makeop{sgn}   \makeop{Spf}
\makeop{Spec}  \makeop{Tr}    \makeop{Nr}    \makeop{Fr}    \makeop{disc}
\makeop{Proj}  \makeop{supp}  \makeop{ker}   \makeop{im}    \makeop{dom}
\makeop{coker} \makeop{Stab}  \makeop{SO}    \makeop{SL}    \makeop{SL}
\makeop{Cl}    \makeop{cond}  \makeop{Br}    \makeop{inv}   \makeop{rank}
\makeop{id}    \makeop{Fil}   \makeop{Frac}  \makeop{GL}    \makeop{SU}
\makeop{Nrd}   \makeop{Sp}    \makeop{Tr}    \makeop{Trd}   \makeop{diag}
\makeop{Res}   \makeop{ind}   \makeop{depth} \makeop{Tr}    \makeop{st}
\makeop{Ad}    \makeop{Int}   \makeop{tr}    \makeop{Sym}   \makeop{can}
\makeop{length}\makeop{SO}    \makeop{torsion} \makeop{GSp} \makeop{Ker}
\makeop{Adm}   \makeop{Mat}
\makeop{Q-isog}
\makeop{Rad}
\makeop{Ind}

\def\makebb#1{\expandafter\def
  \csname bb#1\endcsname{{\mathbb{#1}}}\ignorespaces}
\def\makebf#1{\expandafter\def\csname bf#1\endcsname{{\bf
      #1}}\ignorespaces}
\def\makegr#1{\expandafter\def
  \csname gr#1\endcsname{{\mathfrak{#1}}}\ignorespaces}
\def\makescr#1{\expandafter\def
  \csname scr#1\endcsname{{\EuScript{#1}}}\ignorespaces}
\def\makecal#1{\expandafter\def\csname cal#1\endcsname{{\mathcal
      #1}}\ignorespaces}

\def\doLetters#1{#1A #1B #1C #1D #1E #1F #1G #1H #1I #1J #1K #1L #1M
                 #1N #1O #1P #1Q #1R #1S #1T #1U #1V #1W #1X #1Y #1Z}
\def\doletters#1{#1a #1b #1c #1d #1e #1f #1g #1h #1i #1j #1k #1l #1m
                 #1n #1o #1p #1q #1r #1s #1t #1u #1v #1w #1x #1y #1z}
\doLetters\makebb   \doLetters\makecal  \doLetters\makebf
\doLetters\makescr
\doletters\makebf   \doLetters\makegr   \doletters\makegr

\normalsize

\makeop{Bl}

\def\Fpbar{\overline{\bbF}_p}
\def\Fp{{\bbF}_p}
\def\Fq{{\bbF}_q}

\def\Qp{{\bbQ}_p}

\def\Zp{{\bbZ}_p}
\def\Qbar{\overline{\bbQ}}

\newcommand{\Z}{\mathbb Z}
\newcommand{\Q}{\mathbb Q}

\newcommand{\C}{\mathbb C}

\renewcommand{\O}{\mathcal O}
\newcommand{\F}{\mathbb F}

\newcommand{\whD}{\widehat{\mathcal{D}}}

\newcommand{\whO}{\widehat{O}}

\newcommand{\wbZ}{\widehat{\mathbb{Z}}}
\newcommand{\wbQ}{\widehat{\mathbb{Q}}}

\newcommand{\npr}{\noindent }

\DeclareMathOperator{\Min}{Min}

\newcommand{\<}{\langle}
\renewcommand{\>}{\rangle}

\newcommand{\isoto}{\stackrel{\sim}{\longrightarrow}}
\nc{\embed}{\hookrightarrow}

\newcommand{\dieu}{Dieudonn\'{e} }

\nc{\ol}{\overline}
\nc{\wt}{\widetilde}
\nc{\opp}{\mathrm{opp}}

\def\wh{\widehat}

\makeop{Ram}
\makeop{Rep}

\def\sfF{\mathsf{F}}
\def\sfV{\mathsf{V}}
\def\sp{{\rm sp}}

\usepackage{color}
\usepackage{marvosym}

\begin{document}
\numberwithin{equation}{section}

\title
{Superspecial Points on Shimura Curves}
 \author{Yasuhiro Terakado}
\address{
(Terakado) Faculty of Education,
Teikyo University
\\
359 Otsuka
\\
Hachioji-shi, Tokyo, Japan, 192-0395}
\email{terakado.yasuhiro.nr@teikyo-u.ac.jp}

 \author{Jiangwei Xue}
 \address{(Xue)   School of
  Mathematics and Statistics, Wuhan University \\
   Luojiashan
   \\
  Wuhan, Hubei, P.R. China, 430072}
  \email{xue\_j@whu.edu.cn}

\author{Chia-Fu Yu}
\address{
(Yu) Institute of Mathematics, Academia Sinica \\
Astronomy Mathematics Building \\
No.~1, Sec.~4, Roosevelt Rd. \\
Taipei, Taiwan, 106319}
\email{chiafu@math.sinica.edu.tw}

\date{\today}

\begin{abstract}
Let $X$ be the Shimura curve attached to an indefinite quaternion $\mathbb{Q}$-algebra $B$ with a maximal order $O_B$. 
This paper investigates the reduction $X\otimes \mathbb{F}_p$ of $X$ modulo an arbitrary prime $p$, focusing 
particularly on its superspecial locus.
We give an explicit criterion for
the existence of superspecial $\mathbb{F}_q$-rational points on $X$. Furthermore, we compute both 
the number of geometric superspecial points and the number of $\mathbb{F}_p$-rational superspecial points, through the Eichler class number formula and the Selberg trace formula. 
As a key ingredient, we classify the Dieudonn\'e modules attached to superspecial $O_B$-abelian surfaces, which generalizes Ribet’s classification of admissible quaternion bimodules of rank $2$ by dropping the admissible hypothesis. These results generalize Deuring's explicit formula for supersingular elliptic curves over $\mathbb{F}_p$ and give the Shimura-curve analogue of the Ibukiyama–Katsura formulas for principally polarized superspecial abelian surfaces over $\mathbb{F}_p$.
\end{abstract} 

\maketitle
\tableofcontents
\section{Introduction}\label{sec:I}

Throughout this paper, $p\in \bbN$ denotes a rational prime number, $\Fpbar$ denotes a fixed algebraic closure of the  prime field $\F_p$, and for each power $q=p^a$ of $p$,  $\Fq\subseteq \Fpbar$ denotes the finite subfield with $q$ elements.

An abelian variety $A$ over $\F_q$  is called {\it supersingular} if $A\otimes_{\Fq} \Fpbar$ is isogenous to a product of supersingular elliptic curves, and {\it superspecial} if it is isomorphic to such a product.
The arithmetic of supersingular abelian varieties  is closely related to quaternion algebras.
Indeed, for a  supersingular elliptic curve $E$ over $\Fpbar$,  its endomorphism algebra $\End^0(E)\coloneqq \End(E)\otimes_\Z\Q$ is the unique quaternion $\Q$-algebra $D\coloneqq D_{p, \infty}$ ramified exactly at $p$ and $\infty$, and  $\End(E)$ is a maximal order in $D$ by \cite[Theorem~4.2]{waterhouse:thesis}.
Conversely, every maximal order in $D$ occurs as the endomorphism ring of a supersingular elliptic curve by \cite[Theorem~3.13]{waterhouse:thesis}.
Fix a maximal order $O_D \subseteq D$. The classical Deuring correspondence (\cite{Deuring-1941}\cite[Corollary~42.3.7]{voight-quat-book}) establishes a bijection between the set of isomorphism classes of supersingular elliptic curves over $\Fpbar$ and the ideal class  set $\Cl(O_D)$ of fractional right $O_D$-ideal classes.
Consequently, the number of such isomorphism classes  is  the \emph{class number} $h(D)=\lvert \Cl(O_D) \rvert$, which is independent of the choice of the maximal order $O_D$.
If $t(D)$ denotes the \emph{type number} of $D$, namely the number of isomorphism classes of maximal orders in $D$, then the number of isomorphism classes of supersingular elliptic curves admitting models over $\F_p$ is equal to $2t(D)-h(D)$ (\cite{Deuring-1941},\cite[\S2]{Gross-Height-L-series}).
 Explicit formulas for $h(D)$ and $t(D)$ were obtained by Eichler \cite{eichler-CNF-1938} and Deuring \cite{Deuring1950}.

Deuring's correspondence has been generalized to higher-dimensional supersingular and superspecial abelian varieties.
By the Deligne-Ogus-Shioda theorem \cite[Theorem 6.2]{Ogus-1979}\cite[Theorem~3.5]{Shioda-K3surf-1979}, for every fixed $n>1$  there is a unique isomorphism class of superspecial abelian varieties over  $\Fpbar$ of dimension $n$.
In the polarized setting,
the work of Ibukiyama, Katsura, and Oort \cite{Ibukiyama-Katsura-1994, Ibukiyama-Katsura-Oort-1986} further relates  principally polarized superspecial abelian varieties to quaternion  hermitian lattices.
Together with related work  \cite{Hashimoto-ternary, Hashimoto-Ibukiyama-1,
Ibukiyama:Osaka2018, Ibukiyama:tohoku2019, Ibukiyama-Katsura-1994},
these results lead to explicit class- and type-number formulas and, consequently, to formulas for  superspecial points
and irreducible components of supersingular loci of Siegel modular varieties of low dimensions.\\

Shimura curves are natural analogues of modular
curves.
The reduction of these curves modulo $p$ provides an instructive example for understanding Newton and Ekedahl-Oort stratifications of more general Shimura varieties. 
More precisely, let $B$ be an indefinite quaternion algebra over $\Q$, let $O_B \subseteq B$ be a maximal order, and let
$X=X_{O_B}$ be the coarse moduli scheme over $\Z$  parametrizing
$O_B$-abelian surfaces satisfying the determinant condition (see Section \ref{ss:exist}).
If $B$ splits at $p$, then $X$ has good reduction at $p$, and every supersingular point is superspecial, as for classical modular curves.
If $B$ ramifies at $p$, then $X$ has bad reduction, and the entire special fiber is supersingular.
This raises two natural questions: how superspecial points are classified within the special fiber, and how
many of them are defined over a prescribed finite field.

Our main results answer these questions as follows:
\begin{enumerate}
    \item We give an explicit criterion for the existence of $\F_q$-rational superspecial points on $X$ (Theorem~\ref{thm:Fq-pts}).
   \item We classify
superspecial
$O_B$-abelian surfaces over $\Fpbar$ (Proposition \ref{prop:genus} and Theorem \ref{thm:5-emb}).
    \item  We explicitly determine the numbers of geometric superspecial points and of $\F_p$-rational superspecial points (Theorem \ref{thm:main}).
    In the bad reduction case, we also determine the numbers of irreducible components and singular points of the special fiber $X_{\Fpbar}$.
\end{enumerate}
 The formulas in (3) generalize the classical formulas of Eichler and Deuring from  supersingular elliptic curves to superspecial $O_B$-abelian surfaces.
 They also provide an analogue of the formulas of Hashimoto, Ibukiyama, and Katsura \cite{Hashimoto-Ibukiyama-1,
Ibukiyama:Osaka2018, Ibukiyama-Katsura-1994} for superspecial polarized abelian surfaces. \\

We now outline the main ideas of the proofs. The existence criterion in (1) is obtained
from Honda–Tate theory. The possible supersingular isogeny classes over finite fields are described by supersingular Weil numbers, and the existence of quaternion multiplication by
$B$ is translated into an embedding problem for the corresponding endomorphism algebras.
Together with a suitable integral realization of the $O_B$-action, this yields Theorem \ref{thm:Fq-pts} in Section \ref{ss:exist}.

For (2), we classify
superspecial
$O_B$-abelian surfaces by investigating their \dieu modules endowed with an action of  the quaternion $\Z_p$-order  $O_B \otimes \Z_p$.
When $B$ ramifies at $p$, put $\O_p\coloneqq O_B \otimes \Z_p$; then $\O_p$ is the unique maximal order in the division quaternion algebra $B \otimes _{\Q}\Q_p$.
The classification  of superspecial $\O_p$-\dieu modules reduces to determining the $\GL_2(\O_p)$-conjugacy classes of embeddings
$\O_p \hookrightarrow \Mat_2(\O_p).$
Ribet \cite{Ribet-Bimodules} classified the ``admissible''  embeddings by reducing the problem to the representation theory of a hereditary order (cf.~Remark \ref{rem:adm}).
Without the admissibility assumption, additional integral structures occur. We prove that there are exactly five distinct conjugacy classes, and that they are distinguished by their Lie types and critical indices (Section \ref{ss:conj}).
 These five local classes give rise to all the genera of superspecial $O_B$-abelian surfaces that occur in the bad reduction case.

 For (3),  we first determine the endomorphism rings of the superspecial $O_B$-abelian surfaces  in each genus and give  descriptions of these genera as adelic double coset spaces in Section~\ref{ss:as}. In particular, each genus is identified with the ideal class set of a definite quaternion order. The work of Zink \cite{zink:thesis} and Ribet \cite{Ribet-Bimodules} then identifies, in the bad reduction case, the relevant genera with the superspecial points, the singular points, and the irreducible components of the special fiber. Thus, the geometric counting problem is reduced to the computation of class numbers of  quaternion orders.

We evaluate these class numbers in Section~\ref{ss:count} using Eichler's class number formula and mass formulas together with local optimal embedding numbers. 
This is related to our previous work \cite{terakado-xue-yu:2023} where a relevant more general Shimura variety with a prime-to-$p$ level structure is studied and the number  of irreducible components of the supersingular locus is calculated. 
We also need to calculate the cardinality of a similar class set, and the level structure trivializes the relevant automorphism groups so that the corresponding count is expressed directly by a mass. In the present setting, however, without auxiliary level structures,
nontrivial automorphisms cause the mass formula to differ from the unweighted class number; the theory of optimal embeddings supplies the correction terms needed to pass from the mass to the desired explicit class number formulas.

The count of $\F_p$-rational superspecial points requires a more  involved treatment.  Under the adelic description, Frobenius acts on the relevant double coset space, so $|X^{\rm sp}(\F_p)|$ is expressed as the trace of an Atkin--Lehner type operator $R(\varpi)$ on algebraic modular forms on a definite quaternion algebra. We evaluate this trace directly by applying the Selberg trace formula for compact quotients. The resulting orbital integrals are expressed in terms of optimal embeddings of quadratic orders, leading to the modified Hurwitz class numbers that appear in Theorem~\ref{thm:main}. In the good reduction case this trace is also accessible through the classical Eichler trace formula for Brandt matrices; the Selberg-trace-formula approach is essential for treating the ramified case uniformly. This differs from the classical Eichler--Deuring approach, and from the approach of Ibukiyama and collaborators, in which the corresponding rational-point count is recovered indirectly from class and type numbers.

We finally mention a related result of
Jordan and Livn\'{e} \cite{Jordan-Livne}, who
established a criterion for the existence of rational points on Shimura curves  over local fields using a combinatorial description of the dual graph of the special fiber.
Our criterion in Theorem~\ref{thm:Fq-pts}  concerns superspecial rational points and arises instead from the moduli interpretation, Honda--Tate theory, and quaternionic embedding conditions. Combined with deformation theory, it gives a complementary, more moduli-theoretic approach to the corresponding existence problem over local fields.

The present paper is organized as follows.
Section~\ref{ss:Pre}  reviews \dieu modules with additional structures, and prepares preliminary results on abelian varieties with additional structures. Section~\ref{ss:exist} studies supersingular and superspecial $O_B$-abelian surfaces over finite fields and proves the existence criterion. Section~\ref{ss:conj} classifies superspecial  $\calO_p$-\dieu modules through embeddings $\O_p\hookrightarrow\Mat_2(\O_p)$. Section~\ref{ss:as} determines the corresponding endomorphism rings and gives adelic descriptions of the superspecial genera together with the Galois action. Finally, Section~\ref{ss:count} proves the explicit counting formulas, including the relation $|X^{\rm sp}(\F_p)|=\Tr R(\varpi)$, by applying the Eichler class number formula and the Selberg trace formula for compact quotients.

\ \\
\npr Notation. Let $\Q_\ell$ denote the $\ell$-adic completion of $\Q$ at a prime $\ell$, and let $\Z_\ell$ be its ring of integers.
For a finitely generated $\Z$-module or a finite-dimensional $\Q$-vector space $M$, we set  $M_\ell \coloneqq M\otimes_\Z \Z_\ell$ or $M\otimes_\Q \Q_\ell$, respectively. If $R\to R'$ is a homomorphism between commutative rings, and
$M$ is an $R$-module or an $R$-scheme, then we write $M_{R'}\coloneqq M\otimes_{R} R'$.
For an abelian scheme $A$ over a base scheme $S$, we write $\End(A)=\End_S(A)$ for its endomorphism ring over $S$, and $\End^{0}(A)\coloneqq \End(A)_\Q$ for the endomorphism algebra.
If  $A_1$ and $A_2$ are isogenous abelian varieties over a field $k$ (with the isogeny definable over $k$ as well), then we write  $A_1\sim A_2$.
An abelian scheme $A/S$ of relative dimension two is called an {\it{abelian surface}}.

\section{\dieu modules and supersingular abelian varieties}\label{ss:Pre}
In this section, we review \dieu  modules and abelian varieties equipped with additional structures, and recall some of their basic properties. Throughout this section, $p\in \bbN$ denotes a prime, and $k$ denotes a perfect field of characteristic $p$.

\subsection{\dieu modules}\label{subsec:dieu-modules}
We begin by recalling covariant \dieu module theory \cite[Section 5]{li-oort}.
 Let
 $W(k)$ be the ring of Witt vectors over $k$, equipped with the absolute Frobenius automorphism $\sigma: W(k)\to W(k)$.
 Let $W(k)[{\sf F, V}]$ be the \emph{Dieudonn\'e ring},  that is, the quotient of the free associative  $W(k)$-algebra on two indeterminates ${\sf F, V}$ by the two-sided ideal generated by the relations
 \begin{equation}
     \label{eq:W[FV]}
 {\sf {FV}=VF}=p,
 \quad {\sf F}a=a^{\sigma}{\sf F},
 \quad {\sf V}a^{\sigma}
 =a{\sf V} \quad  {\text{for all}} \quad a \in W(k).
 \end{equation}

 \begin{defn}\label{def:dmod}
 (1) A \emph{Dieudonn\'e module $M$ over $k$} is a left $W(k)[{\mathsf F,\mathsf V}]$-module that is
 finitely generated and free as a $W(k)$-module.
  We write $\End_{\rm DM}(M)$ for the ring of endomorphisms of $M$ over $W(k)[{\sf F, V}]$.
The tensor product $M\otimes_{W(k)} W(k)[1/p]$ is called the  {\it rational \dieu module}  of $M$ (also known in the literature as an $F$-isocrystal).
It is a finite-dimensional $W(k)[1/p]$-vector space equipped  with a bijective $\sigma$-linear operator $\sfF$.

 (2)
 Two \dieu modules  are said to be {\it isogenous} if their associated  rational \dieu modules are isomorphic.

\end{defn}
Let $A$ be an abelian variety over $k$.
Let $M(A)$ and $\Lie(A)$ denote its {\it covariant} \dieu module and its Lie algebra, respectively.
There is a  canonical isomorphism of $k$-vector spaces
\[\Lie(A) \simeq M(A)/\sfV M(A).\]
The \emph{$a$-number} of $A$ is defined as
\[ a(A)=a(M(A))\coloneqq \dim_k M(A)/(\sfF,\sfV) M(A). \]

For a positive integer $n$,
let $\Q_{p^n}$ denote the unique unramified extension of $\Qp$ of degree $n$, and let $\Z_{p^n}=W(\F_{p^n})$ denote its  ring of integers.
Assume that $k \supseteq \F_{p^n}$.
We write $\Sigma_{\Z_{p^n}}$ for the set of $\Z_p$-embeddings of $\Z_{p^n}$ into $W(k)$, that is,
\[ \Sigma_{\Z_{p^n}}\coloneqq \Hom_{\Zp}(\Z_{p^n}, W(k))=\{\tau_i \mid  i\in \Z/n\Z\}, \]
whose elements are indexed so that $\tau_0$ corresponds to the fixed inclusion $\F_{p^n} \subseteq k$ and $\sigma \circ  \tau_i=\tau_{i+1}$ for all $i\in \Z/n\Z$.
There is a natural decomposition
\[ \Z_{p^n}\otimes_{\Zp} W(k)\xrightarrow{\sim} \prod_{i\in \Z/n\Z} W(k), \quad  a\otimes b \mapsto (\tau_i(a) b)_{i\in \Z/n\Z}, \quad \forall\, a\in \Z_{p^n}, b\in W(k).  \]
For any $\Z_{p^n}\otimes_{\Z_p} W(k)$-module $M$,
this induces a  decomposition
\begin{equation}\label{eq:dec}
     M=\bigoplus_{i\in \Z/n\Z} M^i.
     \end{equation}
Here $M^i\coloneqq \{x\in M\mid (a\otimes 1)x=(1\otimes \tau_i(a))x, \,\forall a\in \Z_{p^n}\}$ denotes  the $\tau_i$-component of $M$.

\begin{defn} \label{defn:O-dieu-mod}

(1)
Let $\calO$ be a $\Zp$-algebra that is finite free as a $\Zp$-module.
An {\it $\calO$-\dieu module}  over $k$ is a pair $(M, \varphi)$, where $M$ is a
 \dieu module over $k$, and $\varphi : \calO \to \End_{\rm DM}(M)$ is a monomorphism  of $\Z_p$-algebras.

(2) Let $(M, \varphi)$ be a  $\Z_{p^n}$-\dieu module  over $k\supseteq \F_{p^n}$. The \emph{Lie algebra} of
$M$ is defined by $\Lie(M)\coloneqq M/\sfV M$, which admits an  $(\F_{p^n}, k)$-bimodule structure induced by $\varphi$.   It admits a  decomposition
\begin{equation}\label{eq:Lie-decomp}
 \Lie(M)=\bigoplus_{i\in \Z/n\Z}\Lie(M)^i=\bigoplus_{i\in \Z/n\Z} (M/\sfV M)^i,
\end{equation}
where $\Lie(M)^i\coloneqq \{x\in \Lie(M)\mid \varphi(a)x=a^{p^i} x, \,\forall a\in \F_{p^n}\}$ as usual.
The \emph{Lie type} of $(M,\varphi)$ is defined to be the row vector
$(\dim_k (M/\sfV M)^i)_{i\in \Z/n\Z}$ of $n$ non-negative integers.
\end{defn}
           \begin{thm}[{\cite[Theorem 2 and Remark 3]{Oort-MathAnn-1975}}]\label{thm:sp}
           Let $A$ be an abelian variety of dimension $g$ over $k$.
           Then the following are equivalent:
           \begin{itemize}
               \item[(i)] $A$ is  isomorphic over an algebraic closure $\bar{k}$ to a product of $g$ supersingular elliptic curves;
               \item[(ii)] The \dieu module
           $M \coloneqq M(A)$ satisfies ${\sf V}^2M=pM$;
           \item[(iii)] The $a$-number satisfies $a(A)=g$.
           \end{itemize}
           \end{thm}
               An abelian variety $A$ (resp.~a \dieu module $M$) over $k$  is called \emph{superspecial} if it satisfies the above equivalent conditions (resp. condition (ii)).
\subsection{Abelian varieties with additional structures}\label{ss:av}
         Let $B$ be a semisimple $\Q$-algebra and let $O$ be a $\Z$-order in $B$.
\begin{defn}\label{def:ab}
An \emph{$O$-abelian scheme} over a base scheme $S$ is a pair $(A, \iota)$, where $A$ is an abelian scheme over $S$,  and $\iota  : O\hookrightarrow \End(A)$   is an injective ring homomorphism.
\end{defn}
\begin{lemma}\label{lm:move-to-OB}
    Let $A$ be an abelian variety over a field $k$, and let $\iota^0: B\to \End^0(A)$ be  a homomorphism of $\Q$-algebras.
    Then there exists an $O$-abelian variety $(A_1,\iota_1)$ over $k$ such that $A_1\sim A$ over $k$.
\end{lemma}
\begin{proof}
   Set $O'\coloneqq ({\iota^0})^{-1}(\End (A))\cap O$.
   Then $O'$ is an order contained in $O$ with $\iota^0(O')\subseteq\End (A)$.
   By \cite[Theorem 6.5]{yu:mo}, there  exists a $k$-isogeny $A\to A_1$  such that, under the natural identification $\End^0 (A)\simeq \End^0 (A_1)$, the map $\iota^0$ sends $O$ into $\End (A_1)$.
   In fact, one may take $A_1\coloneqq O\otimes_{O'} A$ by applying the Serre tensor construction \cite[\S6.1]{yu:mo}, and the $k$-isogeny is given by the canonical map $A\to O\otimes_{O'} A$.
\end{proof}
\begin{remark}\label{rem:Waterhouse-constr}
  When the ground field $k$ is finite, the $O$-abelian variety $(A_1, \iota_1)$ can also be constructed using  Waterhouse's method. Indeed, there always exists a maximal order $O_0 \subseteq \End^0(A)$ that contains $\iota^0(O)$.
  In the proof of  \cite[Theorem~3.13]{waterhouse:thesis}, Waterhouse constructs a quotient abelian variety $A_1\coloneqq A/H(I)$ of $A$ by a suitable finite $k$-subgroup scheme $H(I)$ such that $\End (A_1)=O_0$, where  we identify $\End^0(A)$ with $\End^0(A_1)$ via the quotient map $A\to A_1$.
  Then $\iota_1$ is obtained by composing  this identification with $\iota^0$.
\end{remark}
Let $\phi$ be an  endomorphism of $k$.
For an algebraic variety $Y$ over $k$, set ${Y}^{\phi}\coloneqq Y\otimes_{k, \phi}k$,   the $\phi$-twist of $Y$.
  Given a morphism $f: Y\to Z$ of algebraic varieties over $k$, we write $f^{\phi}: {Y}^{\phi}\to {Z}^{\phi}$ for its   $\phi$-twist.

Now let $\phi: x \mapsto x^q$  be the $q$-power Frobenius endomorphism of $k$, and
let  $\pi_{Y}=\pi_{Y/k}: Y \to {Y}^{\phi}$ denote the associated relative $q$-power Frobenius morphism over $k$.
If $Y$ is  defined over $k=\F_q$, then there is a canonical isomorphism ${Y}^{\phi}\simeq Y$.
If $A$ is an abelian variety  over $\F_q$, then $\pi_A : A \to {A}^{\phi}\simeq A$ is a homomorphism of group schemes,  called the {\it{Frobenius endomorphism} of $A$ over $\F_q$}.
\begin{prop}\label{prop:sp_Can}
    Let $\mathrm{Sp}_{\F_{p^2}}(-p)$ denote the category of superspecial abelian varieties $A$ over $\F_{p^2}$ satisfying  $\pi_A=[-p]$.
    Let $k$ be an algebraically closed field of characteristic $p$ with a fixed embedding $\F_{p^2}\hookrightarrow k$, and let $\mathrm{Sp}_k$ be the category of superspecial abelian varieties over $k$.
     Then
    the base change functor
    \[\calF: \mathrm{Sp}_{\F_{p^2}}(-p)\to \mathrm{Sp}_k, \quad A\mapsto A\otimes_{\F_{p^2}} k,\]
    defines an equivalence of categories.
\end{prop}
\begin{proof}
By~\cite[Chap.~IV, Section 4, Theorem 1, p.~93]{MacLane}, it suffices to show that the functor $\calF$ is fully faithful and essentially surjective.
 Since  $[-p]$ commutes with all elements of  $\Hom(A\otimes_{\F_{p^2}} \Fpbar,A'\otimes_{\F_{p^2}} \Fpbar)$, the natural map $\Hom(A, A')\to \Hom(A\otimes_{\F_{p^2}} \Fpbar,A'\otimes_{\F_{p^2}} \Fpbar)$ is an isomorphism by Galois descent. Fix an embedding $\Fpbar\hookrightarrow k$ that induces the given embedding $\F_{p^2}\hookrightarrow k$.
 The base change functor from $\Fpbar$ to $k$ induces an  isomorphism  $\Hom(A\otimes_{\F_{p^2}} \Fpbar,A'\otimes_{\F_{p^2}} \Fpbar)\isoto \Hom(A\otimes_{\F_{p^2}} k, A'\otimes_{\F_{p^2}} k)$ by Chow's Theorem (see~\cite{Chow-1955} and~\cite[Theorem~1.2]{yu:chow}).
 Composing these two canonical isomorphisms shows that $\calF$ is fully faithful.
 Essential surjectivity is well known in the case  $\dim A=1$,  since every supersingular elliptic curve admits a model over $\F_{p^2}$ with Frobenius equal to $[-p]$.
 For $\dim A>1$, it follows from the results of Deligne, Ogus and Shioda; see~\cite[Theorem 6.2]{Ogus-1979}.
\end{proof}
By Proposition~\ref{prop:sp_Can}, any  superspecial abelian variety $A$ over $\Fpbar$ canonically descends to a superspecial abelian  variety $\widetilde{A}$ over  $\F_{p^2}\subseteq \Fpbar$ such that $A\simeq \widetilde{A} \otimes_{\F_{p^2}}\Fpbar$ and  $\pi_{\widetilde{A}}=[-p]$.
\begin{cor}\label{cor:sp_fod}
    Let $B$ be a semisimple $\Q$-algebra with a positive involution $*$,  and let $O$ be  an order stable under $*$.
    Then the base change functor  $\calF$ of  Proposition~\ref{prop:sp_Can} is an equivalence  from the category of superspecial  $O$-abelian varieties over $\F_{p^2}$ satisfying  $\pi_A=[-p]$, to the category of superspecial  $O$-abelian varieties over $k$. The same holds true if  ``superspecial  $O$-abelian varieties" is replaced by ``polarized superspecial  $O$-abelian varieties".
\end{cor}

\subsection{Supersingular isogeny classes over finite fields}\label{subsec:Weil-no}
Let $q=p^a$ be a power of $p$.
\begin{defn}
   (1) An algebraic integer $\pi \in \Qbar$ is called  a {\it{Weil $q$-number}} if  $\lvert \tau(\pi) \rvert=\sqrt{q}$ for every  embedding $\tau : \Q(\pi)\hookrightarrow \C$.
    Two Weil $q$-numbers $\pi_1, \pi_2$ are said to be {\it conjugate} if there exists an isomorphism $\Q(\pi_1) \simeq \Q(\pi_2)$  sending $\pi_1$ to $\pi_2$.

    (2) A Weil $q$-number $\pi$ is called   {\it supersingular} if it is of the form $\pi=\sqrt{q} \zeta$ for some  root of unity $\zeta$.
\end{defn}

For a Weil $q$-number $\pi$,
the field $\Q(\pi)$ is either totally real or CM~\cite[Proposition 4]{honda}.
Let $A$ be an abelian variety over $\F_q$, and denote by  $\pi_A \in \End(A)$  its  Frobenius endomorphism over $\F_q$.
The Honda-Tate theorem asserts that the assignment $A \mapsto \pi_A$ induces a bijection between the set of $\F_q$-isogeny classes of simple abelian varieties over $\F_q$ and  the set of conjugacy classes of Weil $q$-numbers.
A simple abelian variety $A$ over $\F_q$ is supersingular if and only if its Frobenius  $\pi_A$ is supersingular  (cf.~\cite[Theorem 2.9]{Yu-End-QM-2013}).
Moreover, every supersingular Weil  $q$-number $\sqrt{q} \zeta$ is conjugate to either $\sqrt{q} \zeta_n$
or $-\sqrt{q} \zeta_n$ (or both), where  $n\not\equiv 2 \pmod 4$ and  $\zeta_n \coloneqq e^{2\pi i/n}$; see \cite[Corollary~2.4]{xue-yang-yu:sp_as}.

Now  fix a Weil $q$-number $\pi$, and let $A_\pi$ be a simple abelian variety over $\Fq$ in the isogeny class corresponding to  $\pi$.
Set
\[d(\pi)\coloneqq \dim A_{\pi}, \quad \calE \coloneqq \End^0(A_{\pi}), \quad {\text{and}} \quad K \coloneqq \Q(\pi) \subseteq\calE.\]
 Then  $\calE$ is a central division algebra over $K$, and
 \begin{equation}\label{eq:dim}
2d(\pi)=[\calE : K]^{\frac{1}{2}}[K : \Q].
\end{equation}

Let $D_{p, \infty}$ denote  the  quaternion algebra over $\Q$  ramified precisely at $p$ and $\infty$.
If $K$ is an imaginary quadratic field in which $p$ splits, denote by  $\calQ_{K, p}$  the  quaternion $K$-algebra ramified exactly at the two places of $K$ lying above $p$.
By \cite[Propositions 3.1 and 3.5]{xue-yang-yu:sp_as}, the supersingular Weil $q$-numbers $\pi$ with $d(\pi)=1,2$ are classified as follows.

\begin{prop}\label{prop:summ-Weil-num}
    Let $n\not\equiv 2 \pmod 4$, and let $\pi=\pm \sqrt{q}\zeta_n$ be a supersingular Weil $q$-number, where  $q=p^a$.
\begin{itemize}
    \item[(1)]
    Suppose that $a$ is even.
    \begin{itemize}
        \item[(i)] $d(\pi)=1$ if and only if one of the following holds:
        \begin{itemize}
            \item[(a)] $n=1$, in which case $K=\Q$
            and $\calE=D_{p, \infty}$; or
            \item[(b)] $n=3, 4$ and $p\not\equiv 1\pmod n$, in which case $K=\Q(\sqrt{-n})$ and $\calE=K$.
        \end{itemize}
        \item[(ii)] $d(\pi)=2$ if and only if one of the following holds:
        \begin{itemize}
            \item[(a)] $n=3,4$ and $p \equiv 1 \pmod n$, in which case $K=\Q(\sqrt{-n})$ and $\calE=\calQ_{K, p}$; or
            \item[(b)]
            $n=5,8,12$ and $p\not\equiv 1 \pmod n$, in which case  $[K:\Q]=4$ and hence $\calE=K=\Q(\zeta_n)$.
        \end{itemize}
    \end{itemize}
    \item[(2)] Suppose that  $a$ is odd.
    \begin{itemize}
        \item[(i)] $d(\pi)=1$ if and only if $\pi \in
        \{\sqrt{q}\zeta_4$, $\pm \sqrt{q}\zeta_8 (p=2), \pm \sqrt{q}\zeta_{12} (p=3)\}$.
        \item[{(ii)}] $d(\pi)=2$ if and only if
        \[ \pi \in \{\sqrt{q}, \sqrt{q}\zeta_3, \pm\sqrt{q}\zeta_5 \ (p=5),
        \sqrt{q}\zeta_8 \ (p\neq 2), \sqrt{q}\zeta_{12} \  (p\neq 3), \pm \sqrt{q}\zeta_{24} \ (p=2) \}. \]
        With the exception of $\pi=\sqrt{q}$, we have $\calE=K$ in all remaining cases for odd $a$; in the exceptional case, $K=\Q(\sqrt{p})$ and $\calE$ is the unique quaternion $\Q(\sqrt{p})$-algebra $D_{\infty_1, \infty_2}$ ramified exactly at the two infinite places of $\Q(\sqrt{p})$.
        \end{itemize}
\end{itemize}
\end{prop}

We make an observation on the endomorphism ring of a superspecial abelian variety. Waterhouse proves in his
 thesis \cite[Theorem~4.2]{waterhouse:thesis} that if $A$ is a supersingular elliptic curve over $\F_q$ such that $\Q(\pi_A)$ is an  imaginary quadratic field (equivalently, $\pi_A\not\in\Z$), then  $\End(A)$ is an order in $\Q(\pi_A)$ that is maximal at $p$.
This result was  refined in \cite[\S4.2, p.~1620]{xue-yang-yu:sp_as},
where it was  shown that $\End(A)$ is in fact  the maximal order of $\Q(\pi_A)$,  except when $a$ is odd, $\pi_A^2=-q$, and $p\equiv 3\pmod{4}$.
In that case,  $\End(A)$ is isomorphic to either $\Z[\sqrt{-p}]$ or $\Z[(-1+\sqrt{-p})/2]$.
The following lemma extends these results  to higher-dimensional superspecial  abelian varieties.
\begin{lemma}\label{lem:ssp-zeta-stable}
Assume that $q=p^a$ with $a$ even.
Let $A$ be a supersingular abelian variety over $\Fq$ such that its Frobenius endomorphism $\pi_A$ is of the form  $\sqrt{q}\zeta$ for some root of unity $\zeta$.
If $A$ is superspecial, then $\Z[\zeta]\subseteq   \End(A)$. Conversely, if $\Z[\zeta]\subseteq  \End(A)$ and $a=2$, then $A$ is superspecial.
\end{lemma}
\begin{proof}
Suppose first that $A$ is superspecial.
    By Theorem \ref{thm:sp}, the \dieu module $M=M(A)$ satisfies $\sfV^2 M = pM$, and hence $\sfV^a M=p^{a/2} M$.
    On the other hand, as $\pi_A$ acts on the covariant \dieu module $M(A)$ as $\sfV^a$, we have ${\sfV}^a M =\pi_A M=p^{a/2}\zeta M$.
    Comparing these two equalities, we obtain $\zeta M=M$, and hence
    $\zeta\in \End_{\rm DM}(M)\simeq \End(A)_{\Z_p}$.
   Moreover, since $\sqrt{q}=p^{a/2}$ is an  $\ell$-adic unit for every prime $\ell\neq p$, we also have  $\zeta=\pi_A/\sqrt{q}\in \End(A)_{\Z_{\ell}}$.
   Thus, $\Z[\zeta]\subseteq \End(A)$.

    Conversely, assume that $a=2$ and $\Z[\zeta]\subseteq  \End(A)$.
    Then  $\pi_A=p\zeta$ and $\zeta M=M$.
    Hence  $\sfV^2 M=\pi_A M=p\zeta M=pM$, which implies that $A$ is superspecial.
\end{proof}

\begin{remark}
\begin{enumerate}[label=(\arabic*), align=right,   leftmargin=*]
    \item  As shown in the proof, the superspecial condition is crucial here.
Indeed, in Proposition~\ref{prop:ss-but-nonssp} we construct examples of non-superspecial supersingular $O_B$-abelian surfaces $A$ over $\F_{p^2}$ with $\pi_A=\pm p\zeta_n$ for some  $n\in \{3, 4\}$ such that $\zeta_n\not\in \End(A)$.

    \item Assume that $a$ is even,   and set $M\coloneqq W(\Fq)[\zeta, \sfF,\sfV]/(\sfV^{a/2}-\zeta \sfF^{a/2})$ where the relations in  \eqref{eq:W[FV]} hold among  $\sfF$,  $\sfV$, and the elements of $W(\Fq)$.
    Then $M$ is a supersingular $\Zp[\zeta]$-\dieu module of $W(\Fq)$-rank $a [\Zp[\zeta]:\Zp]$, whose $a$-number is given by $a(M)= [\Zp[\zeta]:\Zp]$.
    Therefore,
   if $a\ge 4$ is even, then $a(M) < (\rank_{W(\Fq)} M)/2$, and hence $M$ is not superspecial.
\end{enumerate}
\end{remark}
\section{Existence of superspecial $O_B$-abelian surfaces over $\F_q$}\label{ss:exist}
Throughout  the rest of this paper, we write $B$ for an indefinite quaternion algebra  over $\Q$ with discriminant  $\Delta$, and fix a maximal order  $O_B\subseteq B$.
\subsection{$O_B$-abelian surfaces with the determinant condition}\label{ss:QM}
\begin{defn}\label{defn:QM-absurf}
 An abelian surface $A$ over a base scheme $S$ is said to admit \emph{quaternion multiplication (QM) by $B$} if there exists an injective homomorphism
$\iota^0:B\hookrightarrow \End^0 (A)$ of $\Q$-algebras.
The pair $(A, \iota^0)$ is called a  {\it {QM abelian surface}} over $S$.
\end{defn}

\begin{prop}\label{prop:ssQMFq}
Let $A$ be a supersingular abelian surface over $\F_q$ with $q=p^a$. Then $A$ admits QM by $B$ if and only if one of the following conditions holds depending on the  parity of $a$.
\begin{enumerate}
    \item[(1)] Case $a$ even, which is further divided into two subcases:
    \begin{itemize}
        \item[(i)] $A\sim A_{\pi}^2$, where $\pi=\pm \sqrt{q} \zeta_n$, $d(\pi)=1$, and either
    \begin{itemize}
            \item[(a)] $n=1$, $K=\Q$, or
            \item[(b)] $n \in \{3,4\}$, $p\not\equiv 1 \pmod n$,
            and $K=\Q(\sqrt{-n})$ splits $B$.
        \end{itemize}

        \item[(ii)] $A \sim A_{\pi}$, where   $\pi=\pm \sqrt{q} \zeta_n$, $d(\pi)=2$, with $n \in \{3,4\}$, $p\equiv 1 \pmod n$, $K=\Q(\sqrt{-n})$, and there exists an isomorphism of $K$-algebras $B\otimes_\Q K \isoto  \calQ_{K, p}$.
        \end{itemize}
        \item[(2)] Case  $a$ odd: $A\sim A_\pi^2$,  $d(\pi)=1$,  and  $B$ is split by the imaginary quadratic field $K=\Q(\pi)$.

\end{enumerate}
\end{prop}
\begin{remark}\label{rem:exists-emb}
By Proposition \ref{prop:summ-Weil-num},
these conditions imply that
$\End^0(A)$ is isomorphic to $\Mat_2(D_{p, \infty})$, $\Mat_2(K)$, or $\calQ_{K,p}$, according to the cases (1)(i)(a), (1)(i)(b) and (2), or (1)(ii), respectively.
Note that there always exists an embedding of $\Q$-algebras $B \hookrightarrow \Mat_2(D_{p, \infty})$.
If $K$ is a field, then the following are equivalent:
\begin{enumerate}
    \item  There exists an injective $\Q$-algebra homomorphism $B \hookrightarrow \Mat_2(K)$;
    \item
    The field $K$ splits  $B$, i.e., $B\otimes_{\Q} K\simeq \Mat_2(K)$ as $K$-algebras;
    \item No prime  $\ell \mid \Delta$  splits   in $K$.
\end{enumerate}
Indeed, the equivalence (1) $\Leftrightarrow$ (2) follows from a simple dimension count, while
    (2) $\Leftrightarrow$ (3) follows from the local-global principle \cite[Proposition 14.6.7]{voight-quat-book}.
The relevant imaginary quadratic field $K$ in Proposition~\ref{prop:ssQMFq} is $\Q(\sqrt{-n})$ with
 $n\in \{3, 4\}$.
 For such a $K$, condition (3) above is equivalent to  $\ell \not\equiv 1 \pmod n$
  for every prime $\ell \mid \Delta$.

Similarly, if $K$ is a field in which $p$ splits,
then the following  are equivalent:
\begin{enumerate}[(i)]
    \item There exists an injective $\Q$-algebra homomorphism  $B\hookrightarrow \calQ_{K, p}$;
    \item There exists a $K$-algebra isomorphism  $B\otimes_\Q K \isoto  \calQ_{K, p}$;
    \item $p\mid \Delta$, and
    no prime $\ell\mid (\Delta/p)$ splits in $K$.
\end{enumerate}
The proof is similar to the case above, noting for the equivalence (ii) $\Leftrightarrow$ (iii) that, by definition, $\calQ_{K, p}$ is ramified exactly at the two places of $K$ lying above $p$ and splits elsewhere.
\end{remark}

\begin{proof}[Proof of Proposition~\ref{prop:ssQMFq}]
 By Remark~\ref{rem:exists-emb}, the conditions listed in the statement of Proposition~\ref{prop:ssQMFq} are sufficient for $A$ to admit QM by $B$.

Suppose conversely that $(A, \iota^0)/\Fq$ is an abelian surface with QM by $B$. Necessarily,
\begin{enumerate}[(i)]
    \item $\End^0(A)$ is  non-commutative; and
    \item $A$ is isotypic, that is, $A\sim A_\pi^{2/d(\pi)}$ for some supersingular Weil number $\pi$ with $d(\pi)\in \{1, 2\}$. Moreover, $\pi\neq \sqrt{q}$ when $a$ is odd.
\end{enumerate}
To prove the first part of (ii), suppose on the contrary that $A\sim A_1\times A_2$ with $\dim A_1=\dim A_2=1$ and $A_1\not\sim A_2$. Then $\End^0(A)=\End^0(A_1)\times \End^0(A_2)$. The embedding $\iota^0: B\hookrightarrow \End^0(A)$ induces an embedding $\iota_i^0: B\hookrightarrow \End^0(A_i)$ for each $i=1,2$.  Thus $\dim_\Q\End^0(A_i)\geq 4$,  and hence $\End^0(A_i)\simeq D_{p, \infty}$ for both $i$.
This  leads to a contradiction since $B$ is indefinite by our assumption. For the second part of (ii),   if $a$ is odd and $A\sim A_\pi$ with $\pi=\sqrt{q}$, then $\End^0(A)=D_{\infty_1, \infty_2}$. The embedding  $\iota^0: B\hookrightarrow \End^0(A)$ induces an isomorphism $B\otimes_\Q \Q(\sqrt{p})\simeq D_{\infty_1, \infty_2}$ of quaternion $\Q(\sqrt{p})$-algebras, which is clearly nonsense since $B\otimes_\Q \Q(\sqrt{p})$ is  totally indefinite  by our assumption on $B$ again.

A comparison with the list of Weil numbers $\pi=\pm\sqrt{q}\zeta_n$ with $d(\pi)\in \{1, 2\}$ in Proposition~\ref{prop:summ-Weil-num} shows that conditions (i) and (ii) above rule out the following cases:
\begin{enumerate}
    \item $a$ is even,  $n=5,8,12$ and $p\not\equiv 1 \pmod n$;
    \item $a$ is odd, and $d(\pi)=2$.
\end{enumerate}
For the remaining Weil numbers, the additional conditions listed in Proposition~\ref{prop:ssQMFq} are  necessary for  the existence of an embedding $B\hookrightarrow \End^0(A_\pi^{2/d(\pi)})$ by Remark~\ref{rem:exists-emb}.
\end{proof}

\begin{defn}
 An $O_B$-abelian surface $(A, \iota)$ over $S$ is said to satisfy the {\it determinant condition} if, for every $b\in O_B$, the characteristic polynomial of $\iota(b)$ acting on $\Lie(A)$ coincides with  the image of the reduced characteristic polynomial of $b\in O_B$ under the
natural  map $\Z[T] \to \calO_S[T]$, where $\calO_S$ denotes the structure sheaf of $S$.
\end{defn}

We are primarily concerned with $O_B$-abelian surfaces $(A, \iota)$ defined over a perfect field $k$ of characteristic $p$. The embedding $\iota: O_B\hookrightarrow \End(A)$ induces a $\Z_p$-embedding $\iota_p: O_B\otimes\Z_p\hookrightarrow \End_{\mathrm{DM}}(M(A))$, where $M(A)$ denotes the Dieudonn\'e module of $A$ as usual.  For simplicity,
we write  $B_p \coloneqq B \otimes_\Q \Qp$ and $O_{B, p} \coloneqq O_B \otimes \Z_p$ for the $p$-adic completions of $B$ and $O_B$,  respectively.
If $p \nmid \Delta$, then $B_p \simeq \Mat_2(\Q_p)$ and $O_{B, p} \simeq \Mat_2(\Z_p)$. If   $p\mid \Delta$, then $B_p$ is the quaternion  division  $\Q_p$-algebra, and $O_{B, p}$ coincides with the unique maximal $\Z_p$-order of $B_p$. In this case, we further abbreviate $O_{B, p}$ as $\calO_p$  to distinguish it from the split case.  In order to explicitly write down the action of $\calO_p$ on the Dieudonn\'e module $M(A)$, we fix a presentation of $\calO_p$ as follows.  Recall from Section~\ref{subsec:dieu-modules} that  $\Q_{p^2}$ denotes the unique unramified quadratic extension of $\Q_p$,  with the unique nontrivial automorphism $a\mapsto a^\sigma$ and  ring of integers $\Z_{p^2}\coloneq W(\F_{p^2})$. We  identify $\Q_{p^2}$ with a subfield of $B_p$ via a fixed $\Q_p$-embedding $\Q_{p^2}\hookrightarrow B_p$.
By  \cite[Corollaire~II.1.7]{vigneras},
there exists   $\Pi \in \calO_p$ such that
\begin{equation}\label{eq:O_p}
\calO_p=\Z_{p^2} \oplus \Z_{p^2}\Pi, \quad \Pi^2=-p, \quad \Pi a=a^{\sigma} \Pi \quad  \forall\, a\in \Z_{p^2}.
\end{equation}
Clearly, $\Pi$ generates the unique two-sided maximal ideal of $\calO_p$ with residue field   $\calO_p/(\Pi)=\Z_{p^2}/p\Z_{p^2}=\F_{p^2}$.
Observe that $\Lie(A)$ always carries an $\F_{p^2}\otimes_{\F_p} k$-module structure since
the action of $\Z_{p^2}\subseteq \calO_p$ on $\Lie(A)=M(A)/\sfV M(A)$ factors through $\Z_{p^2}/p\Z_{p^2}$.
If we further assume  that $k\supseteq \F_{p^2}$, then
restricting the $\O_p$-action on $M(A)$ to $\Z_{p^2}$ induces a decomposition of $M(A)$ as in \eqref{eq:dec}, which in turn  defines the  {\it{Lie type}} of $M(A)$ as in Definition \ref{defn:O-dieu-mod} (2).
In this case,  we define the Lie type of an $O_B$-abelian surface $(A, \iota)$ over $k$ to be that of the associated module $(M(A), \iota_p)$.
We will make use of the following fact.
\begin{lemma}\label{lem:det}
Let $(A, \iota)$ be an $O_B$-abelian surface over a perfect field $k$ of characteristic $p$.

(1) If $p\nmid \Delta$, then $(A, \iota)$ satisfies the
determinant condition.

(2) If $p \mid \Delta$, then the following are equivalent:
\begin{itemize}
\item[(a)] $(A, \iota)$ satisfies the determinant condition;
\item[(b)] $\Lie(A)$ is a free $\F_{p^2} \otimes_{\F_p} k$-module;
\item[(c)]  $(A, \iota)$ is of Lie type $(1,1)$, provided $k \supseteq \F_{p^2}$.
\end{itemize}
\end{lemma}
\begin{proof}
   These assertions follow from \cite[Lemmas  5.2 and 5.3]{Yu:Grenoble2021} and \cite[Proposition 4.3]{Drinfeld}.
\end{proof}

\begin{lemma}\label{lem:odd-degree-bdred}
    Let $(A, \iota)$ be an $O_B$-abelian surface   over a perfect field $k$ of characteristic $p$.
If $p\mid \Delta$, then $A$ is supersingular. If further $k=\Fq$ is a finite field of odd degree $a\coloneqq [\Fq:\Fp]$ over $\Fp$, then $(A, \iota)$ is a superspecial $O_B$-abelian surface satisfying the determinant condition.
Moreover, in this case $\Pi\in \calO_p$ vanishes  on $\Lie(A)$, and hence
the action of $\calO_p$ on $\Lie(A)$ factors through the quotient map $\calO_p\to \F_{p^2}$.
\end{lemma}
\begin{proof}
Suppose that $p\mid \Delta$.  It is well known that every $O_B$-abelian surface over $k$ in this case is supersingular; see \cite[Proposition~4.1]{Drinfeld} or \cite[Lemma~4.1]{Ribet-Bimodules}.

Next, suppose further $k=\Fq$ with $a\coloneqq [\Fq:\Fp]$ odd. Then $\F_{p^2}\otimes_{\Fp} \Fq$ is a field. Hence $\Lie(A)$ is automatically a  free $\F_{p^2}\otimes_{\Fp} \Fq$-module, so
     $(A,\iota)$ satisfies the determinant condition by Lemma~\ref{lem:det} (2).

     It remains to show that $A$ is superspecial in this case.
     Clearly, the action of $\calO_p$ on $\Lie(A)$ factors through $\calO_p/p\calO_p$, and hence $\Lie(A)$ carries an $(\calO_p/p\calO_p)\otimes_{\F_p} \Fq$-module structure.  From \eqref{eq:O_p}, $(\calO_p/p\calO_p)\otimes_{\F_p} \Fq$ is a local $\Fq$-algebra with the unique maximal  ideal generated by the image of $\Pi\otimes 1$ and with residue field $\F_{p^2}\otimes_{\F_p}\F_q$.  Necessarily,  $\Lie(A)$ is a simple $(\calO_p/p\calO_p)\otimes_{\F_p} \Fq$-module  since $\dim_{\Fq}\Lie(A)=2$. In particular, $\Pi$ vanishes on $\Lie(A)$. Now it follows from \cite[Proposition~4.2]{Ribet-Bimodules} that $A$ is superspecial.
\end{proof}

\begin{prop}\label{prop:move-to-sp}
    Let $(A,\iota^0:B\to \End^0(A))$ be a QM abelian surface over $\Fq$ with $q=p^a$.
    Then there exists an abelian surface $A'$ over $\Fq$, isogenous to $A$, and a homomorphism $\iota': O_B\to \End(A')$ such that $(A',\iota')$ is an $O_B$-abelian surface satisfying the determinant condition. Moreover, if $A$ is supersingular, one can choose  $A'$ to be  superspecial.
\end{prop}
\begin{proof}
    By Lemma~\ref{lm:move-to-OB}, after replacing $A$ by an isogenous abelian surface if necessary,  we may assume that $(A,\iota)$ is already an $O_B$-abelian surface.

    Assume that $p\nmid \Delta$.
    By Lemma \ref{lem:det} (1), every $O_B$-abelian surface satisfies the determinant condition. Moreover, in this case any supersingular $O_B$-abelian surface is automatically superspecial.

    Now assume that $p \mid \Delta$.  We separate the discussion into two cases according to the parity of $a=[\Fq:\Fp]$. If $a$ is odd, then from Lemma~\ref{lem:odd-degree-bdred} every $O_B$-abelian surface $(A, \iota)/\Fq$ is superspecial and satisfies the determinant condition.

   Now assume that $a$ is even.   Let $M\coloneqq M(A)$ be the covariant \dieu module of $A$ and set $N\coloneqq M\otimes_{\Zp} \Qp$.  We claim that
      it suffices to construct a superspecial   \dieu module $M'\subseteq N$  together with a $\Zp$-embedding $\varphi': \calO_p\to \End_{\rm DM}(M')$ such that $(M', \varphi')$ has Lie type $(1,1)$.
   By Proposition~\ref{prop:ssQMFq}, $\End_{\rm DM}(N)$ is isomorphic to one of the following semisimple $\Q_p$-algebras: (a) $\Mat_2(B_p)$, (b) $\Mat_2(K_p)$, or (c) $B_p\times B_p$, where $K_p/\Q_p$ denotes a quadratic field extension.
   By the Skolem-Noether theorem,  any two $\Qp$-embeddings $B_p\to \End_{\rm DM}(N)$ are conjugate under $\Aut_{\rm DM}(N)$.
   It follows that  $(M',\varphi')$ and $(M,\varphi)$ are $\calO_{p}$-linearly isogenous.
  Consequently,  there exists an $O_B$-linear isogeny  $\rho: (A',\iota')\to (A,\iota)$ with $M(A',\iota') \simeq (M',\varphi')$, verifying our claim.

It remains to construct the desired $\calO_p$-Dieudonn\'e module $(M', \varphi')$.   Write $\calO_p=\Z_{p^2}\oplus \Z_{p^2}\Pi$ as in \eqref{eq:O_p}.  From Waterhouse's construction recalled in Remark~\ref{rem:Waterhouse-constr}, we may assume that $\End_{\mathrm{DM}}(M)$ is  a maximal order in $\End_{\mathrm{DM}}(N)$. We take $M'=M$ for such an $M$, and define $\varphi': \calO_p\hookrightarrow \End_{\mathrm{DM}}(M')$ explicitly case by case as follows.  From the classification of $\End_{\mathrm{DM}}(N)$  above, there are three cases to consider:  (a) $\End_{\mathrm{DM}}(M')=\Mat_2(\calO_p)$, (b) $\End_{\mathrm{DM}}(M')=\Mat_2(O_{K_p})$, or  (c) $\End_{\mathrm{DM}}(M')=\calO_p\times \calO_p$.
Correspondingly,  we define $\varphi'$ as follows:
    \begin{itemize}
        \item [(a)] $\varphi'(\alpha
        )=\begin{bmatrix}
            \alpha & 0\\
            0 & \Pi \alpha \Pi^{-1}
        \end{bmatrix}$, for all $\alpha\in \calO_p$.
        \item[(b')] If $K_p=\Q_{p^2}$ is unramified over $\Q_p$, then we
        set $\varphi'(a+b\Pi)=\begin{bmatrix}
            a  & -p b \\
            b^\sigma & a^\sigma
        \end{bmatrix}$  for all $a,b\in \Z_{p^2}$.
        \item[(b'')] If $K_p$ is ramified over $\Q_p$, then we write $\calO_{p}=O_{K_p}\oplus O_{K_p} x$ for some $x\in \calO_p$, and define $\varphi':\calO_{p}\to \Mat_2(O_{K_p})$ to be the embedding given  by the regular representation.
        \item[(c)] $\varphi'(\alpha)=(\alpha, \Pi \alpha \Pi^{-1})$ for all $\alpha\in \calO_p$.
    \end{itemize}
    In cases (a), (b') and (c), the resulting $\calO_{p}$-\dieu module $(M',\varphi')$ clearly has Lie type $(1,1)$.
    In case (b''),  the action of $\Mat_2(O_{K_p})$ on $M'/\sfV M'$ factors through $\Mat_2(O_{K_p})\to \Mat_2(\Fp)$. Hence the action of $\calO_p$  on $M'/\sfV M'$ factors through $\calO_p \to \F_{p^2} \to \Mat_2(\Fp)$. In particular, $\Tr(\varphi'(a)\vert_{(M'/\sfV M')})\in \Fp$ for every $a\in \F_{p^2}$.   Therefore,  $M'/\sfV M'$ is a free $\F_{p^2}\otimes_{\Fp}\Fq$-module of rank $1$, and hence $(M',\varphi')$ has Lie type $(1,1)$.  Lastly, from the classification  of $\End_{\mathrm{DM}}(M')$, in all cases $M'$ decomposes as a direct sum of two supersingular Dieudonn\'e submodules. It follows that the $a$-number of $M'$ is  two, i.e.,~$M'$ is superspecial.
\end{proof}

We obtain the following theorem by combining Propositions~\ref{prop:summ-Weil-num}, \ref{prop:ssQMFq} and \ref{prop:move-to-sp}, where  the conditions in Proposition~\ref{prop:ssQMFq} for the existence of an embedding $B\hookrightarrow \End^0(A_\pi^{2/d(\pi)})$ have been reformulated into more concrete equivalence conditions according to  Remark~\ref{rem:exists-emb}.

\begin{thm}\label{thm:Fq-pts}
Let $n\geq 1$ with $n\not\equiv 2 \pmod 4$.
Let $q=p^a, \pi=\pm\sqrt{q}\zeta_n$, and $K=\Q(\pi)$.
Then there exists a superspecial $O_B$-abelian surface $(A, \iota)$ over $\Fq$ satisfying the determinant condition such that $A$ is isogenous to a power of $A_\pi$
if and only if one of the following  holds:

\begin{enumerate}
\item $a$ is even and either
\begin{itemize}
\item[(i)] $n=1$, or
\item[(ii)] $n= 3 \text{ or } 4$ and either
\begin{itemize}
\item[(a)] $p \not\equiv 1 \pmod n$ and
no prime $\ell \mid \Delta$ splits in $K$, or
\item[(b)] $p \equiv 1 \pmod n, p \mid \Delta$, and
no prime  $\ell \mid (\Delta/p)$  splits in $K$.
\end{itemize}
\end{itemize}
\item $a$ is odd, $\pi=\sqrt{q}\zeta_4$ or $\pm \sqrt{q}\zeta_{4p}\ (p=2, 3)$,
and
no prime  $\ell \mid \Delta$ splits in $K$.
\end{enumerate}
\end{thm}

\subsection{Moduli interpretation}
We now recall the moduli interpretation of the Shimura curve (\cite{Drinfeld}).
 Let $X=X_{O_B}$ be  the contravariant functor from the category of schemes to the category of sets that assigns to a  scheme $S$ the set of isomorphism
classes of $O_B$-abelian surfaces $(A, \iota)$ over $S$ satisfying  the determinant condition.
This functor is coarsely represented by a $\Z$-scheme which is flat, quasi-projective and of relative dimension one, again denoted by $X$.
It is proper if and only if $B$ is a division algebra.

For an $O_B$-abelian surface $(A, \iota)$ over $\Fpbar$ and an automorphism $\phi$ of $\Fpbar$, we write ${A}^\phi$ for the $\phi$-twist of $A$,
and ${\iota}^\phi: O_B\to \End({A}^\phi)$ for the induced $O_B$-action (see  \S\ref{ss:av}).
The {\it{field of moduli}} $\F_{(A, \iota)}$
of $(A, \iota)$ is defined to be  the minimal subfield $\F_{(A, \iota)}\subseteq \Fpbar$ such that for every  $\phi \in \Gal(\Fpbar/\F_{(A, \iota)})$,    there exists an isomorphism $(A, \iota) \simeq ({A}^\phi, {\iota}^\phi)$ over $\Fpbar$.
The set $X(\F_q)$ of $\F_q$-rational points of $X$ consists precisely of all the points  $[(A, \iota)] \in X(\Fpbar)$ such that $\F_{(A, \iota)}\subseteq \F_q$.
By definition, any $(A, \iota)$ defined over
$\F_q$
 determines a point of
$X(\F_q)$.
The following lemma shows that  over finite fields the converse also holds.
\begin{lemma}
\label{lm:fomfod}
  Let $(A,\iota)$ be an $O_B$-abelian surface over
  $\Fpbar$.
  If $\F_{(A, \iota)}=\Fq$, then $(A, \iota)$ admits a model over
$\F_q$.
  In particular, every  element of  $X(\Fq)$ is represented by an $O_B$-abelian surface defined over $\Fq$.
\end{lemma}
\begin{proof}
  Analogous results for certain polarized abelian varieties were obtained in \cite[Lemma 2.2]{Ibukiyama-Katsura-1994}, \cite[Lemma 4.1]{Ibukiyama:Osaka2018},  \cite[Lemma 2.1]{yu:field-of-def-2017}, and \cite[Lemma 5.4]{Yu-Indiana-ArithSSlocus}.
  Let $\phi:x\mapsto x^q$ denote the Frobenius automorphism in
  $\Gal(\Fpbar/\Fq)$.
Since $\F_{(A, \iota)}=\Fq$, there exists an
  isomorphism
  \[ a_\phi: (A,\iota)\xrightarrow{\;\sim\;}
  (\,{A}^\phi,\,{\iota}^\phi)\]
  of $O_B$-abelian surfaces over $\Fpbar$.

  Since the automorphism group $\Aut(A,\iota)$ is finite,  there exists a positive integer $n$ such that both of the following conditions hold:
  \begin{equation}
   \label{eq:611}
  (A^{\phi^n},\iota^{\phi^n})=(A,\iota)\quad\text{and}\quad  \phi^{n-1}(a_\phi)\cdots \phi(a_\phi) a_\phi=1;
  \end{equation}
  see the argument in \cite[Lemma 5.4]{Yu-Indiana-ArithSSlocus}.
Thus, by Weil's descent criterion \cite[Theorem 3]{weil:fod} (also see \cite[Section 2.4]{Yu-Indiana-ArithSSlocus}), there exists an abelian
surface $A'$ over $\Fq$ and an isomorphism $f:A'\otimes_{\Fq} \Fpbar\simeq
A$ of abelian surfaces over $\Fpbar$ such that $a_\phi={f}^\phi \circ f^{-1}$.

For each $b \in O_B$, we define
\[ \iota'(b)\coloneqq  f^{-1} \circ \iota(b) \circ f:A'\otimes_{\Fq} \Fpbar\to A'\otimes_{\Fq} \Fpbar. \]
Then the identity $\iota(b) =a_\phi^{-1} \circ {\iota^\phi(b)} \circ a_{\phi}$ implies that
\begin{equation*}
  \label{eq:612}
  \begin{split}
    \iota'(b)
      = f^{-1} \circ a_\phi^{-1} \circ {\iota^\phi(b)} \circ  a_\phi \circ f ={(f^{-1})}^\phi
      \circ{\iota^\phi(b)} \circ  {f}^\phi={\iota'^\phi(b)}.
  \end{split}
\end{equation*}
Hence $\iota'$ is defined over $\F_q$, and therefore  $(A',\iota')$ is a model  of
$(A,\iota)$  over $\Fq$.
\end{proof}

\subsection{Superspecial and supersingular loci}
\begin{defn}\label{defn:sp-ss-loci}
 Let $X_{\Fp}\coloneqq X\otimes_\Z \Fp$ denote the special fiber of the $\Z$-scheme $X$ at the prime $p$.
  The
 {\it superspecial locus} and {\it supersingular locus}
 \[ X^{\rm sp} \subseteq X^{\rm ss} \subseteq X_{\F_p} \]
 are defined as the largest reduced closed subschemes of $X_{\F_p}$ such that
 \begin{align*}
 X^{\rm sp}(\Fpbar) & = \{[(A, \iota)]  \in X(\Fpbar) \mid   \text{$A$ is superspecial}\},
 \\
  X^{\rm ss}(\Fpbar) & = \{[(A, \iota)]  \in X(\Fpbar) \mid   \text{$A$ is supersingular}\}.
 \end{align*}
 \end{defn}
 By abuse of notation, we also write $X^{\rm sp}$ for the set  $X^{\rm sp}(\Fpbar)$, which is  non-empty by Theorem \ref{thm:Fq-pts}.
 By Lemma \ref{lm:fomfod}, the existence of $\F_q$-rational superspecial points is equivalent to the existence of a superspecial $O_B$-abelian surface over $\F_q$ satisfying the determinant condition.
  Theorem~\ref{thm:Fq-pts} provides a criterion for the existence of such a surface.
Moreover,
we have $X^{\rm sp}(\Fpbar)=X^{\rm sp}(\F_{p^2})$.
This follows from Corollary \ref{cor:sp_fod}, which states that any superspecial $O_B$-abelian surface
over $\Fpbar$ has a  canonical model
over $\F_{p^2}$.
Finally, the supersingular locus satisfies  $X^{\rm ss}=X^{\rm sp}$ if $p \nmid \Delta$, and $X^{\rm ss}=X_{\Fp}$ if $p \mid \Delta$ (see \cite[\S 3]{Ribet-Bimodules} and Lemma \ref{lem:odd-degree-bdred}).
\begin{prop}\label{prop:K_odd}
Suppose that $p \mid \Delta$, and let $a\in \Z_{>0}$ be an odd integer. Then
\[X^{\mathrm{sp}}(\Fp)=X(\Fp)=X(\F_{p^a}).\]
\end{prop}
\begin{proof}
    From Lemma~\ref{lem:odd-degree-bdred},  every $O_B$-abelian surface $(A, \iota)$ over a finite field $\F_{p^a}$ of odd degree is automatically superspecial in this case.
    We then have inclusions $X(\Fp)\subseteq X(\F_{p^a})\subseteq X^{\rm sp}\subseteq X(\F_{p^2})$.
   Since $\F_{p^a}\cap\F_{p^2}=\Fp$, the assertion follows.
\end{proof}

\begin{cor}\label{cor:XFodd}
    Let $\Fp^{\rm odd}$ be the union of all finite extensions of $\Fp$ of odd degree in $\Fpbar$. If $p \mid \Delta$, then $X(\Fp^{\rm odd})$ is finite.
\end{cor}

\begin{remark}
   By the Weil conjectures, the number of $\F_{p^a}$-rational points of any smooth projective curve over $\F_p$ tends to infinity as $a \to \infty$.
   In contrast, Proposition~\ref{prop:K_odd}  presents a projective curve  $X_{\F_p}$ with only ordinary double points such that $\lvert X(\F_{p^a}) \rvert$  remains bounded along all odd integers $a$.
   This behavior occurs because $X(\F_{p^a})$ is contained in the singular locus
   $X^{\mathrm{sing}}$ for all odd integers $a$  (see \S\ref{ss:bad} and Lemma \ref{lem:twist} (2) for $X(\Fp)\subseteq \Sigma_3=X^{\rm sing}$).
\end{remark}

For the remainder of this section, we assume that $p \mid \Delta$.
In contrast to the odd degree case, when $a$ is even, not all $O_B$-abelian surfaces over $\F_{p^a}$ are superspecial.
\begin{prop}\label{prop:ss-but-nonssp}
  Suppose that  there exists a QM abelian surface $(A, \iota^0: B\hookrightarrow \End^0(A))$ over $\F_{p^2}$ such that $A\sim A_\pi^{2/d(\pi)}$ for some supersingular Weil $p^2$-number $\pi\in \{\pm p\zeta_n\mid n=3, 4\}$;  equivalently, suppose that condition (1.ii) of Theorem~\ref{thm:Fq-pts} holds for some $\pi=\pm p\zeta_n$.
    Then $X^{\mathrm{sp}}$ is \emph{properly} contained in $X(\F_{p^2})$. In other words,  there exists a supersingular but non-superspecial $O_B$-abelian surface $(A', \iota': O_B\hookrightarrow \End(A'))$ over $\F_{p^2}$ with Lie type $(1, 1)$.
\end{prop}
\begin{proof}
The idea of the proof is similar to  that  of Proposition~\ref{prop:move-to-sp}.
By Lemma~\ref{lm:move-to-OB}, we may assume that $(A, \iota)$ is already an $O_B$-abelian surface.
Let $M(A)$ denote  the Dieudonn\'e module of $A$, and set  $N\coloneqq M(A)\otimes_{\Z_p}\Q_p$.
Mapping the Weil $p^2$-number $\pi$ to the Frobenius endomorphism $\pi_A\in \End(A)$ induces an embedding $\Q(\pi)\hookrightarrow \End^0(A)$.
By the argument in the proof of  \cite[Lemma~X.1.3]{van-der-Geer-HMS}, the rational \dieu module $N$ is a free module of rank two over $\Q_{p^2}(\zeta)\coloneqq \Q_{p^2}\otimes_\Q\Q(\pi)$,  where $\zeta\coloneqq \pi/p$.
The embedding $\iota^0: B\hookrightarrow \End^0(A)$ endows  $N$ with the structure of a  left $B_p\otimes_{\Q_p}\Q_{p^2}(\zeta)$-module.
 The operators $\sfF$ and $\sfV$ act $B_p\otimes_{\Q} \Q(\zeta)$-linearly on $N$ and satisfy  the usual   relations  \eqref{eq:W[FV]}, together with the identity  $\sfV^2=p\zeta$.

Let $\calO_p=O_B\otimes\Z_p$ be the unique maximal order in the quaternion division $\Q_p$-algebra $B_p$.
In Lemma~\ref{lem:const-dieu-mod}  we   construct a non-superspecial  $\calO_p$-stable  Dieudonn\'e module $M'\subseteq N$ of full rank with  Lie type $(1,1)$.
This construction gives rise to an  $O_B$-linear quasi-isogeny $\rho: (A', \iota')\to (A, \iota)$ over $\F_{p^2}$, uniquely determined  up to isomorphism by the  conditions:
\[\rho_p(M(A'))=M' \qquad \text{and}\qquad \rho_\ell(T_\ell(A'))=T_\ell(A), \qquad \forall \ell\neq p,  \]
where $T_{\ell}(A)$ is the $\ell$-adic Tate module of $A$.
Then $(A', \iota')$ is  the desired supersingular but   non-superspecial $O_B$-abelian surface over $\F_{p^2}$ with Lie type $(1,1)$.
 To carry out the construction explicitly, we need a more concrete description of the structure of $N$;   this is provided  in  Lemma~\ref{lem:F-isocrystal-str} below.
\end{proof}

\begin{lemma}\label{lem:F-isocrystal-str}
    Keep  the notation and assumptions of Proposition~\ref{prop:ss-but-nonssp}.
    In particular, let   $N\coloneqq M(A)\otimes_{\Z_p}\Q_p$ be the covariant rational Dieudonn\'e module  of $A/\F_{p^2}$, which is a free module of rank two over $\Q_{p^2}(\zeta)\coloneqq \Q_{p^2}\otimes \Q(\pi)$, where $\zeta\coloneqq \pi/p$.
    Then
    there exists an element $e\in N$ such that $e_1\coloneqq e$ and  $e_2\coloneqq \sfV e $ form a $\Q_{p^2}(\zeta)$-basis of $N$.
    With respect to the basis $(e_1, e_2)$,
    the Frobenius   and Verschiebung  operators $\sfF$ and $\sfV$ are given by
\begin{equation}
  \label{eq:2}
  \sfF=
  \begin{bmatrix}
    0 & p \\ \zeta^{-1} & 0
  \end{bmatrix}\sigma, \qquad \sfV=\zeta \sfF=  \begin{bmatrix}
    0 & p\zeta \\ 1 & 0
  \end{bmatrix}\sigma^{-1}.
\end{equation}
Here $\sigma$ denotes the automorphism of $\Q_{p^2}(\zeta)\coloneqq\Q_{p^2}\otimes_{\Q} \Q(\pi)$ acting as the usual Frobenius automorphism on the first factor  $\Q_{p^2}$ and as the identity on the second factor  $\Q(\pi)$.
Moreover, any two $\Q_p$-embeddings $B_p\hookrightarrow \End_{\mathrm{DM}}(N)$ are conjugate under $\Aut_{\mathrm{DM}}(N)$.
\end{lemma}
\begin{proof}
 Since $A$ is an abelian surface over $\F_{p^2}$ that is isogenous to a power of $A_\pi$ for some supersingular Weil $p^2$-number $\pi=p\zeta$,  its covariant rational Dieudonn\'e module $N$
 satisfies the following conditions:
\begin{enumerate}[(i)]
    \item $N$ is a left module over the Dieudonn\'e algebra $\grD_\pi^0\coloneqq \Big(\Q_{p^2}\otimes \Q(\pi)\Big)[\sfV]/(\sfV^2-\pi)$, where $\sfV\pi=\pi\sfV$  and  $\sfV a=a^{\sigma^{-1}}\sfV$ for every   $a\in \Q_{p^2}$;
    \item  restriction of scalars to the central subalgebra $\Q_p\otimes \Q(\pi)\subseteq \grD_\pi^0$ makes  $N$ a free $\Q_p\otimes \Q(\pi)$-module of rank $4=\frac{2[\Q_{p^2}:\Q_p]\dim A}{[\Q(\pi):\Q]}$.
\end{enumerate}
Indeed, it was already shown in the proof of Proposition~\ref{prop:ss-but-nonssp} that  $N$ is free  of rank two over $\Q_{p^2}(\zeta)=\Q_{p^2}\otimes \Q(\pi)$.
Write  $\Q_p\otimes \Q(\pi)=\prod_{v\mid p} \Q(\pi)_v$, where $v$ runs over the  places of $\Q(\pi)$ above $p$, and $\Q(\pi)_v$ denotes the $v$-adic completion of $\Q(\pi)$.
Correspondingly,
\[\grD_\pi^0=\prod_{v\mid p} \grD_{\pi, v}^0, \qquad \text{where}\qquad  \grD_{\pi, v}^0\coloneqq \Big(\Q_{p^2}\otimes_{\Q_p} \Q(\pi)_v\Big)[\sfV]/(\sfV^2-\pi).\]
By \cite[Theorem~A.1.2.1]{Chai-Conrad-Oort}, each algebra  $\grD_{\pi, v}^0$ is central simple over $\Q(\pi)_v$.
Up to isomorphism, there is at most one left $\grD_{\pi, v}^0$-module of a given dimension over $\Q(\pi)_v$.
It follows that the rational \dieu module $N$ is uniquely characterized (up to isomorphism) by conditions (i)--(ii)\footnote{The same argument applies to any abelian variety $A$ over a finite field $\F_q$ isogenous to  a power of $A_\pi$ for some Weil $q$-number $\pi$. }.

On the other hand, let $N_0$ be a free $\Q_{p^2}\otimes \Q(\pi)$-module of rank two with basis $\{e_1, e_2\}$,
and define a rational \dieu module structure over $\F_{p^2}$ by requiring  $\sfF$ and $\sfV$ to act  as in \eqref{eq:2}.
A direct verification shows that $N_0$ satisfies  conditions (i)--(ii).
Hence $N\simeq N_0$.

Finally,  Proposition~\ref{prop:move-to-sp} shows that $\End_{\mathrm{DM}}(N)$ is $\Q_p$-isomorphic to either $B_p\times B_p$ or $\Mat_2(\Q(\zeta)\otimes \Q_p)$ depending on whether $p$ splits in $\Q(\zeta)$ or not. By the Skolem–Noether Theorem,
any two $\Q_p$-embeddings $B_p\hookrightarrow \End_{\mathrm{DM}}(N)$ are conjugate under $\End_{\mathrm{DM}}(N)^\times$.
\end{proof}

  We construct a
supersingular $\O_p$-Dieudonn\'e module  over $\F_{p^2}$
that is \emph{not superspecial}.

\begin{lemma}\label{lem:const-dieu-mod}
 Keep the rational Dieudonn\'e module $N$ over $\F_{p^2}$ fixed as above.
Let $\iota^0_p: B_p\hookrightarrow \End_{\mathrm{DM}}(N)$ be the embedding induced from $\iota^0: B\hookrightarrow \End^0(A)$.
Then there exists a  non-superspecial $\iota_p^0(\calO_p)$-Dieudonn\'e module $M'$ of full rank in $N$ with Lie type $(1,1)$.
\end{lemma}
\begin{proof}
We construct a full $\Z_{p^2}$-lattice inside
$N$ that is stable under $\sfF$, $\sfV$, and $\iota_p^0(\O_p)$,  but not under  $\zeta$, ensuring that it is not superspecial by Lemma~\ref{lem:ssp-zeta-stable}.
Let $e_1, e_2$ be the $\Q_{p^2}(\zeta)$-basis of $N$ from Lemma~\ref{lem:F-isocrystal-str}.
Write
$R_0\coloneqq \Z_{p^2}[\zeta]$ for the maximal order of
$\Q_{p^2}(\zeta)\coloneqq \Q_{p^2}\otimes_\Q \Q(\zeta)$, and $R\coloneqq \Z_{p^2}[p\zeta]$ for its
$\Z_{p^2}$-suborder of conductor $p$.
Consider the $R$-lattice $M\coloneqq Re_1\oplus R_0e_2 \subseteq N$.
It is clearly $\sigma$-stable.
 Moreover,  $\zeta M\neq M$ since
$\zeta\not\in R$.  A direct calculation shows that
\begin{equation}
  \label{eq:4}
  \End_R(M)=\begin{bmatrix}
    R & pR_0 \\ R_0 & R_0
  \end{bmatrix}=\left\{  \begin{bmatrix}
    a & pb \\ c & d
  \end{bmatrix} \, \middle \vert\, a\in R, \, b,  c, d \in R_0  \right\}.
\end{equation}
From the matrix description of $\sfF$ and $\sfV$ in \eqref{eq:2}, it follows that  $M$ is stable under both $\sfF$ and $\sfV$.
Hence $M$ is a Dieudonn\'e module of full rank in $N$.  If $M$ were superspecial, Lemma~\ref{lem:ssp-zeta-stable} would imply
$\zeta M=M$, a contradiction.
Thus $M$ is not superspecial.

To equip $M$ with an $\calO_p$-action, we compute its  endomorphism
ring:
\begin{equation}
  \label{eq:5}
 \End_{\mathrm{DM}}(M)=\left\{  \begin{bmatrix}
    a & p\zeta c^\sigma \\ c & a^\sigma
  \end{bmatrix} \, \middle \vert\, a\in R, \,  c\in R_0  \right\}.
\end{equation}
Since $\Q_{p^2}(\zeta)/\Q_p(\zeta)$ is  \emph{unramified} quadratic, there exists $c_0\in R_0^\times$ such that
$c_0c_0^\sigma=-\zeta^{-1}$.
Define  $\varpi\coloneqq
\begin{bmatrix}
  0 & p\zeta c_0^\sigma \\ c_0 & 0
\end{bmatrix}\in \End_{\mathrm{DM}}(M)$.  Then
\[\varpi^2= \begin{bmatrix}
  -p & 0 \\ 0 & -p
\end{bmatrix},\qquad \varpi\begin{bmatrix}
  a & 0 \\ 0 & a^\sigma
\end{bmatrix} \varpi^{-1}=\begin{bmatrix}
a^\sigma   & 0 \\ 0 & a
\end{bmatrix}, \qquad \forall a\in R. \]
Recall that $\calO_p$ admits a  presentation $\O_p=\Z_{p^2}\oplus\Z_{p^2}\Pi$ with relations in  \eqref{eq:O_p}.
Thus we obtain  an embedding  $\psi: \calO_p\hookrightarrow \End_{\mathrm{DM}}(M)$ defined by
\[\Pi\mapsto \varpi=\begin{bmatrix}
  0 & p\zeta c_0^\sigma \\ c_0 & 0
\end{bmatrix}, \qquad a\mapsto \begin{bmatrix}
  a & 0 \\ 0 & a^\sigma
\end{bmatrix}, \qquad \forall a\in \Z_{p^2}. \]
 Let $\bar{e}_1, \bar{e}_2$ denote  the images of $e_1, e_2$ in $\Lie(M)=M/\sfV M$.
 Then
$\Lie(M)=\Lie(M)^0\oplus \Lie(M)^1$, where  $\Lie(M)^0=(R/pR_0)\bar{e}_1$ and $\Lie(M)^1=(R_0/R)\bar{e}_2$.
Each summand has dimension one  over $\F_{p^2}$, so $(M, \psi)$ has  Lie type $(1,1)$.

Finally,  $\psi$ induces
$\psi^0: B_p\hookrightarrow \End_{\mathrm{DM}}(N)$.
By Lemma~\ref{lem:F-isocrystal-str}, there exists $g\in \Aut_{\mathrm{DM}}(N)$ such that $\iota_p^0=g \psi^0 g^{-1}$.
Setting  $M'\coloneqq gM$, we obtain  a  non-superspecial     $\iota_p^0(\calO_p)$-stable Dieudonn\'e module of full rank in $N$ with Lie type $(1,1)$, as desired.
\end{proof}

\begin{cor}
   Keep the assumption that $p\mid \Delta$. We have  $X^{\mathrm{sp}}=X(\F_{p^2})$  if and only if there exist two  prime divisors $\ell, \ell'$ of $\Delta$ (not necessarily distinct) such that both the Kronecker symbols $\left(\frac{-3}{\ell}\right)$ and $\left(\frac{-4}{\ell'}\right)$ take the value $1$,  subject to the following additional constraints depending on the residue of $p$ modulo $12$:
   \begin{enumerate}[(i)]
       \item case $p=2, 3$ or $p\equiv 1, 11\pmod{12}$: $\ell\neq p$ and $\ell'\neq p$;
       \item case $p\equiv 5\pmod{12}$: $\ell'\neq p$;
       \item case $p\equiv 7\pmod{12}$: $\ell\neq p$.
   \end{enumerate}
\end{cor}
\begin{proof}
    We leave it to the interested reader to check that the proposed characterization of the equality $X^\sp=X(\F_{p^2})$ here is precisely the negation of the assumption in Proposition~\ref{prop:ss-but-nonssp}, whence the necessity is immediate. Conversely, let $(A, \iota)$ be a (supersingular) $O_B$-abelian surface over $\F_{p^2}$.   As condition (1.ii) of Theorem~\ref{thm:Fq-pts} fails by our assumption, $A$ is isogenous to $A_\pi^2$ with $\pi=\pm p$. We conclude from Lemma~\ref{lem:ssp-zeta-stable} that $A$ is superspecial.
\end{proof}

\section{Conjugacy classes of Embeddings $\calO_p\hookrightarrow \Mat_2(\calO_p)$}\label{ss:conj}

Keep $B$, $\Delta$, and $O_B$ as in Section \ref{ss:exist}, and let $k$ be an algebraically closed field of characteristic $p$ with a fixed inclusion $\F_{p^2}\subseteq k$.  For the rest of this paper, we focus on the classification and counting of the isomorphism classes of \emph{superspecial} $O_B$-abelian surfaces over $k$.
\subsection{Classification of superspecial $O_B$-abelian $k$-surfaces into genera}
Fix a supersingular elliptic curve $E$ over $k$. We have seen in the introduction that $\End^0(E)$ coincides with the unique quaternion $\Q$-algebra $D=D_{p, \infty}$ ramified exactly at $p$ and $\infty$, and its endomorphism ring $O_D\coloneqq \End(E)$ is a maximal order in $D$.    From the Deligne-Ogus-Shioda Theorem, every superspecial abelian surface $A$ over $k$ is isomorphic to $E^2$, and hence $\End(A)\simeq\Mat_2(O_D)$.
Thus each embedding $\iota: O_B\hookrightarrow \End(A)$ can be described more concretely as an embedding $O_B\hookrightarrow \Mat_2(O_D)$. A different identification of $A$ with $E^2$ yields another embedding $O_B\hookrightarrow \Mat_2(O_D)$ that is $\GL_2(O_D)$-conjugate to the original one. In particular, we get the following lemma; see \cite[p.~38]{Ribet-Bimodules}.

\begin{lemma}\label{lem:isom-conj-4.1}
    There is a bijection between the following two sets
\begin{equation}\label{eq:x4.1}
 \left\{\parbox{4.2cm}{Isomorphism
classes of superspecial $O_B$-abelian surfaces over $k$}\right\} \longleftrightarrow  \left\{\parbox{4cm}{$\GL_2(O_D)$-conjugacy classes of embeddings $O_B\hookrightarrow
\Mat_2(O_D)$}\right\}.
\end{equation}
\end{lemma}
By the Skolem-Noether theorem, all embeddings $B\hookrightarrow \Mat_2(D)$ form a single $\GL_2(D)$-conjugacy class.  Consequently, any two superspecial $O_B$-abelian surfaces over $k$ are $O_B$-linearly isogenous.  Lemma~\ref{lem:isom-conj-4.1} shows that, to classify them up to $O_B$-linear isomorphism, we need to treat the integral version of the Skolem-Noether problem on the right hand side of \eqref{eq:x4.1}.
As usual, we take $\ell$-adic completions and treat the analogous local problem first.  Correspondingly, for each superspecial $O_B$-abelian surface $(A, \iota)$ over $k$,
we study the induced $O_B\otimes\Z_\ell$-action on the $\ell$-divisible group $A(\ell)\coloneqq A[\ell^\infty]$ for every prime $\ell$ (including $\ell=p$).

\begin{defn}\label{defn:genus}
 Two  superspecial $O_B$-abelian surfaces $(A, \iota)$ and $(A', \iota')$  over $k$ are said to \emph{belong to the same genus} if $(A(\ell), \iota_\ell)$ is isomorphic to $(A'(\ell), \iota'_\ell)$ for every prime $\ell$ (including $\ell=p$). Equivalently, $(A, \iota)$ and $(A', \iota')$ belong to the same genus if and only if both of the following conditions hold:
 \begin{enumerate}[(i)]
     \item the $O_B\otimes\Z_\ell$-Tate modules $(T_\ell(A), \iota_\ell)$ and $(T_\ell(A'), \iota_\ell')$ are isomorphic for every prime $\ell\neq p$; and
     \item the covariant $O_B\otimes\Z_p$-Dieudonn\'e modules $(M(A), \iota_p)$ and $(M(A'), \iota_p')$ are isomorphic.
 \end{enumerate}
\end{defn}
See \cite[Definition~5.5]{xue-yu:counting-av} for the generalization of the \emph{genus} concept to the PEL-type setting, and  see also \cite{xue-yang-yu:ECNF, xue-yu:ppas} for applications of it to the counting of (principally polarized) abelian surfaces with real Weil number $\pm \sqrt{p}$.

We now classify superspecial $O_B$-abelian surfaces  $(A, \iota)$ over $k$ into genera. For each prime $\ell$ distinct from $p$, the Tate module $T_\ell(A)$ is a free left module of rank one over the maximal quaternion $\Z_\ell$-order $O_B\otimes\Z_\ell$ since $\rank_{\Z_\ell}T_\ell(A)=4$.  In particular, there is a unique isomorphism class of such $O_B\otimes\Z_\ell$-Tate modules.
From the $p$-part of Tate's theorem \cite[Theorem~A.1.1.1]{Chai-Conrad-Oort}, we have $\End_{\mathrm{DM}}(M(A))\simeq \Mat_2(O_D\otimes\Z_p)$.
The same argument as that of Lemma~\ref{lem:isom-conj-4.1} shows that the classification of  superspecial $O_B\otimes\Z_p$-Dieudonn\'e modules of $W(k)$-rank $4$ is equivalent to the classification of $\GL_2(O_D\otimes\Z_p)$-conjugacy classes of embeddings $O_B\otimes \Z_p\hookrightarrow \Mat_2(O_D\otimes\Z_p)$.
We separate the classification  into two cases depending on whether $p$ is split in $B$ or not. First, suppose that
$p\nmid \Delta$. In this case, $B\otimes \Q_p\simeq \Mat_2(\Q_p)$, and $O_B\otimes \Z_p \simeq \Mat_2(\Z_p)$.
By Morita  equivalence and the maximality of $O_D\otimes \Z_p$, there is a unique  $\GL_2(O_D\otimes\Z_p)$-conjugacy class of embeddings $\Mat_2(\Z_p)\hookrightarrow \Mat_2(O_D\otimes\Z_p)$.
Next, suppose that $p\mid \Delta$. In this case, $B_p$ is the  quaternion division $\Q_p$-algebra, and $\calO_p\coloneqq O_B\otimes\Z_p$ is the unique maximal order in $B_p$. On the other hand, $O_D\otimes\Z_p$ coincides with the unique maximal order in the quaternion division $\Q_p$-algebra $D_p\coloneqq D\otimes_{\Q} \Q_p$ as well.
For the rest of the paper, we fix an identification\footnote{In \cite[\S2]{Ribet-Bimodules}, Ribet defines  \emph{the orientations of $O_B$ and $O_D$ at $p$} to be fixed choices of residue maps $O_B\to \F_{p^2}$ and $O_D\to \F_{p^2}$ at $p$ respectively. As explained in  \cite[p.~37]{Ribet-Bimodules}, the specification of the orientations of $O_B$ and $O_D$ at $p$ determines an isomorphism $O_B\otimes \Z_p\simeq O_D\otimes \Z_p$ uniquely up to $(O_B\otimes \Z_p)^\times$-conjugacy. Since we are only concerned with $\GL_2(O_D\otimes\Z_p)$-conjugacy classes of embeddings $O_B\otimes \Z_p\hookrightarrow \Mat_2(O_D\otimes\Z_p)$, this has the same effect as fixing an identification $O_B\otimes \Z_p=O_D\otimes \Z_p$ as in  \eqref{eq:orientation-ident}.}
\begin{equation}\label{eq:orientation-ident}
    O_D\otimes \Z_p=\calO_p\coloneqq O_B\otimes \Z_p
\end{equation}
whenever $p\mid \Delta$. This reduces the  classification of $\GL_2(O_D\otimes\Z_p)$-conjugacy classes of embeddings $O_B\otimes \Z_p\hookrightarrow \Mat_2(O_D\otimes\Z_p)$ to that of $\GL_2(\calO_p)$-conjugacy classes of embeddings $\calO_p\hookrightarrow \Mat_2(\calO_p)$. Given a superspecial $O_B$-abelian surface $(A, \iota)$ over $k$, we write $\varphi_{(A, \iota)}: \calO_p\hookrightarrow \Mat_2(\calO_p)$ for the embedding thus obtained. It is uniquely determined up to $\GL_2(\calO_p)$-conjugation by the above discussion. In summary, we have obtained the following result.

\begin{prop}\label{prop:genus}
    (1) If $p\nmid \Delta$, all superspecial $O_B$-abelian surfaces $(A, \iota)$ over $k$ belong to the same genus. \\
    (2) If $p\mid \Delta$, then the genus of a  superspecial $O_B$-abelian $k$-surface $(A, \iota)$ is uniquely determined by the $\GL_2(\calO_p)$-conjugacy class of the associated embedding $\varphi_{(A, \iota)}: \calO_p\hookrightarrow \Mat_2(\calO_p)$.
\end{prop}

To emphasize the division property, for the rest of this section we write $D_p$ for the unique quaternion division $\Q_p$-algebra, whose maximal order is still denoted by $\calO_p$ thanks to \eqref{eq:orientation-ident}.
 As in  \eqref{eq:O_p}, one may write
  \begin{equation}\label{eq:6.1}
\calO_p=\Z_{p^2} \oplus \Z_{p^2}\Pi, \quad \Pi^2=-p, \quad \Pi a=a^{\sigma} \Pi \quad {\text{for all}} \quad a\in \Z_{p^2}.
\end{equation}
We derive a complete classification of $\GL_2(\calO_p)$-conjugacy classes of embeddings $\{\varphi:
\calO_p\hookrightarrow \Mat_2(\calO_p)\}$. The main results of this section, namely Theorem~\ref{thm:5-emb} and Proposition~\ref{prop:Lie},  will be applied in Section \ref{ss:as} below for the study of  the bad reduction of Shimura curves. The reader may wish to skip their proofs on a first reading.

\begin{thm}\label{thm:5-emb}
There are exactly \emph{five} $\GL_2(\calO_p)$-conjugacy classes of
embeddings $\calO_p\hookrightarrow \Mat_2(\calO_p)$,  represented by the following maps:
\begin{alignat*}{4}
 \varphi_1:&\qquad& \Pi\mapsto&   \begin{bmatrix}
    0 & -p \\ 1 & 0
                                \end{bmatrix}, &  \qquad  a&\mapsto
  \begin{bmatrix}
    a & 0 \\ 0 & a^\sigma
  \end{bmatrix}, &\quad \forall &a\in \Z_{p^2} ;\\
 \varphi_2:&\qquad&  \Pi\mapsto&   \begin{bmatrix}
    0 & -p \\ 1 & 0
                                              \end{bmatrix}, &  \qquad a&\mapsto
  \begin{bmatrix}
    a^\sigma & 0 \\ 0 & a
  \end{bmatrix},&\qquad \forall &a\in \Z_{p^2} ;\\
 \varphi_3:&\qquad&  \alpha\mapsto&   \begin{bmatrix}
    \alpha & 0 \\ 0 & \Pi^{-1} \alpha \Pi
                                      \end{bmatrix}, & &&\qquad\forall &\alpha\in \calO_p;\\
 \varphi_4:&\qquad&  \alpha\mapsto&   \begin{bmatrix}
    \Pi^{-1}\alpha \Pi  & 0 \\ 0 & \Pi^{-1} \alpha \Pi  \end{bmatrix},
  & &&\qquad\forall &\alpha\in \calO_p;\\
                       \varphi_5:&\qquad&  \alpha\mapsto&   \begin{bmatrix}
    \alpha & 0 \\ 0 &  \alpha  \end{bmatrix}, & &&\qquad\forall &\alpha\in \calO_p.
\end{alignat*}
\end{thm}

\begin{remark}\label{rem:adm}
  Following \cite[\S1]{Ribet-Bimodules}, we say that an embedding $\varphi: \calO_p\hookrightarrow \Mat_n(\calO_p)$ is \emph{admissible} if  $\varphi(\Pi)\in \Mat_n(\calO_p) \Pi$.  Ribet obtains a complete classification of the $\GL_n(\calO_p)$-conjugacy classes of admissible embeddings $\{\varphi: \calO_p\to \Mat_n(\calO_p)\}$ in \cite[Theorem~1.4]{Ribet-Bimodules} for arbitrary $n>0$. When restricted to the case $n=2$, such conjugacy classes are precisely represented by $\varphi_3, \varphi_4$ and  $\varphi_5$ above.

  Theorem~\ref{thm:5-emb} has been generalized from $n=2$ to arbitrary $n>0$ by the second named author and Xiangning Yang in \cite{xue-yang:bimodules} using a completely different method.
\end{remark}

\subsection{Classification via bimodules}
We first prove that there are exactly five conjugacy classes.
This is done by reformulating the problem in terms of  $\calO_p\otimes_{\Z_p}\calO_p$-module structures.
The argument proceeds in several steps and reduces to  Lemmas~\ref{lem:conj-emb-bimodule}, \ref{lem:e20}, and \ref{lem:mod-str-counting}.

\begin{lemma}\label{lem:conj-emb-bimodule}
  There are bijections between the following three sets:
  \begin{align*}
   \left\{\parbox{4cm}{$\GL_2(\calO_p)$-conjugacy classes of embeddings $\calO_p\hookrightarrow
\Mat_2(\calO_p)$}\right\}
  \longleftrightarrow
\left\{\parbox{3.7cm}{Isomorphism
classes of $(\calO_p, \calO_p)$-bimodule structures on  $\Z_p^8$}\right\}
   \longleftrightarrow
\left\{\parbox{3.7cm}{Isomorphism
classes of $\calO_p \otimes_{\Z_p}\calO_p$-module  structures on  $\Z_p^8$}\right\}.
\end{align*}
\end{lemma}
\begin{proof}
This result is well known,  and it is implicit in the proof of \cite[Theorem~1.4]{Ribet-Bimodules}; see also \cite[Remark~3.2]{Molina-IsraelJ-2012}. We provide a complete proof for the reader's convenience.
  Let $L=\calO_p^{\oplus 2}$
  be the free
  rank $2$ right
  $\calO_p$-module.
  Then
  $\End_{\calO_p}(L)=\Mat_2(\calO_p)$ acts on $L$ from the left.
  Each
  embedding $\varphi: \calO_p\hookrightarrow
\Mat_2(\calO_p)$ thus endows $L$ with a left $\calO_p$-module
structure,
making it an $(\calO_p, \calO_p)$-bimodule  with
 $\rank_{\Z_p}(L)=8$.
 It is immediate that $\GL_2(\calO_p)$-conjugate embeddings give rise to  isomorphic
 bimodule structures.

Conversely, suppose that $L\coloneqq \Z_p^8$ is equipped with an
$(\calO_p, \calO_p)$-bimodule structure.  Since $\calO_p$ is the
unique maximal order in $D_p$,
there exists an isomorphism of right $\calO_p$-modules  $f: L\xrightarrow{\sim}\calO_p^{\oplus 2}$, and any two such isomorphisms differ by an element of $\Aut(\calO_p^{\oplus 2})=\GL_2(\calO_p)$.
 The left $\calO_p$-action on $L$
  induces an embedding
  $\calO_p\hookrightarrow \End_{\mathrm{Mod}\text{-}\calO_p}(L)$, which  gives rise to an embedding $\varphi: \calO_p\hookrightarrow
  \Mat_2(\calO_p)$.
  This construction is well-defined
 up to $\GL_2(\calO_p)$-conjugacy.

  Finally, giving $L$
   an $(\calO_p, \calO_p)$-bimodule structure is equivalent to giving  it a
  left $\calO_p\otimes_{\Z_p}\calO_p^{\opp}$-module structure.
  Since the
  canonical involution $x\mapsto x^t\coloneqq \Tr(x)-x$ on $D_p$ induces an isomorphism $\calO_p^\opp\simeq
  \calO_p$,
this is equivalent to  a left
 $\calO_p\otimes_{\Z_p}\calO_p$-module structure.
\end{proof}
Thus, we study the structure of $L$ as a left $(\calO_p\otimes_{\Z_p}\calO_p)$-module.
For this, it is necessary to understand more precisely the ring structure of $\calO_p\otimes_{\Z_p}\calO_p$.
  Consider the
following subring of $\calO_p\otimes_{\Z_p}\calO_p$:
\begin{equation}\label{eq:e3}
  \calR\coloneqq \calO_p\otimes_{\Z_p}\Z_{p^2}\subset
\calO_p\otimes_{\Z_p}\calO_p=\calO_p\otimes_{\Z_p}(\Z_{p^2} \oplus \Z_{p^2}\Pi).
\end{equation}
 From \cite[Lemma~2.10]{li-xue-yu:unit-gp},   $\calR$ is
$\Z_{p^2}$-isomorphic to an Eichler order of level $p\Z_{p^2}$
 in
$\Mat_2(\Z_{p^2})$, where  $\calR=\calO_p\otimes_{\Z_p}\Z_{p^2}$ is regarded as a $\Z_{p^2}$-algebra through its second factor. An explicit isomorphism is given as
follows
\begin{equation}\label{eq:e4}
\begin{aligned}
  \calR\coloneqq \calO_p\otimes_{\Z_p}\Z_{p^2}\xrightarrow{\sim}& \begin{bmatrix}
  \Z_{p^2} & p\Z_{p^2}\\  \Z_{p^2} & \Z_{p^2}
\end{bmatrix}, &       \Pi\otimes 1 &\mapsto   \begin{bmatrix}
    0 & -p \\ 1 & 0
                                \end{bmatrix},\quad \text{and } \\
 a\otimes 1\mapsto&
  \begin{bmatrix}
    a & 0 \\ 0 & a^\sigma
  \end{bmatrix},&  1\otimes b&\mapsto \begin{bmatrix}
    b & 0 \\ 0 & b
  \end{bmatrix},  \quad \forall a, b\in \Z_{p^2}.
\end{aligned}
\end{equation}
Henceforth we
identify $\calR$ with
$\begin{bmatrix} \Z_{p^2} & p\Z_{p^2}\\ \Z_{p^2} & \Z_{p^2}
 \end{bmatrix}$ via the above isomorphism.
 Applying the unique nontrivial automorphism $\sigma\in \Gal(\Q_{p^2}/\Q_p)$ entry-wise to each
 $\xi=\begin{bmatrix}
    x & py \\ z & w
  \end{bmatrix}\in \calR$ defines a $(\Z_{p^2}, \sigma)$-linear
 automorphism $\xi\mapsto \xi^\sigma\coloneqq \begin{bmatrix}
    x^\sigma & py^\sigma \\ z^\sigma & w^\sigma
  \end{bmatrix}$  of $\calR$.  A priori, there is another $(\Z_{p^2}, \sigma)$-linear automorphism defined by applying $\sigma$ to the second factor of $\calR=\calO_p\otimes_{\Z_p}\Z_{p^2}$, which is also induced from the conjugation by $1\otimes \Pi\in \calO_p\otimes_{\Z_p}\calO_p$. We reserve the notation $\xi\mapsto \xi^\sigma$ for the entry-wise one, since it is more conducive for calculation purposes.  Moreover,
  these two $\sigma$-linear automorphisms of $\calR$ are related by the second equation in \eqref{eq:e7} below, so  the second $(\Z_{p^2}, \sigma)$-linear automorphism of $\calR$ can be expressed in terms of the first.
\begin{lemma}
 Let $\Pi_\calR\coloneqq \Pi\otimes 1 =\begin{bmatrix}
   0 & -p \\ 1 & 0
 \end{bmatrix}\in \calR$, and $\varpi\coloneqq 1\otimes \Pi\in \calO_p\otimes_{\Z_p}\calO_p$.  Then  $\calO_p\otimes_{\Z_p}\calO_p$ is a free  left
$\calR$-module of rank $2$ with basis $\{1, \varpi\}$ subject to multiplication rules
     \begin{equation}
   \label{eq:e7}\varpi^2=-p, \quad\text{and}\quad \varpi \xi
\varpi^{-1}=\Pi_\calR
 \xi^\sigma \Pi_\calR^{-1}, \quad \forall\, \xi\in  \calR.
 \end{equation}
\end{lemma}
\begin{proof}
From \eqref{eq:e3},  we have
\[\calO_p\otimes_{\Z_p}\calO_p=\calR\oplus\calR\varpi, \quad\text{with}\quad \varpi^2=-p.\]
As mentioned above,  conjugation by $\varpi$ induces the $(\Z_{p^2}, \sigma)$-linear
 automorphism of $\calR$ uniquely characterized by the following properties
 \begin{equation}\label{eq:e6}
   \varpi(\alpha\otimes 1)\varpi^{-1}=\alpha\otimes 1,\quad\forall
   \alpha\in \calO_p, \quad \text{and}\quad
   \varpi(1\otimes b)\varpi^{-1}=1\otimes b^\sigma, \quad \forall \,b\in \Z_{p^2}.
 \end{equation}
On the other hand, from the identification rule of $\calR$ with $\begin{bmatrix} \Z_{p^2} & p\Z_{p^2}\\ \Z_{p^2} & \Z_{p^2}
 \end{bmatrix}$ in \eqref{eq:e4}, we have
 \begin{alignat*}{2}
(a\otimes 1)^\sigma&=    \begin{bmatrix}
    a & 0 \\ 0 & a^\sigma
  \end{bmatrix}^\sigma= \begin{bmatrix}
    a^\sigma & 0 \\ 0 & a
  \end{bmatrix}&=a^\sigma\otimes 1, \qquad \forall\, a\in \Z_{p^2},\\
  (\Pi\otimes 1)^\sigma&=\Pi\otimes 1,\qquad (1\otimes b)^\sigma&=1\otimes b^\sigma, \qquad \forall\, b\in \Z_{p^2}.
 \end{alignat*}
It follows immediately from \eqref{eq:6.1} and \eqref{eq:e6} that      $\Pi_\calR
 \xi^\sigma \Pi_\calR^{-1}=\varpi \xi
\varpi^{-1}$ for every  $\xi\in  \calR$.
\end{proof}

From \eqref{eq:e7}, endowing $L=\Z_p^8$ with a left
$(\calO_p\otimes_{\Z_p}\calO_p)$-module structure amounts to equipping it with a left $\calR$-module structure together with a $\Z_p$-linear
map $\rho_\varpi: L\to L$ satisfying
\begin{equation}\label{eq:e8}
\rho_\varpi^2=-p, \quad\text{and} \quad \rho_\varpi (\xi X)=(\Pi_\calR
  \xi^\sigma \Pi_\calR^{-1})\rho_\varpi(X), \quad \forall\, \xi\in \calR, X\in L.
\end{equation}
We begin by analyzing the situation after tensoring with $\Q_p$.
Since
$(\calO_p\otimes_{\Z_p}\calO_p)\otimes_{\Z_p}\Q_p=D_p\otimes_{\Q_p}D_p\simeq
\Mat_4(\Q_p)$, there is, up to isomorphism,  a unique left
$(D_p\otimes_{\Q_p}D_p)$-module of dimension $8$ over $\Q_p$.
This can be
interpreted in terms of the datum $(\calR, L, \rho_\varpi)$ as
follows.  Since $\calR\otimes_{\Z_{p^2}}\Q_{p^2}=\Mat_2(\Q_{p^2})$ and
$\rank_{\Z_p}(L)=8$, once $L$ is equipped with a left $\calR$-module
structure, then $L\otimes_{\Z_p}\Q_p$ is a free left
$\Mat_2(\Q_{p^2})$-module of rank one. Henceforth, we fix an identification  $L\otimes_{\Z_p}\Q_p$ with
$V\coloneqq \Mat_2(\Q_{p^2})$ and regard $L$ as a full left $\calR$-lattice inside $V$.
Let $\rho_\varpi: L\to L$ be a $\Z_p$-linear
map satisfying \eqref{eq:e8}, and extend it $\Q_p$-linearly to $V$.
Consider the map
  \[\vartheta:  V\to   V, \qquad X\mapsto
    \Pi_\calR^{-1}\rho_\varpi(X)^\sigma, \]
where $\sigma$ is applied entry-wise to each element $X\in V=\Mat_{2}(\Q_{p^2})$. From \eqref{eq:e8},
$\vartheta$ is $\Mat_2(\Q_{p^2})$-linear, so there exists
$C'\in \End_{\Mat_2(\Q_{p^2})}(V)^\opp=\Mat_2(\Q_{p^2})$ such
that $\vartheta(X)=XC'$ for all $X\in  V$.  From this we
deduce that
\begin{equation}\label{eq:e35}
 \rho_\varpi(X)=\Pi_\calR X^\sigma C, \qquad \forall X\in  V\quad\text{with} \quad C\coloneqq
(C')^\sigma.
\end{equation}
Now it follows from the facts  $\rho_\varpi^2=-p$ and
$\Pi_\calR^2=-p$ that $C^\sigma C=I_2$,  where $I_2\in \Mat_2(\Q_{p^2})$ denotes
the identity matrix. In particular, $C$ determines a
1-cocycle of $\Gal(\Q_{p^2}/\Q_p)$ valued in
$\GL_2(\Q_{p^2})$.
 Had we taken a different identification of $L\otimes_{\Z_p}\Q_p$ with $V=\Mat_2(\Q_{p^2})$, it would differ from the current one by an element of $\Aut_{\Mat_2(\Q_{p^2})}(V)^\opp=\GL_2(\Q_{p^2})$, and correspondingly $C$ would be replaced by an element cohomologous to $C$.
Conversely, given $C\in
\GL_2(\Q_{p^2})$ with $C^\sigma C=I_2$, we define a
$\Q_p$-linear map $\rho_\varpi^C$ satisfying \eqref{eq:e8} as follows:
\[\rho_\varpi^C: V\to V, \qquad X \mapsto \Pi_\calR X^\sigma C. \]
It is straightforward to check by Hilbert's Theorem 90
\cite[Lemma~2.2, \S2.2]{Platonov-Rapinchuk} that up to
isomorphism,  the left $(D_p\otimes_{\Q_p} D_p)$-module structure on
$V$ induced by $\rho_\varpi^C$ is independent of the
choice of $C$. The above argument not only reaffirms the uniqueness of
the left $(D_p\otimes_{\Q_p}D_p)$-module structure on $\Q_p^8$ up to
isomorphism, but also explicitly describes it in terms of the
identity $D_p\otimes_{\Q_p}
D_p=\Mat_2(\Q_{p^2})+\Mat_2(\Q_{p^2})\varpi$.

We return to the integral setting.
Since $\rho_\varpi$ is a $\Z_p$-linear endomorphism of $L$, we must have $\rho_\varpi(L)\subseteq L$, that is, $\Pi_\calR X^\sigma C\in L$ for all $X\in L$. Thus,  given an $\calR$-lattice $L$ in $V$, we define
\begin{equation}
  \label{eq:e10}
  \scrS(L)\coloneqq \{C\in \GL_2(\Q_{p^2})\mid C^\sigma C=I_2, \quad \Pi_\calR X^\sigma C\in L, \forall X\in L\}.
\end{equation}
  From the above discussion, each $\Z_p$-linear
map $\rho_\varpi: L\to L$ satisfying \eqref{eq:e8} uniquely determines  an element $C\in \scrS(L)$,  and vice versa.
In other words, extending the
 $\calR$-module structure on $L$ to an
 $(\calR \oplus \calR \varpi)$-module structure amounts to choosing
 $C\in \scrS(L)$ and defining the
 $\Z_p$-linear map $\rho_\varpi^C: X\mapsto \Pi_\calR X^\sigma C$.

In light of Lemma~\ref{lem:conj-emb-bimodule}, we need to determine which elements of $\scrS(L)$ define isomorphic $(\calR \oplus \calR \varpi)$-module structures on $L$.
Let $\grO_r(L)$ denote the right order of $L$, namely,
\[\grO_r(L)\coloneqq \End_{\calR}(L)^\opp=\{T\in \Mat_2(\Q_{p^2})\mid LT\subseteq L\}.\]
The unit group $\grO_r(L)^\times$ acts on $\scrS(L)$  by  $\sigma$-conjugation
\begin{equation}\label{eqn:Or-grp-act}
    \grO_r(L)^\times\times \scrS(L)\to \scrS(L), \qquad (T, C)\mapsto  T^\sigma  C T^{-1}.
\end{equation}
Indeed, it is straightforward to check that $(T^\sigma  C T^{-1})^\sigma (T^\sigma  C T^{-1})=I_2$.
Moreover, for any $X\in L$, we have $\Pi_\calR X^\sigma (T^\sigma  C T^{-1})=(\Pi_\calR (XT)^\sigma C) T^{-1}\in L T^{-1}=L$.
We write
\[\ol{\scrS(L)}\coloneqq \grO_r(L)^\times\backslash \scrS(L)\]
for the orbit space of the group action in \eqref{eqn:Or-grp-act}.
\begin{lemma}\label{lem:e20}
There is a bijection between the following sets:
\begin{equation}
  \label{eq:e20}
 \left\{\parbox{6.4cm}{Isomorphism
classes of $(\calO_p \otimes_{\Z_p}\calO_p)$-module structures on $\Z_p^8$}\right\} \quad
\longleftrightarrow\quad \bigsqcup_{[L]}\, \ol{\scrS(L)}
\end{equation}
where the disjoint union is taken over  isomorphism classes of $\calR$-lattices $L$ in $V$.
\end{lemma}
\begin{proof}
For each $i\in \{1,2\}$, let $(L_i, \rho_\varpi^{C_i})$ be an ordered pair consisting of an $\calR$-lattice $L_i\subseteq V$ together with a $\Z_p$-linear map $\rho_\varpi^{C_i}: L_i\to L_i$ for some $C_i\in \scrS(L_i)$.  As discussed above, each $L_i$ is equipped with an $(\calR \oplus \calR\varpi)$-module structure via $\rho_\varpi^{C_i}$. A $\Z_p$-linear map $f: L_1\to L_2$ defines an $(\calR \oplus \calR\varpi)$-module isomorphism if and only if:
\begin{enumerate}[(i)]
\item $f$ is an isomorphism of $\calR$-lattices; and
\item  $\rho_\varpi^{C_2}\circ f=f\circ \rho_\varpi^{C_1}$.
\end{enumerate}
Since $\End_{\Mat_2(\Q_{p^2})}(V)^\opp=\Mat_2(\Q_{p^2})$, which acts on $V$
by right multiplication, an $\calR$-linear isomorphism $f: L_1\to L_2$ must be of the form $f_T:
X\mapsto XT$ for some  $T\in \GL_2(\Q_{p^2})$ such that
$L_1T=L_2$. Now condition (ii) above translates to $  C_1=T^\sigma C_2T^{-1}$.
In particular, for a fixed $\calR$-lattice $L\subseteq V$, two elements
$C_1, C_2\in \scrS(L)$ give rise to isomorphic
$(\calR \oplus \calR\varpi)$-module structures on $L$ if and only if
  $C_1=T^\sigma C_2T^{-1}$
for some   $T\in  \grO_r(L)^\times$.
\end{proof}

Recall that $\calR=
\begin{bmatrix} \Z_{p^2} & p\Z_{p^2}\\ \Z_{p^2} & \Z_{p^2}
 \end{bmatrix}$ is an Eichler order of level $p\Z_{p^2}$ in
 $\Mat_2(\Q_{p^2})$, so every $\calR$-lattice in $V$ is isomorphic to
 one of the following mutually non-isomorphic $\calR$-lattices:
 \begin{equation}\label{eq:e21}
 L_1\coloneqq \calR=\begin{bmatrix} \Z_{p^2} & p\Z_{p^2}\\ \Z_{p^2} & \Z_{p^2}
 \end{bmatrix},\quad   L_2\coloneqq \Mat_2(\Z_{p^2})=\begin{bmatrix} \Z_{p^2} & \Z_{p^2}\\ \Z_{p^2} & \Z_{p^2}
 \end{bmatrix},\quad  L_3\coloneqq \Pi_\calR L_2=\begin{bmatrix} p\Z_{p^2} & p\Z_{p^2}\\ \Z_{p^2} & \Z_{p^2}
 \end{bmatrix}.
 \end{equation}
\begin{lemma}\label{lem:mod-str-counting}
    We have
    $\lvert \overline{\scrS(L_1)} \rvert = 3$  and $\lvert \overline{\scrS(L_2)} \rvert
    =
     \lvert \overline{\scrS(L_3)} \rvert=1$.
\end{lemma}
From Lemmas \ref{lem:conj-emb-bimodule}, \ref{lem:e20},  and \ref{lem:mod-str-counting}, the total number of $\GL_2(\calO_p)$-conjugacy classes of embeddings $\{ \calO_p\hookrightarrow
\Mat_2(\calO_p)\}$ is $3+1+1=5$.
\begin{proof}
We treat the easier cases $i=2,3$ first and keep such an $i$
fixed.
Clearly, $\grO_r(L_i)=\Mat_2(\Z_{p^2})$.  On the other hand,
$L_i$ is stable under $\sigma$, and a direct computation shows that
\[\Pi_\calR L_i C\subseteq L_i\quad \Longleftrightarrow \quad C\in \Mat_2(\Z_{p^2}). \]
It follows that $\scrS(L_i)=\{C\in \GL_2(\Z_{p^2})\mid C^\sigma C=I_2\}$, which is
 identifiable with the set of 1-cocycles of
$\Gal(\Q_{p^2}/\Q_p)$ valued in $\GL_2(\Z_{p^2})$. Therefore,
\begin{equation}
  \label{eq:e12}
\forall i\in \{2, 3\}, \qquad  \ol{\scrS(L_i)} \simeq  H^1(\Gal(\Q_{p^2}/\Q_p),
\GL_2(\Z_{p^2}))=\{1\},
\end{equation}
where the last equality  follows from Lemma \ref{lem:h1} below.
The elements in
  $\ol{\scrS(L_2)}$ and $\ol{\scrS(L_3)}$ are  both represented by the
  identity matrix $I_2$.

Next, we treat the case $i=1$.
In this case, $L_1$ is a free left
$\calR$-module of rank one,  so
$\grO_r(L_1)=\calR$. A direct computation shows that
\[\Pi_\calR L_1 C\subseteq L_1\quad \Longleftrightarrow \quad
  C\in \begin{bmatrix}\Z_{p^2}& \Z_{p^2}\\\frac{1}{p}\Z_{p^2}& \Z_{p^2} \end{bmatrix}. \]
Given such a $C$, if we further require that $\det(C)\in
\Z_{p^2}^\times$, then $C$ belongs to either $\GL_2(\Z_{p^2})$ or $\Pi_\calR
\GL_2(\Z_{p^2}) \Pi_\calR^{-1}=\begin{bmatrix} \Z_{p^2} & p \Z_{p^2} \\
                           \frac{1}{p}\Z_{p^2} &
                                                 \Z_{p^2}\end{bmatrix}^\times$. We conclude that
\begin{equation}
  \label{eq:e13}
  \scrS(L_1)=\{C\in \GL_2(\Q_{p^2})\mid C^\sigma C=I_2, \quad C\in
  \GL_2(\Z_{p^2})\cup \Pi_\calR\GL_2(\Z_{p^2}) \Pi_\calR^{-1}\}.
\end{equation}
Let us put $\scrS^\dagger(L_1)\coloneqq   \scrS(L_1)\cap
\GL_2(\Z_{p^2})$ and $\scrS^\ddagger(L_1)\coloneqq   \scrS(L_1)\cap
(\Pi_\calR\GL_2(\Z_{p^2})\Pi_\calR^{-1})$. Then
\begin{equation}
  \label{eq:e14}
  \scrS(L_1)=  \scrS^\dagger(L_1)\cup \scrS^\ddagger(L_1), \quad
  \text{and}\quad \scrS^\dagger(L_1)\cap
  \scrS^\ddagger(L_1)=\scrS(L_1)\cap \calR^\times,
\end{equation}
where the last equality follows from the fact that
$\Mat_2(\Z_{p^2})\cap (\Pi_\calR \Mat_2(\Z_{p^2})\Pi_\calR^{-1})=\calR$. Note that
both $\scrS^\dagger(L_1)$ and $\scrS^\ddagger(L_1)$ are stable under $\sigma$-conjugation by
$\calR^\times$, and conjugation by $\Pi_\calR$ induces a bijection
$\scrS^\dagger(L_1)\to \scrS^\ddagger(L_1)$. Since $\Pi_\calR$ is
$\sigma$-invariant and normalizes  $\calR^\times$, this bijection $\scrS^\dagger(L_1)\to
\scrS^\ddagger(L_1)$ descends to a bijection
between the orbit spaces $\ol{\scrS^\dagger(L_1)}\coloneqq \calR^\times\backslash
\scrS^\dagger(L_1)$ and $\ol{\scrS^\ddagger(L_1)}\coloneqq \calR^\times\backslash
\scrS^\ddagger(L_1)$. Thus
\begin{equation}
  \label{eq:e15}
  |\ol{\scrS(L_1)}|=2
  |\ol{\scrS^\dagger(L_1)}|-|\calR^\times\backslash (\scrS(L_1)\cap
  \calR^\times)|.
  \end{equation}
By definition and Lemma \ref{lem:h1} below,  we have
\[ \calR^\times\backslash (\scrS(L_1)\cap
  \calR^\times)\simeq  H^1(\Gal(\Q_{p^2}/\Q_p), \calR^\times)=\{1\}.\]

It remains to compute $|\overline{\scrS^\dagger(L_1)}|$.
Since $H^1(\Gal(\Q_{p^2}/\Q_p),
\GL_2(\Z_{p^2}))=\{1\}$ by  Lemma \ref{lem:h1} below,
there is a surjective map
\begin{equation}
  \label{eq:e17}
  \GL_2(\Z_{p^2})\twoheadrightarrow \scrS^\dagger(L_1), \qquad
  X\mapsto X^\sigma X^{-1}.
\end{equation}
Moreover, two elements of $\GL_2(\Z_{p^2})$ have the same image if and
only if they differ on the right  by   an element of $\GL_2(\Z_p)$.
Thus, $\scrS^\dagger(L_1)\simeq  \GL_2(\Z_{p^2})/\GL_2(\Z_p)$, and hence
\begin{equation}
  \label{eq:e18} \ol{\scrS^\dagger(L_1)}\simeq
  \calR^\times\backslash\GL_2(\Z_{p^2})/\GL_2(\Z_p).
\end{equation}
Given $x\in \Z_{p^2}$,  we write $\bar{x}$ for its canonical image in
$\F_{p^2}=\Z_{p^2}/(p)$.
Then there is a
$\GL_2(\Z_p)$-equivariant bijection as follows
\[\calR^\times\backslash\GL_2(\Z_{p^2})\xrightarrow{\sim} \bbP^1(\F_{p^2}), \qquad \begin{bmatrix}x & y\\
                                                 z&
                                                    w\end{bmatrix}\mapsto
                                                  [\bar{x}:
                                                  \bar{y}]\in
                                                  \bbP^1(\F_{p^2}), \]
where each $\begin{bmatrix}a & b\\
                                                 c&
                                                    d\end{bmatrix}\in
                                                  \GL_2(\Z_p)$  acts
                                                  on
                                                  $\bbP^1(\F_{p^2})$
                                                  by sending  $[\bar{x}:
                                                  \bar{y}]$ to
                                                  $[\bar{a}\bar{x}+\bar{c}\bar{y}:
                                                  \bar{b}\bar{x}+\bar{d}\bar{y}]$. Under
                                                  this action,
                                                  $\bbP^1(\F_{p^2})$
                                                  separates into two
                                                  $\
                                                  \GL_2(\Z_p)$-orbits by  \cite[Exercise~6, \S II.2.1]{Serre-trees},
                                                  namely,
                                                  \[
                                                    \bbP^1(\F_{p^2})=\bbP^1(\F_p)\bigsqcup
                                                    \left(\bbP^1(\F_{p^2})\smallsetminus
                                                      \bbP^1(\F_p)\right),  \]
which then implies that
                                                  \begin{equation}
                                                    \label{eq:e19}
                                                    |\ol{\scrS^\dagger(L_1)}|=
  |\calR^\times\backslash\GL_2(\Z_{p^2})/\GL_2(\Z_p)|=|\bbP^1(\F_{p^2})/\GL_2(\Z_p)|=2.
                                                  \end{equation}
                                                  Combining this with
                                                  \eqref{eq:e15} and
                                                  \eqref{eq:e16}, we
                                                  obtain
                                                  $|\ol{\scrS(L_1)}|=
                                                  2\times 2-1=3$.
                                                  Note that
                                                  $\gamma\coloneqq \begin{bmatrix}
                                                    0 & 1\\ 1&
                                                    0\end{bmatrix}$ is
                                                  an element of
                                                  $\scrS^\dagger(L_1)$
                                                not belonging to
                                                $\calR^\times$. We
                                                find that
                                                $\ol{\scrS^\dagger(L_1)}$
                                                is represented by
                                                $\{\gamma, I_2\}$, and
                                                hence
                                                $\ol{\scrS(L_1)}$ is
                                                represented by
                                                $\{\gamma, I_2, \Pi_\calR
                                                \gamma \Pi_\calR^{-1}\}$.
\end{proof}

\begin{lemma}\label{lem:h1} Let $G \coloneqq \Gal(\Q_{p^2}/\Q_p)$.
Then we have
\begin{equation}
    \label{eq:e16}
    H^1(G, \calR^\times)=H^1(G,
\GL_2(\Z_{p^2}))=\{1\}.
  \end{equation}
\end{lemma}
\begin{proof}   The second equality follows from a general form of Hilbert's
Theorem 90  \cite[Lemma~III.4.10]{Milne-etale} (together with \cite[Theorem~11.4]{milneLEC}) since $\Z_{p^2}$ is the ring of integers of the
unramified extension $\Q_{p^2}/\Q_p$.
Let $\calK(p)\coloneqq 1+p\Mat_2(\Z_{p^2})$ be  the kernel
of the reduction map  $\GL_2(\Z_{p^2})\to
\GL_2(\F_{p^2})$.
Since the $G$-invariant subgroup
$\GL_2(\Z_p)$ surjects  onto $\GL_2(\F_p)$, the associated long exact sequence in cohomology shows that the natural map
$H^1(G, \calK(p))\to H^1(G,
\GL_2(\Z_{p^2}))$ is injective.
Therefore,
$H^1(G, \calK(p))=\{1\}$.
Next, observe that  $\calK(p)$ is a
normal subgroup of  $\calR^\times$, and the quotient is given by the Borel subgroup $\scrB\coloneqq
\begin{bmatrix}
  \F_{p^2}^\times & 0 \\ \F_{p^2} & \F_{p^2}^\times
\end{bmatrix} \subseteq \GL_2(\F_{p^2}).
$
Thus we have a short  exact sequence $1\to \calK(p)\to \calR^\times\to
\scrB\to 1$, which induces the exact sequence
\[ \{1\}=H^1(G, \calK(p))\to H^1(G,
  \calR^\times)\to H^1(\Gal(\F_{p^2}/\F_p), \scrB).\]
By  Lang's Theorem
\cite[Theorem~6.1]{Platonov-Rapinchuk}, we have $H^1(\Gal(\F_{p^2}/\F_p),
\scrB)=\{1\}$.
It follows from the exact sequence
above that $H^1(G, \calR^\times)=\{1\}$, which proves the first equality.
\end{proof}

\subsection{Distinguishing conjugacy classes via \dieu modules}
We prove the second part of Theorem \ref{thm:5-emb}.
More precisely, we show that the embeddings $\varphi_j$ ($1\leq j \leq 5$) represent  distinct $\GL_2(\calO_p)$-conjugacy classes by comparing the $\calO_p$-\dieu module structures they induce on a superspecial \dieu
module.

Let $(M, \varphi)$ be an $\calO_p$-\dieu module over a perfect field $k_0 \supseteq \F_{p^2}$.
The  $\calO_p$-action on the Lie algebra $\Lie(M)$ clearly factors through $\calO_p/p\calO_p$.
By Definition~\ref{defn:O-dieu-mod} and \S\ref{ss:QM}, the \emph{Lie type} of $(M, \varphi)$ is defined to be  the Lie type of the $\Z_{p^2}$-Dieudonn\'e module obtained by restricting $\varphi$ to the subring $\Z_{p^2}\subseteq \calO_p$.
If the  Lie type of $(M, \varphi)$ is $(1, 1)$, that is, $\Lie(M)=\oplus_{i\in \Z/2\Z} \Lie(M)^i$ as in \eqref{eq:Lie-decomp} with $\dim_k\Lie(M)^0=\dim_k\Lie(M)^1=1$, then
we define its critical indices as follows.

\begin{defn}[{\cite[Definition 2.9]{zink:thesis}}]\label{defn:crit-index}
     We  put $\varepsilon \coloneqq \Pi \pmod p$ so that  $\calO_p/p \calO_p=\F_{p^2}[\varepsilon]$ with $\varepsilon^2=0$ and $\varepsilon a =a^p \varepsilon$ for every $a\in \F_{p^2}$.
  Let $(M, \varphi)$ be an $\calO_p$-\dieu module over a perfect field $k_0\supseteq \F_{p^2}$ with Lie type $(1, 1)$.
  We regard
    $\Lie(M)=\oplus_{i\in \Z/2\Z} \Lie(M)^i$ as an $\F_{p^2}[\varepsilon]\otimes_{\Fp} k_0$-module, where $\varepsilon$ maps $\Lie(M)^i$ into $\Lie(M)^{i+1}$.
    An index $i \in \Z/2\Z$ is called  {\it{critical}} if $\varepsilon(\Lie(M)^i)=0$.
 The set of critical indices of $(M, \varphi)$ is  denoted by $S(M,\varphi)$.
\end{defn}
The set $S(M,\varphi)$ is one of  $\{0,1\}$, $\{0\}$ or $\{1\}$.
If $S(M,\varphi)=\{0,1\}$, then  $M$ is superspecial by \cite[Proposition~4.2]{Ribet-Bimodules}.
Indeed, we have $\sfV M=\Pi M$ and hence  $\sfV^2 M=pM$.

Let $M_{1,1}$ be the \dieu module over $\F_{p^2}$ defined by
\begin{align}\label{eq:bbD}
      (M_{1,1}, \mathsf F, \mathsf V)
      \coloneqq
       \left( W(\F_{p^2})^{\oplus 2}, \quad \begin{bmatrix}
        0 & -p
         \\
         1 & 0
       \end{bmatrix}\sigma, \quad
       \begin{bmatrix}
        0 & p
         \\
         -1 & 0
       \end{bmatrix}\sigma^{-1}
       \right).
       \end{align}
     Then $M_{1,1}$ is superspecial.
     Moreover, it is isomorphic to
 the covariant \dieu module  of any supersingular elliptic curve $E$ over $\F_{p^2}$ with $\pi_E=[-p]$.

 We identify $W(\F_{p^2})$ with $\Z_{p^2}$.
 A direct computation shows that
 \begin{align}  \label{eq:bbD2}
  \End_{{\rm DM}}(M_{1,1})  \simeq \left \lbrace \begin{bmatrix}
            a & pb
            \\
            -b^{\sigma} & a^{\sigma}
            \end{bmatrix}
            \, \middle| \,
            a, b \in W(\F_{p^2}) \right\rbrace
        \simeq
            \calO_p, \qquad \begin{bmatrix}
                 a & pb
            \\
            -b^{\sigma} & a^{\sigma}
            \end{bmatrix}\mapsto a+b\Pi.
    \end{align}
Hence
 $\End_{\rm DM}(M_{1,1}^{\oplus 2}) \simeq
 \Mat_2(\calO_p)$.
This shows that giving an $\calO_p$-\dieu  module structure on $M_{1,1}^{\oplus 2}$
is equivalent to giving an embedding $\varphi: \calO_p \hookrightarrow \Mat_2(\calO_p)$. Moreover, two such embeddings induce isomorphic  $\calO_p$-\dieu  module structures on $M_{1, 1}^{\oplus 2}$ if and only if they are $\GL_2(\calO_p)$-conjugate.
\begin{prop}
\label{prop:Lie}
 The Lie types and critical indices (when defined) of
the $\O_p$-\dieu modules $(M_{1,1}^{\oplus 2}, \varphi_j)$ for $1\leq j\leq 5$ are given in  the following table.
 \begin{center}
   \renewcommand{\arraystretch}{1.3}
      \begin{tabular}{|c|c|c|}
      \hline
        Embeddings   & Lie type  & Critical indices \\ \hline
         $\varphi_1$  & (1, 1) & \{1\} \\ \hline
         $\varphi_2$  & (1, 1) & \{0\}\\ \hline
         $\varphi_3$  & (1, 1) & \{0, 1\}\\ \hline
         $\varphi_4$  & (0, 2) & \ \\ \hline
         $\varphi_5$  & (2, 0) & \ \\ \hline
      \end{tabular}
  \end{center}
  In particular, the embeddings $\{\varphi_j\mid 1\leq j\leq 5\}$ belong to distinct $\GL_2(\calO_p)$-conjugacy classes.
\end{prop}
\begin{proof}
We write $M \coloneqq M_{1,1}$ and fix a $\Z_{p^2}$-basis of  $M$ as in \eqref{eq:bbD},  writing elements of $M$ as column vectors.
Then the Verschiebung $\sfV$ acts as
\begin{equation}\label{eq:V}
\sfV \begin{bmatrix}
    \Z_{p^2}
\\
0
\end{bmatrix}=\begin{bmatrix}
    0
    \\
    \Z_{p^2}
\end{bmatrix} \quad {\text{and}}   \quad \sfV \begin{bmatrix}
    0
    \\
    \Z_{p^2}
\end{bmatrix}=
\begin{bmatrix}
    p\Z_{p^2}
    \\
    0
\end{bmatrix}.
\end{equation}
We define an injective homomorphism  $\chi : \Z_{p^2} \to  \End_{{\rm DM}} (M)$
by $\chi(a)=
\begin{bmatrix}
    a & 0
    \\
    0 & a^{\sigma}
\end{bmatrix}$.
 As in \eqref{eq:dec},
the $\Z_{p^2}$-\dieu module $(M, \chi)$  decomposes as     $M=M^0\oplus M^1$ with
 $M^0=\begin{bmatrix}
    \Z_{p^2}
    \\
    0
\end{bmatrix}$ and
$M^{1}= \begin{bmatrix}
    0
    \\
    \Z_{p^2}
\end{bmatrix}$.
This and  \eqref{eq:V} imply that
$\sfV M^0=M^1$ and   $\sfV M^1=pM^0$.
Hence  $\Lie(M)=M/ \sfV M=(M/\sfV M)^0$, and therefore $(M,\chi)$ has Lie type $(1,0)$.
Similarly,  define
$\psi : \Z_{p^2}\to \End_{{\rm DM}}(M)$ by $\psi(a)=\begin{bmatrix}
    a^{\sigma} & 0
    \\
    0 & a
\end{bmatrix}$.  Then $(M, \psi)$ has Lie type $(0, 1)$.

By \eqref{eq:bbD2}, we have  $\calO_p \simeq \End_{\rm DM}(M)$, so that  $\Mat_2(\calO_p)\simeq \End_{\rm  DM}(M^{\oplus 2})$.
Under this identification, the  restrictions to $\Z_{p^2}$ of $\varphi_j$ listed in Theorem \ref{thm:5-emb} are
\begin{equation}\label{eq:rest}
\varphi_1=\varphi_3=\diag(\chi, \psi), \quad \varphi_2=\diag(\psi, \chi),
\quad   \varphi_4=\diag(\psi, \psi),
\quad \varphi_5=\diag(\chi, \chi).\end{equation}
The Lie types of $(M^{\oplus 2}, \varphi_j)$ follow by additivity with respect to direct sums.

It remains to compute the critical indices for $1 \le j \le 3$.
The case $j=3$ is immediate as $\varphi_3(\Pi)=\begin{bmatrix}
    \Pi & 0 \\ 0 & \Pi
\end{bmatrix}$ acts diagonally on $M^{\oplus 2}$.
Since $\Pi^2=-p$, the element $\Pi\in \calO_p=\End_{\mathrm{DM}}(M)$ vanishes on the one-dimensional $\F_{p^2}$-space $\Lie(M)$. Therefore,  $S(M^{\oplus 2}, \varphi_3)=\{0, 1\}$.  Next, consider the case $j=1$.
By \eqref{eq:rest}, the restriction of $\varphi_1$ to $\Z_{p^2}$ is given by $\diag(\chi,\psi)$.
Hence the decomposition
$\Lie(M^{\oplus2})=\Lie(M^{\oplus2})^0\oplus \Lie(M^{\oplus2})^1$ coincides with the canonical one $M^{\oplus 2}/\sfV M^{\oplus 2}=M/\sfV M\oplus M/\sfV M$.
From  $\varphi_1(\Pi)=\begin{bmatrix}
    0 & -p \\ 1 & 0
\end{bmatrix}$, we see that modulo $\sfV M^{\oplus 2}$,
the induced map on $\Lie(M^{\oplus 2})$ satisfies
\[
\Lie(M^{\oplus2})^0 \xrightarrow{\ \simeq\ } \Lie(M^{\oplus2})^1,
\qquad
\Lie(M^{\oplus2})^1 \xrightarrow{\ 0\ } \Lie(M^{\oplus2})^0.
\]
Hence, $S(M^{\oplus 2}, \varphi_1)=\{1\}$.
By a similar argument, $S(M^{\oplus 2}, \varphi_2)=\{0\}$.

As the Lie types and critical indices are preserved under $\mathrm{GL}_2(\mathcal O_p)$-conjugation,
it follows that the embeddings $\varphi_j$ represent distinct $\mathrm{GL}_2(\mathcal O_p)$-conjugacy classes.
\end{proof}

Recall from Proposition~\ref{prop:sp_Can} that every superspecial $O_B$-abelian surface $(A, \iota)$ over  an algebraically closed field  $k$ of characteristic $p$ admits a canonical model over $\F_{p^2}$ with relative Frobenius endomorphism $[-p]$. Thus we can apply the result of Proposition~\ref{prop:Lie} to the study of superspecial $O_B$-abelian surfaces over $k$.

\begin{cor}\label{cor:embed-act-Lie}
    Let $(A, \iota)$ be a superspecial $O_B$-abelian surface over $k$, and fix an identification of $O_D\otimes \Z_p=\calO_p$ as in \eqref{eq:orientation-ident}.  Put $\F_{p^2}[\varepsilon]\coloneqq \calO_p/p\calO_p$ as in Definition~\ref{defn:crit-index}.
Then the embedding $\varphi_{(A, \iota)}: \calO_p\hookrightarrow \Mat_2(\calO_p)$ attached to $(A, \iota)$ is uniquely determined up to $\GL_2(\calO_p)$-conjugacy by the induced action of $\F_{p^2}[\varepsilon]$  on $\Lie(A)$.
\end{cor}

Corollary~\ref{cor:embed-act-Lie} can also be reformulated purely in terms of $(\calO_p, \calO_p)$-bimodules, without reference to superspecial $O_B$-abelian surfaces. The following corollary may be regarded as a generalization of \cite[Theorem~1.3]{Ribet-Bimodules} in the case where the $\Z_p$-rank is 8. It can be proved in exactly the same way as Proposition~\ref{prop:Lie}.
\begin{cor}
Let $L$ and $L'$ be two $(\calO_p, \calO_p)$-bimodules that both are free of rank $8$ over $\Z_p$.  Then $L$ and $L'$ are isomorphic if and only if the $(\F_{p^2}[\varepsilon], \F_{p^2})$-bimodules $L/L\Pi$ and $L'/L'\Pi$ are isomorphic.
\end{cor}
\begin{proof}
    Indeed, let $L\simeq \calO_p^{\oplus 2}$ be a  right $\calO_p$-module that is free of rank $8$ over $\Z_p$. We re-interpret the right $\calO_p$-module structure on $L$ as a Dieudonn\'e module structure over $\F_{p^2}$ in the following way. Consider the Dieudonn\'e ring
 $\grD_{\F_{p^2}}\coloneqq \Z_{p^2}[\sfF, \sfV]$, which is a non-commutative ring  subject to the relations $\sfF\sfV=\sfV\sfF=p$, $\sfF a=a^\sigma\sfF $ and $\sfV a^\sigma= a\sfV$ for all $a\in \Z_{p^2}$.  From the multiplication rule of $\calO_p$ in \eqref{eq:O_p}, it is clear that $\calO_p$ can be realized as a quotient ring of $\grD_{\F_{p^2}}$ as follows:
 \[\grD_{\F_{p^2}}/(\sfF+\sfV, \sfV^2+p) =\Z_{p^2}[\sfV]/(\sfV^2+p) \xrightarrow{\sim}\calO_p=\Z_{p^2}+\Z_{p^2}\Pi,\quad \sfV\mapsto \Pi, \quad a\mapsto a, \quad \forall a\in \Z_{p^2}. \]
Observe that the map $a+b\Pi\mapsto a+ \Pi b$ defines an isomorphism between $\calO_p$ and its opposite ring $\calO_p^\opp$. Therefore, we may rewrite the right $\calO_p$-module structure on $L$ as a left module over $\grD_{\F_{p^2}}/(\sfF+\sfV, \sfV^2+p)$ by defining $(a+\sfV b)x\coloneqq x (a+b\Pi)$ for all $x\in L$ and $a, b\in \Z_{p^2}$.
Thus, $L\simeq M_{1, 1}^{\oplus 2}$ as  Dieudonn\'e modules over $\F_{p^2}$.
Now an $(\calO_p, \calO_p)$-bimodule free over $\Z_p$ is simply a superspecial $\calO_p$-Dieudonn\'e module over $\F_{p^2}$, so the corollary follows directly from Proposition~\ref{prop:Lie}.
\end{proof}

\begin{remark}\label{rem:match-str-cls}
    There remains the  question of which isomorphism class of left
$(\calO_p\otimes_{\Z_p}\calO_p)$-module structure classified in
Lemma~\ref{lem:mod-str-counting} corresponds to which embedding
$\varphi_j$ as listed in Theorem~\ref{thm:5-emb}. Recall that a left
$(\calO_p\otimes_{\Z_p}\calO_p)$-module of $\Z_p$-rank $8$
is determined by an ordered pair $(L, C)$, where $L$ is  a left $\calR$-module
with $\rank_{\Z_p}(L)=8$, and $C$ is an element of $\scrS(L)$
as in \eqref{eq:e10}.
For each $1\leq i\leq 3$, let $L_i$  be the full $\calR$-lattice in $V=\Mat_2(\Q_{p^2})$ given in \eqref{eq:e21}.
Set $\gamma\coloneqq
\begin{bmatrix}
  0 & 1 \\ 1 & 0
\end{bmatrix}$ and $\delta\coloneqq\Pi_\calR \gamma \Pi_\calR^{-1}=\begin{bmatrix}
  0 & p \\ 1/p & 0
\end{bmatrix}$.
The correspondence is as follows:
\begin{align}
\varphi_1&\longleftrightarrow  (L_1, \delta),   \quad &\quad
                                                      \varphi_2&\longleftrightarrow
                                                                 (L_1,
                                                       \gamma),
\quad  & \quad
\varphi_3 &\longleftrightarrow (L_1, I),  \label{eq:e26}\\
\varphi_4&\longleftrightarrow  (L_2, I),   \quad &\quad
                                                 \varphi_5&\longleftrightarrow
                                                            (L_3,
                                                       I).  & & \label{eq:e27}
\end{align}
This  is not used elsewhere in the paper, so we omit the details.
\end{remark}

\section{Adelic descriptions of genera of superspecial $O_B$-abelian surfaces}\label{ss:as}

Keep $B$, $\Delta$, and $O_B$ as in Section \ref{ss:exist}. More explicitly, $B$ is an indefinite quaternion algebra  over $\Q$ with reduced discriminant  $\Delta$, and $O_B$ is a fixed maximal order in $B$. For each prime $\ell$, let  $B_{\ell}\coloneqq B\otimes_\Q \Q_{\ell}$ and $O_{B, \ell} \coloneqq O_B \otimes \Z_{\ell}$.
If $p\mid \Delta$, we further fix an identification of $\calO_p=O_B\otimes\Z_p$ with $O_D\otimes\Z_p$ as in \eqref{eq:orientation-ident}, where $O_D$ denotes a fixed maximal order in $D=D_{p, \infty}$.
In this section, we give an adelic description of each genus of superspecial $O_B$-abelian surfaces, which
sets up the preliminaries for the count the isomorphism classes within each genus.
Following the work of Zink \cite{zink:thesis}, Ribet \cite{Ribet-Bimodules}, and others, we explain the geometric meaning of these genera and their cardinalities.
The detailed calculation for the number of isomorphism classes will be carried out in Section~\ref{ss:count}.

Combining Proposition~\ref{prop:genus} and Theorem~\ref{thm:5-emb}, we see that
\begin{enumerate}
    \item if $p\nmid \Delta$, then there is a unique genus of superspecial $O_B$-abelian surfaces, all of which satisfy the determinant condition by Lemma~\ref{lem:det};
    \item if $p\mid \Delta$, then there are five genera of  superspecial $O_B$-abelian surfaces, corresponding to the five $\GL_2(\calO_p)$-conjugacy classes of embeddings $\calO_p\hookrightarrow \Mat_2(\calO_p)$.
\end{enumerate}
When $p\mid \Delta$, for each $1\leq j\leq 5$ we write $\Sigma_j$ for the set of isomorphism classes of $O_B$-abelian varieties $(A, \iota)$ over $\Fpbar$ such that the embedding $\varphi_{(A, \iota)}: \calO_p\hookrightarrow \Mat_2(\calO_p)$ attached to $(A, \iota)$ is $\GL_2(\calO_p)$-conjugate to $\varphi_j$ listed in Theorem~\ref{thm:5-emb}.

Let $X=X_{O_B}$ be the (coarse) moduli $\Z$-scheme parametrizing $O_B$-abelian surfaces satisfying the determinant condition, and $X^{\sp}$ be the superspecial locus of its special fiber $X_{\F_p}\coloneqq X\otimes_\Z\Fp$ as in Definition~\ref{defn:sp-ss-loci}.  By Lemma \ref{lem:det} (2),  an $O_B$-abelian surface
$(A, \iota)$ over $\Fpbar$  satisfies the determinant condition if and only if its $O_{B, p}$-Dieudonn\'e module $(M(A), \iota_p)$ has Lie type $(1,1)$.  Thus the set of $\Fpbar$-rational points $X^{\sp}(\Fpbar)$ is identified with the set of isomorphism classes of superspecial $O_B$-abelian $\Fpbar$-surfaces of Lie type $(1, 1)$. If $p\nmid \Delta$, such isomorphism classes of $O_B$-abelian surfaces form a single genus as explained above. If $p\mid \Delta$, then $X^{\sp}(\Fpbar)$ is separated into three genera according to Proposition~\ref{prop:Lie}, that is,
\begin{equation}\label{eq:bad}
 X^{\rm sp}(\Fpbar)=X^{\rm sp}(\F_{p^2})\simeq \Sigma_{1} \sqcup \Sigma_{2} \sqcup \Sigma_{3}.
 \end{equation}
It turns out that the remaining two genera $\Sigma_4$ and $\Sigma_5$ parametrize the irreducible components of $X_{\Fpbar}$, as explained below.

\subsection{Bad reduction of the Shimura curves}\label{ss:bad} Suppose that $p\mid \Delta$. From the work of Cerednik \cite{Cerednik-1976} and Drinfeld \cite{Drinfeld}, $X_{\Fpbar}$ is a projective curve whose normalization is a disjoint union of rational curves, and whose only singular points are ordinary double points. Zink \cite{zink:thesis} and Ribet \cite[Theorem 5.3]{Ribet-Bimodules} gave
 moduli interpretations  of the singular locus,  irreducible components, and their incidence relations.
The singular points are represented by superspecial $O_B$-abelian surfaces $(A, \iota)$ with Lie type $(1, 1)$ and critical indices $\{0, 1\}$.
By Proposition \ref{prop:Lie}, the singular locus $X^{\rm sing}$ coincides with $\Sigma_3$.
To classify the irreducible components ${\rm Irr}(X_{\Fpbar})$ of $X_{\Fpbar}$, we consider  a superspecial point $(A, \iota)$ with Lie type $(0,2)$ or $(2, 0)$, which belongs to $\Sigma_4 \sqcup \Sigma_5$.
Then every subgroup $H \subseteq A$  isomorphic to $\alpha_p$ is stable under the action of $O_B$.
Such subgroups are in bijection  with the points of $\mathbb P_1$ by \cite[Proposition~4.5]{Ribet-Bimodules}.
By dividing $(A, \iota)$ by these subgroups $H$, we obtain a family of $O_B$-abelian surfaces with Lie type $(1, 1)$.
This construction yields a bijection ${\rm Irr}(X_{\Fpbar})\simeq \Sigma_4 \sqcup \Sigma_5$.

Finally,  we explain the incidence relation due to Ribet between the superspecial points $X^\sp(\Fpbar)=\Sigma_1\sqcup \Sigma_2\sqcup \Sigma_3$ and the irreducible components using our present terminology.
We treat the singular locus $X^{\mathrm{sing}}=\Sigma_3$ first.
Let $x$ be a point in  $\Sigma_3$ represented by $(A, \iota)$.
From \cite[Proposition~4.4]{Ribet-Bimodules}, $A$ has precisely two $O_B$-stable subgroups that are isomorphic to $\alpha_p$, say $I_0$ and $I_1$.
Then
the two irreducible components of $X_{\Fpbar}$ passing through $x$ correspond to the points in $\Sigma_4 \sqcup \Sigma_5$ represented by $(({A/I_0})^\sigma, \bar{\iota}^\sigma)$ and $(({A/I_1})^\sigma, \bar{\iota}^\sigma)$, with exactly one of them lying in $\Sigma_4$ and the other in $\Sigma_5$ according to \cite[Proposition~4.7]{Ribet-Bimodules}.
Here $\bar{\iota}$ denotes the induced $O_B$-action on  the quotients.
Furthermore, for $i=1, 2$, let $y_i$  be a point in $\Sigma_i$ represented by $(A_i, \iota_i)$.
It is shown in \cite[Proposition~4.4]{Ribet-Bimodules} that $(A_i, \iota_i)$ has exactly one $O_B$-stable subgroup $H_i$ isomorphic to $\alpha_p$.
The unique irreducible component  of $X_{\Fpbar}$ passing through $y_i$ corresponds to $(({A_i/H_i})^\sigma, \bar{\iota}_i^\sigma)$, which belongs to $\Sigma_4$ if $i=1$,  and $\Sigma_5$ if $i=2$ according to \cite[Proposition~4.7]{Ribet-Bimodules} again.
Indeed, since $\dim_k\Lie(H_i)=1$, the action of $\calO_p/p\calO_p$ on $\Lie(H_i)$ factors through the quotient map $\calO_p/p\calO_p\twoheadrightarrow \F_{p^2}$, and the same holds for  the action of $\calO_p/p\calO_p$ on $\Lie({A_i/H_i})^\sigma$ by \cite[Proposition~4.7]{Ribet-Bimodules}.  Moreover, the same proposition shows that $\F_{p^2}$ acts on $\Lie(H_i)$ and $\Lie({A_i/H_i})^\sigma$ through the same embedding $\F_{p^2}\hookrightarrow k$. Therefore, the result $(({A_i/H_i})^\sigma, \bar{\iota}_i^\sigma)\in \Sigma_{i+3}$ for $i=1,2$ follows directly from the table of Proposition~\ref{prop:Lie}.

\subsection{Adelic descriptions of genera}
\label{subsec:adelic-desc}
We return to the general setting, where $p$ may or may not divide $\Delta$.
Let $\Sigma$ be the set of isomorphism classes of superspecial  $O_B$-abelian surfaces over $\Fpbar$ within one genus.  In other words,
\begin{equation}
    \Sigma=\begin{cases*}
         X^\sp(\Fpbar) & if $p\nmid \Delta$,\\
        \Sigma_j  & if $p\mid \Delta$
    \end{cases*}
\end{equation}
 for some $1\leq j\leq 5$.  We provide an adelic description of $\Sigma$ in terms of a double coset space of a linear algebraic group.

Fix $x=[(A, \iota)]\in \Sigma$.
By definition, the \emph{group $\calI$  of self-isogenies of $(A, \iota)$} is the group scheme over $\Z$ representing the functor sending a commutative ring $R$ to $(\End(A, \iota)\otimes R)^\times$. In particular, $\calI(\Z)=O^\times$, where $O$ denotes the endomorphism ring  $\End(A, \iota)$. Let $\calD\coloneqq O\otimes_\Z\Q$ be the endomorphism algebra of $(A, \iota)$,
which coincides with the centralizer of $\iota^0(B)$ in $\End^0(A)=\Mat_2(D_{p, \infty})$.
 The generic fiber $\calI_\Q\coloneqq \calI\otimes \Q$ of $\calI$ can be identified with the linear algebraic $\Q$-group $\underline{\calD}^\times$, i.e.~the multiplicative group of $\calD$.
From the double centralizer theorem \cite[Theorem, \S12.7]{Pierce-asso-alg}, $\calD$ is a definite quaternion $\Q$-algebra with reduced discriminant
\begin{equation}\label{eq:red-disc-centralizer}
    \Delta_{\calD}=\begin{cases*}
        -p\Delta  & if $p\nmid \Delta$;\\
        -\Delta/p & if $p\mid \Delta$.
    \end{cases*}
\end{equation}
  Let $\wbZ\coloneqq \varprojlim_n \Z/n\Z$ be the profinite completion of $\Z$, and $\wbQ\coloneqq \wbZ\otimes \Q$ be the ring of finite adeles of $\Q$.  By definition, we have $\calI(\wbZ)=\whO^\times$ and $\calI(\wbQ)=\whD^\times$, where $\widehat{O}\coloneqq O\otimes \wbZ$  denotes the profinite completion of
    $O$, and   $\whD\coloneqq \whO\otimes_\Z\Q$ denotes the ring of finite adeles of $\calD$.

\begin{prop}\label{prop:adelic-desc}
   There is a bijection
  \begin{equation}\label{eq:unram}
 \Sigma \simeq \calI(\Q)\backslash \calI(\wbQ)/\calI(\wbZ)=\calD^{\times}
  \backslash
    \whD^{\times}/\widehat{O}^{\times},
    \end{equation}
     sending the chosen point  $x\in \Sigma$ to the neutral double coset $\calD^{\times} \cdot 1\cdot \widehat{O}^{\times}$.  In particular, if we write  $\Cl(O)$ for the ideal class set of fractional right $O$-ideal classes, and $h(O)\coloneqq |\Cl(O)|$ for its cardinality (i.e.~the class number of $O$), then
     \begin{equation}\label{eq:genus-cls-no}
    \Sigma\simeq \Cl(O)\qquad \text{and}\qquad      |\Sigma|=h(O).
     \end{equation}
\end{prop}
\begin{proof}
From \cite[Lemma~27.6.8]{voight-quat-book}, the ideal class set $\Cl(O)$ enjoys the same adelic description as the right hand side of  \eqref{eq:unram}, from which \eqref{eq:genus-cls-no} follows.
    The bijection in \eqref{eq:unram} is a variant of the \emph{simple mass formula} \cite[\S2]{yu:smf} for PEL-type abelian varieties, which is also explained in \cite[\S5]{xue-yu:counting-av}.
    For the reader's convenience, we provide a brief sketch by constructing the map   $\whD^{\times}
\to \Sigma$.
 For each prime $\ell$, let $A[\ell^{\infty}]$ denote the  $\ell$-divisible group of $A$  with its $O_{B, \ell}$-action.
 Let $(h_\ell)_\ell \in \whD^\times$.
 We can choose a positive integer $N$  such that
 $f_{\ell}\coloneqq N h_{\ell}^{-1}$ is an actual endomorphism of  $A[\ell^{\infty}]$ for all $\ell$.
 Furthermore, $f_\ell$ is an isomorphism for almost all $\ell$.
Let $g : A \to A' \coloneqq A\Big/\Bigl(\prod_\ell \ker f_\ell\Bigr)$  be the induced isogeny.
Each $\ker f_\ell$ is stable under $O_{B,\ell}$,
so the $O_B$-action on $A$ descends to $A'$ and $g$ is $O_B$-linear.
Since $f_{\ell}$ and $g_{\ell}$ have the same
kernel, there exists an $O_B$-linear isomorphism $\alpha_{\ell} : A[\ell^{\infty}] \to A'[\ell^{\infty}]$ such that $\alpha_{\ell} \circ f_{\ell} =g_{\ell}$.
In particular,   $A'$ is superspecial.
Thus we obtain a point of $\Sigma$.
 The proof that this map induces the bijection \eqref{eq:unram}  reduces to \cite[Theorem~2.2]{yu:smf}.
\end{proof}

To apply Proposition~\ref{prop:adelic-desc}, we need a good understanding of the quaternion order  $O=\End(A, \iota)$, which is the centralizer of $\iota(O_B)$ in $\End(A)$.

\begin{lemma}\label{lem:end-ring-unram}
    If $p\nmid \Delta$, then $O$ is a maximal order in  $\calD$.
\end{lemma}
\begin{proof}
    In this case we have $B_p\simeq \Mat_2(\Q_p)$ and $O_{B, p}\simeq \Mat_2(\Z_p)$. By Corollary \ref{cor:sp_fod}, we may replace $(A, \iota)$ over $\Fpbar$ by  its canonical model over $\F_{p^2}$. To prove the maximality of $O$, it is enough to prove that $O$ is locally maximal everywhere.
    For each prime $\ell$, let  $O_{\ell} = O \otimes \Z_{\ell}$, which is the centralizer of $\iota(O_{B, \ell})$ in $\End(A)_{\Z_{\ell}}$.
It follows from Tate's theorem that $\End(A)_{\Z_{\ell}}
\simeq
\End_{\rm DM}(M(A))$
(resp. $\End_{\Z_{\ell}}(T_{\ell}(A))$) if $\ell=p$ (resp. $\ell \neq p$) where $M(A)$ is the \dieu module over $\F_{p^2}$ (resp. $T_{\ell}(A)$ is the Tate module).
By \eqref{eq:bbD} and \eqref{eq:bbD2},
we have $M(A)\simeq M_{1,1}^{\oplus2}$ and  $\End_{\rm DM}(M_{1,1}^{\oplus2}) \simeq
 \Mat_2(\calO_p)$, where $\calO_p$ denotes the maximal order of the unique division quaternion algebra over $\Q_p$.
Since
$O_{B,p}\simeq\Mat_2(\Z_p)$, the Morita equivalence implies
$O_p\simeq\calO_p$.
When $\ell \neq p$, the Tate module $T_{\ell}(A)$ is a
free $\iota(O_{B, \ell})$-module of rank one.
Hence $O_{\ell}   \simeq O_{B, \ell}$,
which is a maximal order in $B_{\ell}$.
\end{proof}

Next, we consider the case $p\mid \Delta$, where  $B_p$ is the division quaternion $\Q_p$-algebra,  and $O_{B, p}$ is the unique  maximal order $\calO_p$.
 Let $\varphi_j : \calO_p \hookrightarrow \Mat_2(\calO_p)$ be the embeddings listed in Theorem~\ref{thm:5-emb}.
Following \cite[Definition 24.3.2]{voight-quat-book}, a quaternion $\Z_p$-order $\scrO$ is called \emph{residually inert} if $\scrO$ modulo its Jacobson radical is $\F_{p^2}$.

  \begin{prop}\label{prop:end-DM-5cases}
    For $1 \leq j \leq 5$, let $C_j$ denote the centralizer of $\varphi_j(\calO_p)$ in $\Mat_2(\calO_p)$.
    Then $C_j$ is isomorphic to one of the following $\Z_p$-orders  in $\Mat_2(\Q_{p})$:
 \begin{itemize}
     \item[(a)]
     If $j=1$ or $2$,  a residually inert Bass order
 $
\Z_{p^2}+p\Mat_2(\Z_p),$
where $\Z_{p^2}$ is embedded in $\Mat_2(\Q_p)$ via the regular representation of $\Q_{p^2}$ over $\Q_p$;
\item[(b)]
If $j=3$, an Eichler order
      of level $p$;
\item[(c)]
If $j=4$ or  $5$, a    maximal order.
\end{itemize}
\end{prop}
\begin{proof}
The cases $j=3,4,5$ follow  from \cite[Corollary 1.5]{Ribet-Bimodules}.
For $j=1$ and $2$, the argument is identical, so we treat them simultaneously.
We write $\calO_p =\Z_{p^2} \oplus \Z_{p^2}\Pi$ as in \eqref{eq:O_p}.
The image of $\Z_{p^2}$ under $\varphi_j$ is  $\varphi_j(\Z_{p^2})=\left\lbrace
   \begin{bmatrix}
        a & 0
        \\
        0 & a^{\sigma}
    \end{bmatrix}
     \, \middle\vert \,
        a \in \Z_{p^2}
        \right \rbrace \subseteq \Mat_2(\calO_p).$
        Its centralizer in $\Mat_2(\calO_p)$ is $\begin{bmatrix}
        \Z_{p^2} & \Z_{p^2}\Pi
        \\
        \Z_{p^2}\Pi & \Z_{p^2}
    \end{bmatrix}$.
  Let  $\begin{bmatrix}
        x & y
        \\
        z & w
    \end{bmatrix}$ be an element of this ring and put $\theta\coloneqq \begin{bmatrix}
        1 & 0\\ 0 & \Pi
    \end{bmatrix}$.
    Commuting with $\varphi_j(\Pi)=\begin{bmatrix}
        0 & -p
        \\
        1 & 0
    \end{bmatrix}$ forces  $x=w$ and $y=-pz$, which yields
    \[C_j= \left \lbrace \begin{bmatrix}
        a & -pb\Pi
        \\
        b\Pi & a
    \end{bmatrix} \, \middle\vert \,  a, b  \in \Z_{p^2} \right\rbrace \xrightarrow[X\mapsto \theta X\theta^{-1}]{\simeq}
    \left\lbrace \begin{bmatrix}
        a & 0
        \\
         0 & a^{\sigma}
    \end{bmatrix}
    -p
    \begin{bmatrix}
        b & 0
        \\
         0 & b^{\sigma}
    \end{bmatrix}
    \begin{bmatrix}
        0 & 1
        \\
         1 & 0
    \end{bmatrix}
    \, \middle\vert \, a, b \in \Z_{p^2}\right \rbrace. \]
Thus,  $C_j\otimes_{\Z_p} \Q_p$ is the split cyclic algebra $(\Q_{p^2}/\Q_p, \sigma, 1) \simeq \Mat_2(\Q_{p})$, where the $\Z_p$-order  $\left\lbrace \begin{bmatrix}
        a & b
        \\
         b^\sigma & a^{\sigma}
    \end{bmatrix}\middle\vert\, a, b \in \Z_{p^2} \right\rbrace
\subseteq (\Q_{p^2}/\Q_p, \sigma, 1)$
is identified with the maximal order $\Mat_2(\Z_p)\subseteq \Mat_2(\Q_p)$.
Furthermore,  $C_j$ can be identified with the preimage of $\F_{p^2}$ under the reduction map $\Mat_2(\Z_p)\to\Mat_2(\F_{p})$, for a fixed embedding $\F_{p^2} \hookrightarrow \Mat_2(\F_p)$.
Therefore  $C_j$ is a residually inert $\Z_p$-order in $\Mat_2(\Q_p)$, and by \cite[Proposition 3.1]{Brzezinski-1983} it is a Bass order.
\end{proof}
We refer the interested reader to the works of Brzezinski \cite{Brzezinski-1983, Brzezinski-crelle-1990} or Voight \cite[Chapter~24]{voight-quat-book} for detailed expositions on quaternion Bass orders.  See also \cite[\S37]{curtis-reiner:1} for  Bass orders in more general settings.

\begin{cor}\label{cor:ram}
 Assume $p \mid \Delta$,  and let   $[(A_j, \iota_j)] \in \Sigma_{j}$ for $1 \leq j \leq 5$.
Then the centralizer $O_j \coloneqq \End(A_j, \iota_j)$  of  $\iota(O_B)$ in $\End(A_j)$ is
\begin{itemize}
    \item[(a)] a Bass order of reduced discriminant $p\Delta$ that is residually inert at $p$ if $j = 1$ or $2$,
    \item[(b)] an Eichler order of level $p$ if $j = 3$, or
    \item[(c)] a maximal order if $j = 4$ or $5$.
\end{itemize}
 \end{cor}

 \begin{proof}
 The proof is identical to that of Lemma~\ref{lem:end-ring-unram}, except that
  the centralizer  $O_j \otimes \Z_p$ of $\iota(O_{B, p})$ in $\End_{\rm DM}(M(A_j))$ is isomorphic to $C_j$
 described in  Proposition \ref{prop:end-DM-5cases}.
\end{proof}

\subsection{Galois action}
Let  $\Gamma \coloneqq \Gal(\Fpbar/\F_p)$.
For each $\phi \in \Gamma$,  the $\phi$-twist $(A, \iota)\mapsto ({A}^\phi, {\iota}^\phi)$ induces a natural $\Gamma$-action on $X(\Fpbar)$.
The superspecial locus $X^{\rm sp} (\Fpbar)=X^{\rm sp}(\F_{p^2})$ is $\Gamma$-stable by \cite[Proposition 3.1]{Yu-Indiana-ArithSSlocus}, and $X^{\rm sp} (\Fpbar)^{\Gamma}=X^{\rm sp}(\Fp)$.
Each point in $X(\F_p)$ is represented by a model $(A, \iota)$ over $\F_p$ according to Lemma \ref{lm:fomfod}.
\begin{lemma}\label{lem:twist}
Assume that $p \mid \Delta$.

(1) Let  $\sigma=\sigma_p \in \Gamma$ be  the Frobenius element.
Then
the $\sigma$-twist $(A, \iota) \mapsto  ({A}^{\sigma}, {\iota}^{\sigma})$  induces bijections $\Sigma_{ 1} \xrightarrow{\sim} \Sigma_{2}$,  $\Sigma_{ 3}\xrightarrow{\sim} \Sigma_{ 3}$, and  $\Sigma_{ 4}\xrightarrow{\sim}\Sigma_{5}$, as well as their inverses.

(2) We have
$X(\F_p)=X^{\rm sp}(\F_p)=\Sigma_{3}^\Gamma$.
 In particular,   $\Sigma_j\cap X(\F_p)$ is empty for $j \in \{1,2,4,5\}$.
 \end{lemma}
 \begin{proof}
 (1)   The $\sigma$-twist swaps the $\tau_0$ and $\tau_1$-components of $M(A)$, as there is a canonical isomorphism $M({A}^{\sigma})\simeq W(\F_{p^2})\otimes _{W(\F_{p^2}), \sigma} M(A)$. Thus, it shifts the corresponding Lie types and critical indices by $1 \pmod 2$, which induces the above bijections.

 (2)
  It suffices to prove that  $X(\F_p)\subseteq \Sigma_{3}$.
 Let
 $x \in X(\F_p)$ be a point represented by a  model   $(A, \iota)$ over $\F_p$.
 By Lemma~\ref{lem:odd-degree-bdred} concerning $O_B$-abelian surfaces over finite fields of odd degree, the action of $\Pi\in \calO_p$ vanishes on $\Lie(A)$.
 Definition \ref{defn:crit-index} then implies that $(M(A\otimes_{\Fp}\Fpbar), \iota_p)$ has Lie type $(1, 1)$ and critical index $S(M(A\otimes_{\Fp}\Fpbar), \iota_p)=\{0, 1\}$.
 By Proposition \ref{prop:Lie}, this yields $x\in \Sigma_{3}$.
 Alternatively, this inclusion can be deduced from  Proposition~\ref{prop:K_odd}, Equation~\eqref{eq:bad}, and statement (1).
 \end{proof}

We return to the general setting with no restriction on the divisibility of $\Delta$ by $p$. Fix a superspecial  $O_B$-abelian surface $(A, \iota)/\F_p$, and  write $O\coloneqq \End_{\Fpbar}(A_{\Fpbar}, \iota)$ and $\calD\coloneqq O \otimes \Q$ as before.
For $\phi \in \Gamma$ and  $f \in O$, the $\phi$-twist     ${f}^\phi : ({A}^\phi_{\Fpbar}, {\iota}^\phi)  \to   ({A}^\phi_{\Fpbar}, {\iota}^\phi)$ again belongs to $O$ because $(A_{\Fpbar}, \iota)\simeq ({A}^\phi_{\Fpbar}, {\iota}^\phi)$.
Extending  to quasi-isogenies of the $\ell$-divisible groups  for all primes $\ell$, we obtain the natural  $\Gamma$-action  on $\whD^{\times}$.

Let $\pi_{A} \in  O$  denote the Frobenius endomorphism of $A$ over $\F_p$.
By Theorem~\ref{thm:Fq-pts}, we have
 \begin{equation}\label{eq:pi}
 \pi_A=\begin{cases*}
 \sqrt{p}\zeta_4
      \text{ or }
      \pm \sqrt{p}\zeta_{4p} &  if  $p=2, 3$, \\
     \sqrt{p}\zeta_4 & otherwise,
     \end{cases*}
     \end{equation}
     and further
 no prime $\ell \mid \Delta$ splits in $\Q(\pi_A)$.

\begin{prop}\label{prop:ssp-Fp-pts}
 Suppose  $X^{\rm sp}(\F_p) \neq \emptyset$, and let $x \in X^{\rm sp}(\F_p)$ be represented by a model  $(A, \iota)$ over $\F_p$. Write $\Sigma\coloneqq X^{\rm sp}(\Fpbar)$ if $p\nmid \Delta$ and $\Sigma\coloneqq \Sigma_3$ if $p\mid \Delta$ so that $[(A, \iota)\otimes\Fpbar]\in \Sigma$.
 Then the bijection in \eqref{eq:unram}  is  $\Gamma$-equivariant.
 Moreover, the Frobenius  $\sigma_p \in \Gamma$ acts on $\calD^\times\backslash
    \whD^\times/\whO^\times$ by   $[g] \mapsto [g\pi_A]$ for $g \in \whD^{\times}$.
 In particular,
     \begin{equation*}
    X^{\mathrm{sp}}(\F_p)\simeq (\calD^\times\backslash
    \whD^\times/\whO^\times)^{\pi_A}=\{[g]\in (\calD^\times\backslash
  \whD^\times/\whO^\times)\mid [g\pi_A]=[g]\}.
\end{equation*}
\end{prop}
\begin{proof}
The proof follows the same strategy as
 \cite[Theorem 1.1]{Yu-Indiana-ArithSSlocus}, where the case of principally polarized superspecial abelian varieties was treated.
In particular,  each quasi-isogeny $g_{\ell}\in \calD_{\ell}$ satisfies
$\sigma_p(g_{\ell})=\pi_A g_{\ell} \pi_A^{-1}$.
 We also have $\pi_A\in \calD^{\times}$, and in the completion, $\pi_A^2=pu$ for some  $u \in \whO^{\times}$ by \eqref{eq:pi}.
 Hence  for $g \in \whD^\times$,
\[\sigma_p([g])
=[\pi_A g\pi_A^{-1}]=[g (p^{-1}u^{-1}\pi_A)]=[p^{-1}g(\pi_A u^{-1})]=[g \pi_A].
\qedhere\]
\end{proof}

\section{Counting the superspecial points}\label{ss:count}

Keep the notation of Section~\ref{ss:as}. We explicitly determine the cardinalities of $X^\sp(\Fpbar)$ and $X^\sp(\Fp)$.
    In the bad reduction case where $p\mid \Delta$, we compute the cardinality of each $\Sigma_j$ for $1\leq j\leq 5$, which allows us to determine the numbers of irreducible components and singular points of the special fiber $X_{\Fpbar}$.

Let $d<0$ be a negative discriminant, and denote by
 $S(d)\coloneqq \Z[(d+\sqrt{d})/2] \subseteq \Q(\sqrt{d})$   the  quadratic order  of discriminant $d$.
 Write $d=f^2 d'$, where $d'$ is the fundamental discriminant of $\Q(\sqrt{d})$,  and $f \coloneqq [S(d'):S(d)]$ is
 the conductor of $S(d)$.
 Let $\left( \frac{\cdot}{\ell}\right)$ denote  the Kronecker symbol.
The class number  $h(d)=h(S(d))$ is  expressed in terms of the class number $h(d')$ of the maximal order $S(d')$ by Dedekind's formula \cite[\S III.5, p.~95]{vigneras}:
\begin{equation}\label{eq:dedekind}
  h(d)=\frac{h(d')f}{[S(d')^\times: S(d)^\times]}\prod_{\ell\mid f}\left(1-\frac{1}{\ell}\left(\frac{d'}{\ell}\right)\right).
\end{equation}
For each  prime  $\ell$, we  define the Eichler symbol \cite[p.~94]{vigneras} by
\begin{equation}\label{eq:eich} \left( \frac{S(d)}{\ell}\right)\coloneqq
\begin{dcases*}
    1 & if $\ell \mid f$;
    \\
    \left( \frac{d'}{\ell}\right) & otherwise.
\end{dcases*}
\end{equation}
For an integer $N$,
 we define the  modified Hurwitz class number
\begin{equation}\label{eq:mod-hurwitz-cln}
h_{N}(d) \coloneqq \frac{h(d)}{[S(d)^{\times}: \Z^{\times}]}
\prod_{ \ell \mid N}
\left(1- \left( \frac{S(d)}{\ell} \right) \right).
\end{equation}
If $d$ is not a discriminant, we set  $h_N(d)\coloneqq 0$.

The following theorem gives explicit formulas
for the numbers of superspecial points.
\begin{thm}\label{thm:main}
Put $\Delta_\calD\coloneqq -p\Delta$ if $p\nmid \Delta$, and $\Delta_\calD\coloneqq -\Delta/p$ if $p\mid \Delta$ as in \eqref{eq:red-disc-centralizer}.
We have the following formulas.
\begin{enumerate}[label=(\arabic*), align=right,   leftmargin=*]
    \item Assume that $p \nmid \Delta$.
    Then
    \[ \lvert X^{\rm sp}(\Fpbar) \rvert=
    \frac{1}{12}
    \prod_{\ell \mid \Delta_\calD}(\ell-1)+\frac{1}{4}\prod_{\ell \mid \Delta_\calD}
    \left( 1-\left( \frac{-4}{\ell}\right)\right)+\frac{1}{3}\prod_{\ell \mid \Delta_\calD}\left(
    1-
    \left(
    \frac{-3}{\ell}\right)\right).
    \]
\item
Assume that $p \mid \Delta$.
 Then
\begin{equation*}
 \lvert \Sigma_{ j} \rvert
 =
\begin{dcases*}
\frac{p-1}{12}\prod_{\ell \mid \Delta_{\calD}}(\ell-1)+ \frac{1}{2}\delta_4\prod_{\ell \mid \Delta_{\calD}}\left( 1- \left(\frac{-4}{\ell}\right)\right)
+
\frac{2}{3}
\delta_3
\prod_{\ell \mid \Delta_{\calD}}
\left( 1- \left( \frac{-3}{\ell}\right) \right)
& if $j=1, 2$;
\\
\begin{aligned}
\frac{p+1}{12}\prod_{\ell \mid \Delta_{\calD}}(\ell-1)
 &  +
\frac{1}{4}\left( 1+ \left(\frac{-4}{p}\right)\right)
\prod_{\ell \mid \Delta_{\calD}}
\left(1- \left( \frac{-4}{\ell}\right)\right) \\
   & + \frac{1}{3}\left(1+\left( \frac{-3}{p}\right)\right)
\prod_{\ell \mid \Delta_{\calD}}
\left( 1- \left(\frac{-3}{\ell}\right)\right)
\end{aligned}
& if
$j=3$;
\\
\frac{1}{12}\prod_{\ell \mid \Delta_{\calD}}(\ell-1)+
\frac{1}{4}\prod_{\ell \mid \Delta_{\calD}}\left( 1-\left( \frac{-4}{\ell}\right)\right) +
\frac{1}{3}\prod_{\ell \mid \Delta_{\calD}}
\left( 1-\left(\frac{-3}{\ell}\right)\right) &
if $j=4, 5$,
\end{dcases*}
\end{equation*}
where $\delta_4$ and $\delta_3$ are defined by
\[\delta_4=\begin{cases*}
    1 & if $p\equiv 3\pmod{4}$,\\
    0 & otherwise,
\end{cases*}\qquad\text{and}\qquad  \delta_3=\begin{cases*}
    1 & if $p\equiv 2\pmod{3}$,\\
    0 & otherwise.
\end{cases*} \]
\item Put $\Delta_0\coloneqq \Delta$ if $p\nmid \Delta$, and $\Delta_0\coloneqq \Delta/p$ if $p\mid \Delta$.
Then
\[ \lvert X^{\rm sp}(\F_p) \rvert =
    \begin{dcases*}
    \frac{1}{2}\prod_{\ell \mid \Delta_0}\left(1-\left(\frac{-2}{\ell}\right)\right)+\frac{1}{2}\prod_{\ell \mid \Delta_0}\left(1-\left(\frac{-1}{\ell}\right)\right)& if $p=2$;\\
    \prod_{\ell \mid \Delta_0} \left( 1-\left( \frac{-3}{\ell} \right)\right) &
    if $p=3$ and $2 \nmid \Delta_0$;
    \\
    \frac{1}{2}\prod_{\ell \mid \Delta_0} \left( 1-\left( \frac{-3}{\ell} \right)\right) &
    if $p=3$ and $2 \mid \Delta_0$;
    \\
       \frac{1}{2} \Big(h_{\Delta_0}(-p)+h_{\Delta_0}(-4p)\Big) & if $p >3$.
    \end{dcases*}\]
\end{enumerate}
\end{thm}
The above formulas yield several geometric quantities for the bad reduction of the Shimura curve.
Suppose $p \mid \Delta$. First, recall from \eqref{eq:bad} that
$X^{\rm sp} (\Fpbar)=\Sigma_1\sqcup \Sigma_2 \sqcup \Sigma_3$, so the number  $\lvert X^{\rm sp} (\Fpbar) \rvert$ of  geometric superspecial points is obtained by summing up the corresponding formulas for $\Sigma_i$ with $1\leq i\leq 3$ in  part (2).
Next, part (3) and Proposition \ref{prop:K_odd} yield the number $\lvert X(\F_p) \rvert$ of
$\F_p$-rational points on
$X$.
Moreover, part (2) together with \S\ref{ss:bad} provides explicit formulas for the  numbers of singular points $X^{\rm sing}= \Sigma_3$ and irreducible components ${\rm Irr}(X_{\Fpbar}) \simeq \Sigma_4 \sqcup \Sigma_5$.

\begin{proof}[Proof of part (1) and (2) of Theorem \ref{thm:main}]

Fix a point $[(A, \iota)] \in X^{\rm sp}(\Fpbar)$ if $p \nmid \Delta$ (resp. $[(A, \iota)] \in \Sigma_{j}$ if $p\mid \Delta$).
By Proposition~\ref{prop:adelic-desc}, the set  $X^{\rm sp}(\Fpbar)$ (resp.~$\Sigma_{ j}$) is in bijection with
  the class set $\Cl(O)$ (\cite[p.~87]{vigneras}) of the quaternion  order $O=\End(A, \iota)$. From \eqref{eq:red-disc-centralizer}, $\calD\coloneqq O\otimes\Q$ is a definite quaternion $\Q$-algebra of reduced discriminant $\Delta_\calD$.
 The formulas are obtained by computing
 $h(O)\coloneqq \lvert \Cl(O)\rvert$
using the Eichler class number formula.

Let $K/\Q$ be an imaginary quadratic extension, and let $S$ be a $\Z$-order in $K$.
For each prime $\ell$,
let $\Emb(S_{\ell}, O_{\ell})$ denote  the set of optimal embeddings $\varphi : S_{\ell}\hookrightarrow O_{\ell}$, i.e. $\Z_{\ell}$-algebra  embeddings satisfying   $\varphi(S_{\ell})=O_{\ell} \cap\varphi(K_{\ell})$.
The group $O_{\ell}^{\times}$ acts from the right on  $\Emb(S_{\ell}, O_{\ell})$ by conjugation.
We define the local embedding number by
 \[m_\ell(S_\ell, O_\ell, O_\ell^\times) \coloneqq |\Emb(S_\ell, O_\ell)/O_\ell^\times|.\]
 For brevity, we simply write $m_\ell(S)\coloneqq m_\ell(S_\ell, O_\ell, O_\ell^\times)$ as $O$ will always be clear from the context.
 Assume that $O_{\ell}$ is a hereditary order in $\calD_{\ell}= O\otimes \Q_{\ell}$.
 Then, using the notation of \eqref{eq:eich}, the local embedding number is given by
  \cite[(3.1)]{Brzezinski-crelle-1990} (see also \cite[Theorem~3.2 and \S II.3]{vigneras}):
 \begin{equation}\label{eq:emb}
     m_\ell(S)
     =
     \begin{dcases*}
         1 & if $\calD_{\ell}$ is split and $O_{\ell}$ is a maximal order,
         \\
         1+\left( \frac{S}{\ell}\right) & if $\calD_{\ell}$ is split and $O_{\ell}$ is an Eichler order of level $\ell$,
         \\
         1-\left( \frac{S}{\ell}\right) & if $\calD_{\ell}$ is ramified.
         \end{dcases*}
     \end{equation}

The only imaginary quadratic orders $S(d)$ with $[S(d)^\times: \Z^\times]>1$ are those with $d\in \{-4, -3\}$. The orders $S(-4)=\Z[\sqrt{-1}]$ and $S(-3)=\Z[\zeta_6]$ are the rings of integers of
the two cyclotomic quadratic fields $\Q(\sqrt{-1})$ and $\Q(\sqrt{-3})$ respectively, and both  have class number one.
  By the Eichler class number formula  \cite[Main Theorem 30.8.6]{voight-quat-book}, we obtain
 \begin{equation}\label{eq:h(O)}
 h(O)=\sum_{[I] \in {\Cl(O)}}
 \frac{1}{\lvert \grO_{{l}}(I)^{\times}/\{ \pm 1\} \rvert}
 +\frac{1}{4}\prod_{\ell}m_{\ell}(S(-4))+
 \frac{1}{3}\prod_{\ell}m_{\ell}(S(-3)),\end{equation}
where $\grO_{l}(I)\coloneqq \{x\in \calD\mid xI\subseteq I\}$ denotes the left order of $I$, and the products are taken over all prime numbers $\ell$.
Moreover, by the Eichler mass formula
(
\cite[Main Theorem 25.3.19]{voight-quat-book}),
we have
\begin{equation}\label{eq:mass}
\sum_{[I] \in {\Cl(O)}}
 \frac{1}{\lvert \grO_{{l}}(I)^{\times}/\{ \pm 1\} \rvert}=\frac{| \Delta_O |}{12}\prod_{\ell \mid \Delta_O} \frac{1-{\ell}^{-2}}{1-e_{\ell}(O_{\ell})\ell^{-1}},
 \end{equation}
where $\Delta_O$ denotes the reduced discriminant of $O$, and $e_{\ell}(O_{\ell})$ denotes the Eichler invariant of $O_\ell$ \cite[\S 1]{Brzezinski-crelle-1990}.
By Lemma~\ref{lem:end-ring-unram}, Proposition~\ref{prop:end-DM-5cases}, and Corollary~\ref{cor:ram}, the local orders $O_{\ell}$ are classified as follows.
 If one of the following conditions holds: (i) $p\nmid \Delta$ and $\ell \nmid p\Delta$; (ii) $p \mid \Delta$ and $\ell \nmid \Delta$; (iii) $p\mid \Delta$, $j\in \{4, 5\}$, and  $\ell=p$, then $\calD_{\ell}$ is split and $O_{\ell}$ is maximal.
In this case,  $\ell \nmid \Delta_O$.
 If either $p \nmid \Delta$ and $\ell \mid p\Delta$, or $p \mid \Delta$ and $\ell \mid \Delta_{\calD}$, then $\calD_{\ell}$ is ramified and $O_{\ell}$ is maximal.
In this case,  $e_{\ell}(O_{\ell})=-1$.
 If $p \mid \Delta$ and  $j\in \{1, 2\}$, then $\calD_{p}$ is split and $O_p$ is a residually inert Bass order.
In this case, $e_p(O_p)=-1$, and by \cite[(3.1)]{Brzezinski-crelle-1990},
\begin{equation}\label{eq:embBass}
m_{p}(S(d))= \begin{dcases*}
    2 & if $p$ is inert in $\Q(\sqrt{d})$,
    \\
    0 & if $p$ is split or ramified.
    \end{dcases*}
    \end{equation}
 If $p \mid \Delta$ and $j=3$, then $\calD_p$ is split and $O_p$ is an Eichler order of level $p$.
In this case, $e_p(O_p)=1$.
Finally, combining the above classification with formulas \eqref{eq:emb} and \eqref{eq:embBass}
and substituting into \eqref{eq:h(O)} and \eqref{eq:mass}, we obtain the formulas in the theorem.
\end{proof}

Now we compute $\lvert X^{\rm sp}(\F_p) \rvert$.
Let $W_p^{\mathrm{ss}}(1)$ be the set of
$\Gal(\overline{\Q}/\Q)$-conjugacy classes of supersingular Weil
$p$-numbers $\pi$ with $d(\pi)=1$.
By
Proposition~\ref{prop:summ-Weil-num}, the set $W_p^{\mathrm{ss}}(1)$
can be identified with $\{ \pm \sqrt{p}\zeta_{4p}, \sqrt{-p} \}$ if
$p=2 $ or $3$, and with $\{\sqrt{-p}\}$ otherwise.
Equivalently, sending
each conjugacy class of Weil numbers to its  minimal polynomial over
$\Z$ establishes a bijection between  $W_p^{\mathrm{ss}}(1)$
and the set of quadratic polynomials
\begin{equation}
  \label{eq:z1}
 \scrP\coloneqq  \{T^2-tT+p\in \Z[T]\mid t\equiv
  0\pmod{p}\quad\text{and}\quad t^2-4p<0\}.
\end{equation}

By Propositions~\ref{prop:move-to-sp} and~\ref{prop:ssQMFq}, the set $X(\F_p)$ is
nonempty if and only if there exists $\pi\in W_p^{\mathrm{ss}}(1)$ such
that $B$ is split by $\Q(\pi)$; see  Theorem~\ref{thm:Fq-pts} for an explicit criterion.

Assume $X(\F_p)\neq \emptyset$.
Let $(A, \iota)$  be an
$O_B$-abelian surface over $\F_p$, and $\pi_A$ be the Frobenius
endomorphism of $A$ over $\F_p$.
By Lemma~\ref{lem:odd-degree-bdred}, $(A, \iota)$
 automatically represents a
point in $X(\F_p)$.  Let $\calI$ be the $\Z$-group scheme of
self-isogenies of $(A_{\Fpbar}, \iota)$ as in \S\ref{subsec:adelic-desc}, which  represents the functor sending a commutative ring $R$ to $(\End_{\Fpbar}(A_{\Fpbar}, \iota)\otimes R)^\times$.
As usual, we put $O\coloneqq \End_{\Fpbar}(A_{\Fpbar}, \iota)$, and $\calD\coloneqq O\otimes_\Z\Q$.
If $p\nmid \Delta$, then $O$ is a maximal order in $\calD$ by Lemma~\ref{lem:end-ring-unram}.
If $p\mid \Delta$, then $O$ is an Eichler order of reduced discriminant $\Delta_O=\Delta$ by Corollary~\ref{cor:ram} and Lemma \ref{lem:twist} (2).
More precisely,
$O_\ell$ is a maximal order in
$\calD_\ell$ for every prime $\ell\neq p$, and at the prime $p$ there
exists an identification of $\calD_p \simeq \Mat_2(\Q_p)$ such that $O_p=
\begin{bmatrix}
  \Z_p & p\Z_p\\ \Z_p & \Z_p
\end{bmatrix}$   by Proposition~\ref{prop:end-DM-5cases}.
Let $\grJ(O_p)$ denote the Jacobson radical of $O_p$.
If $p\nmid \Delta$, then $O_p$ is the maximal order in the quaternion division algebra $\calD_p$, so  $\grJ(O_p)$ is its unique maximal ideal.
If $p\mid \Delta$, then  $\grJ(O_p)=\begin{bmatrix}
  p\Z_p & p\Z_p\\ \Z_p & p\Z_p
\end{bmatrix}=O_p\begin{bmatrix}
    0 & p \\ 1 & 0
\end{bmatrix}$.
From the description of the minimal polynomial of the Frobenius endomorphism
$\pi_A\in O$ in \eqref{eq:z1}, we see that in both cases $\pi_A\in \grJ(O_p)$.
In fact,
\begin{equation}
  \label{eq:z2}
 \pi_AO_p=\grJ(O_p)\quad\text{since}\quad  [O_p:\pi_AO_p]=\Nr(\pi_A)^2=p^2=[O_p:\grJ(O_p)].
\end{equation}
For each prime $\ell\neq p$, we have $\pi_A\in
O_\ell^\times$ since its reduced norm $\Nr(\pi_A)=p$.

Let $\whO=O\otimes \wbZ$ be the profinite completion of $O$, and set
$\whD\coloneqq \whO\otimes_\Z\Q$, the ring of finite adeles of $\calD$.
 The
finite idele group $\whD^\times=\calI(\wbQ)$ is a locally compact
totally disconnected topological group equipped with a unimodular Haar
measure $dx$, which is normalized so that the open compact subgroup
$U\coloneqq \calI(\wbZ)=\whO^\times$ has volume $1$.
Since $\calD$ is
definite, the group $\calD^\times=\calI(\Q)$ is discrete and cocompact in
$\whD^\times$, so we equip it with the counting measure.
Let $L^2(\calD^\times\backslash \whD^\times/U)$ denote  the
space of functions $\psi: \whD^\times \to \bbC$ satisfying $\psi(\gamma x u)=\psi(x)$ for all $\gamma\in \calD^\times$ and
$u\in U$.
By definition,  $L^2(\calD^\times\backslash \whD^\times/U)$ can
be identified with the space of $\bbC$-valued functions on the finite
set $\calD^\times\backslash \whD^\times/U$, and hence
\begin{equation}\label{eq:L^2}
   \dim_\bbC L^2(\calD^\times\backslash
  \whD^\times/U)=| \calD^\times\backslash \whD^\times/\whO^\times |=h(O).
\end{equation}
By
\cite[Exercise~I.4.6]{vigneras}, $\pi_A$ normalizes\footnote{Geometrically, $O\coloneqq \End_{\Fpbar}(A\otimes \Fpbar, \iota)$ is stable under the $\Gal(\Fpbar/\Fp)$-action, so it is  normalized by $\pi_A$.} $U=\whO^\times$ since $\pi_A O$ is a two-sided ideal of $O$. We form a
Hecke operator $R(\pi_A)$ on $L^2(\calD^\times\backslash \whD^\times/U)$
by convolution with the characteristic function
$\mathbbm{1}_{U\pi_A}:\whD^\times\to \bbC$ of the double coset
$U\pi_A U=\pi_A U=U\pi_A$.
More explicitly,
\[  (R(\pi_A)\psi)(y)=\int_{\whD^\times}
  \mathbbm{1}_{U\pi_A }(x)\psi(yx)dx,  \qquad \forall\, \psi \in
  L^2(\calD^\times\backslash \whD^\times/U),\quad \forall\, y \in
  \whD^\times. \]
It is shown in \cite[\S8.1]{Yu-Indiana-ArithSSlocus}  that
$(R(\pi_A)\psi)(y)=\psi(y\pi_A)$, so the trace of $R(\pi_A)$ equals the number of fixed points of the right multiplicative action of $\pi_A$ on $\calD^\times\backslash
    \whD^\times/U$:
    \begin{equation}
      \label{eq:z3}
          \Tr(R(\pi_A))=|(\calD^\times\backslash
    \whD^\times/U)^{\pi_A}|=\#\{[y]\in (\calD^\times\backslash
  \whD^\times/U)\mid [y\pi_A]=[y]\}.
    \end{equation}
Therefore, by Proposition~\ref{prop:ssp-Fp-pts},
\begin{equation}\label{eq:z17}
    |X^{\mathrm{sp}}(\Fp)|=|(\calD^\times\backslash
    \whD^\times/\whO^\times)^{\pi_A}|=\Tr(R(\pi_A)).
\end{equation}
On the other hand, the formation of the Hecke operator $R(\pi_A)$ is purely arithmetic.
Let $\varpi=(\varpi_\ell)_\ell\in  \whD^\times$ be an element having similar properties to $\pi_A$ in the sense that
\begin{equation}\label{eq:z13}
    \varpi_p O_p=\grJ(O_p), \qquad\text{and}\qquad \varpi_\ell\in O_\ell^\times, \quad \forall \ell\neq p.
\end{equation}
Then we have
\begin{equation}\label{eq:z6}
 U\varpi=(O_p^\times\varpi_p)\times \prod_{\ell\neq p} O_\ell^\times= U\pi_A, \quad \text{where}\quad  O_p^\times\varpi_p=\{x_p\in \grJ(O_p)\mid \Nr(x_p)\in p\Z_p^\times\}.
\end{equation}
In particular, the two operators $R(\varpi)$ and $R(\pi_A)$ coincide, and
the value of $\Tr(R(\varpi))$ depends only on the ordered pair $(p, \Delta)$ and does not depend on the particular choice of $\varpi$ satisfying \eqref{eq:z13}.  Alternatively, the coset $O_p^\times \varpi_p$ can also be described by
\begin{equation}\label{eq:z14}
 O_p^\times\varpi_p=\{x_p\in O_p\mid \Nr(x_p)\in p\Z_p^\times \quad \text{and}\quad \Tr(x_p)\equiv 0\pmod{p}\}.
\end{equation}
Indeed, if $p\mid \Delta$, then this follows directly from  $\grJ(O_p)=\begin{bmatrix}
  p\Z_p & p\Z_p\\ \Z_p & p\Z_p
\end{bmatrix}$.
If $p\nmid \Delta$, then $O_p$ is the unique maximal order in the division quaternion $\Q_p$-algebra $\calD_p$, so every element $x_p\in O_p$ with $\Nr(x_p)\in p\Z_p^\times$ automatically lies in $\grJ(O_p)$ and satisfies  $\Tr(x_p)\equiv 0\pmod{p}$.

To compute $\Tr(R(\varpi))$, we apply the \emph{Selberg trace formula} for compact quotients,  as explained in \cite[\S8.2]{Yu-Indiana-ArithSSlocus} or \cite[\S1]{Arthur-Intro-trace-formula-Clay4-2005}.
For each $\gamma\in \calD^\times$, let $K_\gamma$ be the centralizer of $\gamma$ in $\calD$.
Let $\{\gamma\}$ denote the $\calD^\times$-conjugacy class of $\gamma\in \calD^\times$, and let $\{\calD^\times\}$ be the set of all $\calD^\times$-conjugacy classes in $\calD^\times$.
Applying the Selberg trace formula to $f=\mathbbm{1}_{U\varpi}$, we obtain
\begin{equation}\label{eq:z4}
    \Tr(R(\varpi))=\sum_{\{\gamma\}\in \{\calD^\times\}} \int_{K_\gamma^\times\backslash \whD^\times} \mathbbm{1}_{U\varpi}(x^{-1}\gamma x)\,dx.
\end{equation}
Here $K_\gamma^\times$ is equipped with the counting measure, and the homogeneous space $K_\gamma^\times\backslash \whD^\times$ is equipped with the induced right $\whD^\times$-invariant measure \cite[Corollary~4, \S III.4]{Nachbin-Haar-Integral},  which is still denoted by $dx$ by an abuse of notation.
For example, if $V$ is an  open compact subgroup of $\whD^\times$, then by \cite[(2.4.4)]{Kottwitz:ClayMathProc}  we have
\begin{equation}\label{eq:z9}
\int_{K_\gamma^\times\backslash \whD^\times} \mathbbm{1}_{gV}(x) \,dx=\frac{1}{|K_\gamma^\times\cap gVg^{-1}|}\int_{\whD^\times} \mathbbm{1}_{V}(x)\,dx, \qquad \forall g\in \whD^\times.
\end{equation}

To evaluate the orbital integrals in \eqref{eq:z4}, we introduce
some additional notation.
For each polynomial $P(T)$ in the set $\scrP$ in \eqref{eq:z1}, we write $K_P\coloneqq \Q[T]/(P(T))$ for the imaginary quadratic field defined by $P(T)$, and $S_P\coloneqq \Z[T]/(P(T))$ for the quadratic $\Z$-order in $K_P$ generated by the image of $T$.
\begin{prop}
For any element $\varpi\in \whD^\times$ satisfying the conditions in \eqref{eq:z13}, we have
    \begin{equation}\label{eq:z15}
     \Tr(R(\varpi))=\sum_{P\in \scrP}\,\sum_{S_P\subseteq S\subseteq K_P}\frac{h(S)}{|S^\times|}\prod_{\ell}m_\ell(S),
    \end{equation}
    where the second sum runs over all $\Z$-orders $S$ in $K_P$ containing $S_P$, and the product runs over all primes $\ell$ (including $\ell=p$).
\end{prop}

\begin{proof}
For each $\gamma\in \calD^\times$, consider  the left $K_\gamma^\times$-invariant set
\begin{equation}\label{eqn:z5}
\wh\calE(\gamma)\coloneqq \{x\in \whD^\times\mid x^{-1}\gamma x\in U\varpi\}.
\end{equation}
Since $\wh\calE(\gamma)$ is an open subset of $\whD^\times$, the orbital integral  $\int_{K_\gamma^\times\backslash \whD^\times} \mathbbm{1}_{U\varpi}(x^{-1}\gamma x)dx$ is nonzero if and only if $\wh\calE(\gamma)\neq \emptyset$.
From the description of $U\varpi$ in \eqref{eq:z6}, if $\wh\calE(\gamma)\neq \emptyset$, then necessarily $\Nr(\gamma)=p$ and hence $\gamma\not\in \Q^\times$, which further implies that the centralizer $K_\gamma$ of $\gamma$ in $\calD$ coincides with $\Q(\gamma)$.
Furthermore,
 $\wh\calE(\gamma)\neq \emptyset$ also implies that $\Tr(\gamma)\equiv 0\pmod{p}$ by \eqref{eq:z14}, so the minimal polynomial $\Min_\gamma(T)\in \Z[T]$ of $\gamma$ belongs to $\scrP$.
 Thus, we can simplify \eqref{eq:z4} into
\begin{equation}\label{eq:z11}
   \Tr(R(\varpi))=\sum_{\substack{\{\gamma\}\in \{\calD^\times\}\\\Min_\gamma(T)\in \scrP}} \int_{K_\gamma^\times\backslash \whD^\times} \mathbbm{1}_{U\varpi}(x^{-1}\gamma x)dx,
\end{equation}
which is a finite sum,  since each conjugacy class $\{\gamma\}$ above is uniquely determined by its minimal polynomial $\Min_\gamma(T)\in \scrP$ by  the Skolem-Noether theorem.

Now fix  $\gamma\in \calD^\times$  with minimal polynomial $P(T)\in \scrP$.
By \eqref{eq:z13} and \eqref{eq:z14}, the condition
$x^{-1}\gamma x\in U\varpi$ is equivalent to $x^{-1}\gamma x\in\whO$.
Hence we may rewrite $\wh\calE(\gamma)$ as
\begin{equation}\label{eq:z7}
 \wh\calE(\gamma)= \{x\in \whD^\times\mid x^{-1}\gamma x\in \whO\}.
\end{equation}
We identify $K_P=\Q[T]/(P(T))$ with $\Q(\gamma)$  by sending $T+(P(T))$ to  $\gamma$, so the centralizer $K_\gamma$ is identified with $K_P$ as well.
 For each $x\in \wh\calE(\gamma)$, the intersection $K_P\cap x\whO x^{-1}$ defines a $\Z$-order  in $K_P$ containing $S_P$.
Grouping together those $x$ giving rise to the same order in $K_P$, we obtain  a  partition $\wh\calE(\gamma)=\bigsqcup_{S_P\subseteq S \subseteq K_P}\wh\calE(S) $ of $\wh\calE(\gamma)$  into a disjoint union of $K_P^\times$-invariant open subsets,
where
\begin{equation}\label{eq:z12}
   \wh\calE(S)\coloneqq \{x\in \whD^\times\mid K_P\cap x \whO x^{-1}=S\}.
\end{equation}  Correspondingly, the orbital integral decomposes as
\begin{equation}\label{eq:z10}
     \int_{K_\gamma^\times\backslash \whD^\times} \mathbbm{1}_{U\varpi}(x^{-1}\gamma x)dx= \sum_{S_P\subseteq S\subseteq K_P} \int_{K_P^\times\backslash \whD^\times} \mathbbm{1}_{\wh\calE(S)}(x)dx.
\end{equation}
Note that each $\wh\calE(S)$ is also right invariant under $\whO^\times$.
In fact, it is shown in the proof of the classical \emph{trace formula for optimal embeddings} \cite[Theorem~III.5.11]{vigneras} \cite[Theorem~30.4.7]{voight-quat-book} that
\begin{equation}\label{eq:z8}
  N_S\coloneqq  |K_P^\times\backslash \wh\calE(S)/\whO^\times|=h(S)\prod_{\ell}m_\ell(S).
\end{equation}
If we decompose $\wh\calE(S)$ as a disjoint union of double cosets
$\wh\calE(S)=\bigsqcup_{i=1}^{N_S} K_P^\times g_i \whO^\times$, then
\[\begin{split}
  \int_{K_P^\times\backslash \whD^\times} \mathbbm{1}_{\wh\calE(S)}(x)dx
    &=\sum_{i=1}^{N_S} \int_{K_P^\times\backslash \whD^\times} \mathbbm{1}_{K_P^\times g_i \whO^\times}(x)dx \\&\xeq{\eqref{eq:z9}}  \sum_{i=1}^{N_S}\frac{1}{|K_P^\times\cap g_i\whO^\times g_i^{-1}|}\int_{\whD^\times} \mathbbm{1}_{\whO^\times}(x)dx \xeq{\eqref{eq:z12}}\sum_{i=1}^{N_S}\frac{1}{|S^\times|}\\&=\frac{N_S}{|S^\times|}\xeq{\eqref{eq:z8}}\frac{h(S)}{|S^\times|}\prod_{\ell}m_\ell(S).
\end{split}\]
Combining    \eqref{eq:z11} and \eqref{eq:z10} with the above computation, we obtain \eqref{eq:z15}.
\end{proof}
\begin{remark}
    Consider the open subset $U(p)\coloneqq \{x\in \whO \mid \Nr(x)\in p\Z_p^\times\}$ of $\whD^\times$, which is stable under left and right multiplication by $U$.
    Note that if $p\mid \Delta$, then $U\varpi$ is a proper subset of $U(p)$.
However, if $p\nmid \Delta$, then $U(p)=U\varpi$.
    We  form another Hecke operator $R(\mathbbm{1}_{U(p)})$ on  $L^2(\calD^\times\backslash \whD^\times/U)$
by convolution with the characteristic function
$\mathbbm{1}_{U(p)}:\whD^\times\to \bbC$ of $U(p)$.
It is shown in \cite[Theorem~4.2.1]{xue-yang-yu:ECNF} that the matrix of $R(\mathbbm{1}_{U(p)})$ with respect to a suitable basis of $L^2(\calD^\times\backslash \whD^\times/U)$ coincides with the Brandt matrix $T(p)$ as defined in \cite[\S41.2]{voight-quat-book}.
Hence its trace can be computed using the Eichler trace formula given in \cite[Proposition~V.2.4]{vigneras} or \cite[Main Theorem~41.5.2]{voight-quat-book}.
Therefore, when $p\nmid \Delta$,  the trace  $\Tr(R(\varpi))$ can be computed directly from the Eichler trace formula \cite[Theorem~41.1.13]{voight-quat-book}.
See also \cite[(3.18)]{xue-yu:spinor-class-no} for a refinement of the Eichler trace formula.
\end{remark}

\begin{prop}\label{prop:trRw}
    If $p \mid \Delta$, we have
    \[ |X(\Fp)|
    =|X^{\mathrm{sp}}(\Fp)|
    =|\Sigma_3^{\Gamma}|
    =\Tr R(\varpi), \]
    where $\varpi\in \whD^\times$ is an element satisfying the properties in \eqref{eq:z13}.
\end{prop}
\begin{proof}
 The equality $X(\Fp)=X^{\mathrm{sp}}(\Fp)$ has already been proved in
Proposition~\ref{prop:K_odd}.
 We show that the equality $|X(\F_p)|=\Tr(R(\varpi))$ holds true in all cases, even  when $X(\F_p)=\emptyset$.
 From \eqref{eq:z17},  if $X(\F_p)\neq\emptyset$, then $\Tr(R(\varpi))=\Tr(R(\pi_A))=|X^{\mathrm{sp}}(\F_p)|>0$.
 Conversely, suppose that $\Tr(R(\varpi))>0$.
 We show that $X(\F_p)\neq\emptyset$, or equivalently, that there exists an embedding $\Q(\pi)\hookrightarrow B$ for some supersingular Weil $p$-number $\pi\in W_p^{\mathrm{ss}}(1)$.
 From \eqref{eq:z15}, the assumption $\Tr(R(\varpi))>0$ implies that there exist a polynomial $P(T)\in \scrP$ and a quadratic order $S$ containing $S_P$ in $K_P=\Q[T]/(P(T))$ such that $\prod_\ell m_\ell(S)>0$, where the product runs over all primes  $\ell$.
Hence there exists an embedding of $K_P\otimes_\Q \Q_\ell=S_\ell\otimes_{\Z_\ell} \Q_\ell$ into $\calD_\ell=O_\ell\otimes_{\Z_\ell} \Q_\ell$ for every prime $\ell$.
Since $B_\ell\simeq \calD_\ell$ for every prime $\ell\neq p$, we get an embedding $K_P\otimes_\Q \Q_\ell\hookrightarrow B_\ell$ for every prime $\ell\neq p$ as well.
On the other hand, $P(T)$ is an Eisenstein polynomial at the prime $p$, so $K_P\otimes_\Q \Q_p$ is a ramified quadratic extension of $\Q_p$, and hence there exists an embedding $K_P\otimes_\Q \Q_p\hookrightarrow B_p$.
From the local-global principle \cite[Theorem~III.3.8]{vigneras}, there exists an embedding of $K_P$ into the indefinite quaternion $\Q$-algebra $B$.
Recall that $P(T)$ is the minimal polynomial of a supersingular Weil $p$-number $\pi\in W_p^{\mathrm{ss}}(1)$, so the above discussion shows the  existence of an embedding of $\Q(\pi)\simeq K_P$ into $B$ as desired.
In summary, we have shown that $X(\Fp)=\emptyset$ if and only if $\Tr(R(\varpi))=0$. Therefore, the equality  $|X(\F_p)|=\Tr(R(\varpi))$ holds unconditionally.
\end{proof}

\begin{proof}[Proof of part (3) of Theorem \ref{thm:main}]

By Proposition~\ref{prop:trRw}, it remains to write the right-hand side of \eqref{eq:z15} in more concrete terms.
For each prime  $\ell\nmid p\Delta_{\calD}$, the order $O_\ell$ is maximal
in the split quaternion algebra $\calD_\ell\simeq \Mat_2(\Q_\ell)$, and hence $m_\ell(S)=1$ for every quadratic $\Z$-order $S$ by \eqref{eq:emb}.
At the prime $p$,  if $p\nmid \Delta$, then $O_p$ is the unique maximal order in the quaternion division algebra $\calD_p$. If $p\mid \Delta$, then  $O_p$ is an Eichler order of level $p$ in the split quaternion algebra $\calD_p\simeq \Mat_2(\Q_p)$.
It follows from \eqref{eq:emb} that
\begin{equation}
    m_p(S)=\begin{dcases*}
        1-\left(\frac{S}{p}\right) & if $p\nmid \Delta$;\\
        1+\left(\frac{S}{p}\right) & if $p\mid \Delta$.
    \end{dcases*}
\end{equation}
On the other hand, since each $P(T)\in \scrP$ is an Eisenstein polynomial at $p$, the $\Z_p$-order $S_P\otimes \Z_p=\Z_p[T]/(P(T))$ is the maximal order in $K_P\otimes \Q_p=\Q_p[T]/(P(T))$.
Thus every  $\Z$-order $S$ in $K_P$ containing $S_P$ is maximal at $p$,  and hence $\left(\frac{S}{p}\right)=0$ since $K_P/\Q_p$ is a ramified quadratic extension.
In other words, we have
\begin{equation}
 m_p(S)=1\qquad \forall S_P\subseteq S\subseteq K_P, \quad \forall P(T)\in \scrP.
\end{equation}
Lastly, consider the remaining primes $\ell$, namely the prime divisors of $\Delta_0\coloneqq -\Delta_{\calD}/\gcd(\Delta_{\calD}, p)$.  At each prime $\ell\mid \Delta_0$, the order $O_\ell$  is the unique maximal order in the quaternion division algebra $\calD_\ell$, so we have
\begin{equation}
    m_\ell(S)=1-\left(\frac{S}{\ell}\right), \qquad \forall \ell\mid \Delta_0.
\end{equation}
The above discussion shows that \eqref{eq:z15} can be further simplified into
\begin{equation}\label{eq:z16}
      \Tr(R(\varpi))=\sum_{P\in \scrP}\,\sum_{S_P\subseteq S\subseteq K_P}\frac{h(S)}{|S^\times|}\prod_{\ell\mid \Delta_0}\left(1-\left(\frac{S}{\ell}\right)\right).
\end{equation}

Now we classify the quadratic $\Z$-orders $S$ containing $S_P$ for some $P\in \scrP$. It is conceptually more convenient to identify $K_P=\Q[T]/(P(T))$  with $\Q(\pi)$ by mapping $T+(P(T))$ to  the root $\pi \in W_p^{\mathrm{ss}}(1)$ of $P(T)$. In turn, the quadratic order $S_P$ is identified with $\Z[\pi]$.
Suppose that $p>3$.
Then $\scrP$ consists of the single polynomial $x^2+p$, so $K_P=\Q(\sqrt{-p})$ and $S_P=\Z[\sqrt{-p}]$, which has discriminant $-4p$.
The only roots of unity in $K_P$ are $\pm 1$, so  $S^\times=\{\pm 1\}$ for every $S$ in $K_P$.
There are two subcases to consider depending on whether $S_P=\Z[\sqrt{-p}]$ is  maximal  in $\Q(\sqrt{-p})$.
If $p\equiv 1\pmod{4}$, then $S_P=\Z[\sqrt{-p}]$ is already maximal.
  If $p\equiv 3\pmod{4}$, then $\Z[\sqrt{-p}]$ has index $2$ in the maximal order $\Z[(-1+\sqrt{-p})/2]$, which has discriminant $-p$.
Combining \eqref{eq:z16} with \eqref{eq:mod-hurwitz-cln}, we get
\begin{equation}
    \Tr(R(\varpi))=\frac{1}{2}(h_{\Delta_0}(-p)+h_{\Delta_0}(-4p)),  \qquad \forall\, p>3.
\end{equation}
Here $h_{\Delta_0}(-p)\coloneqq 0$ when $p\equiv 1\pmod{4}$ by our convention since $-p$ is not a fundamental discriminant in this case.

Next, suppose that $p\in \{2, 3\}$.  In these two cases $\scrP=\{x^2+p, x^2\mp px+p\}$, which corresponds bijectively to the set $W_p^{\mathrm{ss}}(1)=\{\sqrt{-p}, \pm \sqrt{p}\zeta_{4p}\}$. A straightforward computation shows that
$\Z[\pm \sqrt{p}\zeta_{4p}]=\Z[\zeta_{2p}]$ for both $p=2, 3$. In particular,
$h(\Z[\pm\sqrt{p}\zeta_{4p}])=1$ and $\Z[\pm\sqrt{p}\zeta_{4p}]^\times=\langle \zeta_{2p}\rangle$.
On the other hand, if $p=2$, then $\Z[\sqrt{-2}]$ is already the maximal order in $\Q(\sqrt{-2})$.  We have $h(\Z[\sqrt{-2}])=1$ and $\Z[\sqrt{-2}]^\times=\{\pm 1\}$. If $p=3$, then there are two $\Z$-orders containing $\Z[\sqrt{-3}]$ in $\Q(\sqrt{-3})$, namely, $\Z[\sqrt{-3}]$ itself and the maximal order $\Z[\zeta_6]$.  We have $\Z[\sqrt{-3}]^\times=\{\pm 1\}$ and $h(\Z[\sqrt{-3}])=1$ by \eqref{eq:dedekind}.

Substituting the above data into \eqref{eq:z16}, we obtain
\[ \Tr(R(\varpi))=
     \begin{dcases*}
         \frac{1}{2}\prod_{\ell\mid \Delta_0}\left(1-\left(\frac{-8}{\ell}\right)\right)+2\cdot\frac{1}{4}\prod_{\ell\mid \Delta_0}\left(1-\left(\frac{-4}{\ell}\right)\right), & if $p=2$;\\
         \frac{1}{2}\prod_{\ell\mid \Delta_0}\left(1-\left(\frac{\Z[\sqrt{-3}]}{\ell}\right)\right)+3\cdot \frac{1}{6}\prod_{\ell\mid \Delta_0}\left(1-\left(\frac{-3}{\ell}\right)\right) & if $p=3$.
     \end{dcases*}\]
The desired formulas are obtained by simplifying the above expressions  by noticing that $\gcd(2, \Delta_0)=1$ if $p=2$, and $\left(\frac{\Z[\sqrt{-3}]}{2}\right)=1$ if $p=3$ and $2\mid \Delta_0$.
\end{proof}
\begin{remark}\label{rem:orbits}
   Let $\Sigma \coloneqq X^{\sp}(\Fpbar)$ or $\Sigma_3$ according to whether $p \nmid \Delta$ or $p \mid \Delta$.
   The $\Gamma$-orbits $\Sigma/\Gamma$ correspond to the closed points of $\Sigma$, viewed as a subscheme of $X_{\F_p}$.
   Combining the orbit-counting relation $2\lvert \Sigma /\Gamma \rvert=\lvert \Sigma \rvert + \lvert \Sigma^{\Gamma} \rvert$ with the formulas \eqref{eq:z17},   $\lvert \Sigma \rvert=h(O)$, and $\lvert \Sigma^{\Gamma} \rvert =\lvert X^{\rm sp}(\Fp) \rvert$, we obtain
   \begin{equation}\label{eq:2t-h}
\lvert X^{\rm sp}(\F_p) \rvert = 2 \lvert \Sigma/\Gamma \rvert - h(O)= \Tr(R(\pi_A)).
\end{equation}
Furthermore,
  Proposition~\ref{prop:ssp-Fp-pts} implies  $\Sigma/\Gamma \simeq \calD^{\times} \backslash \whD^{\times}/ \whO^{\times} \<\pi_A \>$.

  This relation recovers the Deuring formula as follows.
  Suppose now that $B = \Mat_2(\Q)$, so that $X$ is identified with the standard modular curve via Morita equivalence.
  In this case, a computation shows    $\calD^{\times} \backslash \whD^{\times}/ \whO^{\times} \<\pi_A \> \simeq \calD^{\times} \backslash \whD^{\times}/ N_{\whD^{\times}}(\whO^{\times})$, where $N_{\whD^{\times}}(\whO^{\times})$ is the normalizer of $\whO^{\times}$ in $\whD^{\times}$.
    According to \cite[(27.6.24)]{voight-quat-book}, its cardinality equals the type number of $\calD$.
\end{remark}
\thank
The authors thank Tomoyoshi Ibukiyama for bringing his previous work to our attention and offering insightful suggestions.
Terakado is supported by JSPS KAKENHI Grant Number JP26K06747.
Xue is partially supported by the National Natural Science Foundation of China grants No.~12331002 and No.~12271410. Yu is partially supported by the National Science and Technology Council (NSTC) grant 115-2115-M-001-008-MY3 and the Academia Sinica IVA grant AS-IA-112-M01.

\def\cprime{$'$}

\end{document}